\documentclass[11pt,a4paper]{article}

\usepackage{amssymb}
\usepackage[usenames]{color}
\usepackage{theorem}
\usepackage[utf8]{inputenc}
\usepackage{enumerate}
\usepackage{amsmath}
\usepackage{comment}

\usepackage{mathdots}

\numberwithin{equation}{section}

\newtheorem{theorem}{Theorem}[section]
\newtheorem{proposition}[theorem]{Proposition}
\newtheorem{lemma}[theorem]{Lemma}
\newtheorem{corollary}[theorem]{Corollary}
\newtheorem{definition}[theorem]{Definition}
\newtheorem{remark}[theorem]{Remark}

\newenvironment{proof}{\noindent\textbf{Proof}}
                      {$\Box$\vskip\theorempostskipamount}

\begin{document}

\title{\textbf{Twisted associativity of the cyclically reduced product of words, part 2}}

\author{Carmelo Vaccaro}

\date{}

\maketitle

\begin{abstract} The cyclically reduced product of two words $u, v$, denoted $u * v$, is the cyclically reduced form of the concatenation of $u$ by $v$. This product is not associative. Recently S. V. Ivanov has proved that the Andrews-Curtis conjecture can be restated in terms of the cyclically reduced product and cyclic permutations instead of the reduced product and conjugations.
	
In a previous paper we have proved that $*$ verifies generalizations of properties of the product in the free group. In another previous paper we have proved that $*$ verifies a generalized version of the associativity property in a special case. In the present paper we prove that a more general version of the associativity property holds for $*$ in the general case.
 \end{abstract}

\smallskip \smallskip


\smallskip \smallskip


\section*{Introduction}

Let $X$ be a set of letters, let $X^{-1}$ be the set of inverses of elements of $X$ and let $\mathcal{M}(X \cup X^{-1})$ be the free monoid on $X \cup X^{-1}$. The elements of $\mathcal{M}(X \cup X^{-1})$ are the non-necessarily reduced words on $X$. We denote $\mathcal{F}(X)$ the free group on $X$ and we consider it as the subset of $\mathcal{M}(X \cup X^{-1})$ consisting of reduced words. We denote $\hat{\mathcal{F}}(X)$ the set of cyclically reduced words on $X$.

Given $v, w \in \mathcal{M}(X \cup X^{-1})$ the cyclically reduced product of $v$ by $w$, denoted $v*w$, is defined as the cyclically reduced form of the concatenation $vw$. By contrast, the reduced product of $v$ by $w$, which we denote $v \cdot w$, is defined as the reduced form of $vw$.

The cyclically reduced product has applications to the Andrews-Curtis conjecture: in \cite{Ivanov06} and \cite{Ivanov18} S. V. Ivanov has proved that the conjecture (with and without stabilizations) is true if and only if in the definition of the conjecture we replace the operations of reduced product and conjugations with the cyclically reduced product and cyclic permutations. The importance of this result stems from the fact that while there are infinitely many conjugates of one word, there are only finitely many cyclic permutations, thus making much easier the search of Andrews-Curtis trivializations by enumerations of relators, like for example the approaches used in \cite{BMc} or \cite{PU}. 

\smallskip

The set $\hat{\mathcal{F}}(X)$ is closed with respect to $*$, like the free group is closed with respect to $\cdot$, but the structure of $(\hat{\mathcal{F}}(X), *)$ is much less nice than that of $(\mathcal{F}(X), \cdot)$, mainly because $*$ is not associative.

In \cite{FirstArticle} we have started the study of properties of $*$ and of $\hat{\mathcal{F}}(X)$ equipped with $*$ ; in particular we have proved that $*$ verifies generalizations of properties of the reduced product. In \cite{SecondArticle} we have proved that $*$ verifies a generalized version of the associativity property in a special case. In the present paper we prove that a more general version of the associativity property holds for $*$ in the general case.

\smallskip

Let us show this result with an example. Let $X = \{x, y\}$ and let $u := x y x y$, $v := y^{-1} x^{-1} y^{-2}$, $w := x y x^{-1} y$. Then 
		\begin{equation} \label{exp1} (u*v)*w = x y^{-1} x y x^{-1} y,\end{equation}
while
		$$u*(v*w) = x y x y x^{-1} y^{-2} x y x^{-1}.$$
The two expressions are different and not even of the same length (in particular they are not cyclic permutation one of the other).

In \cite{SecondArticle} we have proved that if we take $v = u^{-1}$ then the cyclically reduced product verifies a ``twisted" version of the associative property, i.e., up to cyclic permutations of some of the terms. However for a general $v$ this property does not hold true, but a more general one holds. 

Indeed let $u, v$ and $w$ be as defined above; we have that 
		$$v * w = x^{-1} y^{-2} x y x^{-1}.$$
The word $f := y^{-2} x y x^{-2}$ is a cyclic permutation of $v * w$ and we have that
		$$u * f = y x y^{-1} x y x^{-1},$$
which is a cyclic permutation of $u*(v*w)$. Therefore there exists a cyclic permutation $f$ of $v * w$ such that $u * f$ is a cyclic permutation of $(u*v)*w$.

But we can say more. Indeed we have that $u *w = x y x y x y x^{-1} y$. The word $g := x y x^{-1} y x y x y$ is a cyclic permutation of $u * w$ and we have that
		$$v * g = y^{-1} x y x^{-1} y x,$$
which is a cyclic permutation of $u*(v*w)$. Therefore there exists a cyclic permutation $g$ of $u * w$ such that $v * g$ is a cyclic permutation of $(u*v)*w$.

We can conjecture that this situation is general: given words $u, v, w$ there exist cyclic permutations $u', v', w'$ of $u, v, w$ respectively and cyclic permutations $f$ and $g$ of $v' * w'$ and of $u' * w'$ respectively such that $(u*v)*w$ is a cyclic permutation of $u' * f$ or of $v' * g$ (or of both).

We will prove in Theorem \ref{mainTheor2} that this is indeed the case, but as in \cite{SecondArticle}, that also another situation can happen: that there exists a non-empty word $h$ such that the concatenation of words $u' h f h^{-1}$ (or the concatenation $v' h g h^{-1}$) is cyclically reduced and is a cyclic permutation of $(u*v)*w$.

\smallskip

The result that we prove in Theorem \ref{mainTheor2} is even stronger: first of all, we do not need to assume that $u, v$ and $w$ be cyclically reduced, not even that they are reduced; secondly, the result is true also if in (\ref{exp1}) instead of $u*v$ we take a cyclic permutation of $u*v$.

Let us state the results of Theorem \ref{mainTheor2}\footnote{Part a) is stated in a less precise way than in Theorem \ref{mainTheor2}.}; we use the symbol $\sim$ to denote that two words are cyclic permutation one of the other. If $u, v, w$ are words and $d$ a cyclic permutation of $u*v$, there exist words $p, q, w', f, h$ such that	
	\begin{enumerate}[a)]
		
		\item either $p \sim u$ and $q \sim v$ or $p \sim v$ and $q \sim u$,
		
		\item $w' \sim w$,
		
		\item $f \sim q * w'$,	
	\end{enumerate}
and such that
	\begin{equation} \label{d*wIntro} d*w \sim p*(h f h^{-1}),\end{equation} 
with $p*(h f h^{-1}) = p h f h^{-1}$ if $h \neq 1$.

\smallskip

As we will see in (\ref{permCon2}) of Proposition \ref{summar}, the fact that $d*w$ is a cyclic permutation of $p*(h f h^{-1})$ implies that a conjugate of $d*w$ is equal to $p*(h f h^{-1})$. Since $d*w$ and $p*(h f h^{-1})$ are products of conjugates of $u, v, w$, the equivalence (\ref{d*wIntro}) gives an equality between two products of conjugates of $u, v, w$. Theorem \ref{mainTheor2} also proves, as Remark \ref{samProdConjRel} makes it clear, that these two expressions are indeed the same, that is with the same conjugating elements.

Now the question arises whether, in the case where $u, v$ and $w$ are relators of a group presentation, the van Kampen diagrams associated with $d*w$ and $p*(h f h^{-1})$ are homeomorphic. Remark \ref{nonHomeoVK} shows that this is not the case, however Theorem \ref{mainTheor2} proves that even if these two diagrams are not homeomorphic, they have some internal paths that are homeomorphic.

\smallskip

The proof of Theorem \ref{mainTheor2} is long 42 pages and occupies Sections \ref{mainSec}-\ref{prLe4}. The proof is not direct: we show that Theorem \ref{mainTheor2} follows from Lemma \ref{mainLemma} which requires analyzing a number of subcases much less than what a direct proof would have required.

\smallskip

This paper is a continuation of \cite{FirstArticle} and \cite{SecondArticle} but can be read independently since it is self contained. All the results from \cite{FirstArticle} and \cite{SecondArticle} needed in this paper are stated in Section \ref{S1}.

\section{Words, cyclic permutations and cyclically reduced product} \label{S1}

Let $Y$ be a set and let us consider $\mathcal{M}(Y)$, the free monoid on $Y$. The elements of $Y$ are called \textit{letters}, those of $\mathcal{M}(Y)$ \textit{the words in $Y$}. As usual given words $v, w \in \mathcal{M}(Y)$ we will denote $vw$ the product of $v$ by $w$, which is the concatenation of the words $v$ and $w$. The word with no letters, which is the identity element of $\mathcal{M}(Y)$, is denoted 1.

Let $v, w \in \mathcal{M}(Y)$; we say that $w$ is a \textit{subword of $v$} if there exist $p, q \in \mathcal{M}(Y)$ such that $v = p w q$. In this case we say that $w$ is a \textit{prefix of $v$} if $p = 1$ and that $w$ is a \textit{suffix of $v$} if $q = 1$.

Let $v = y_1 \cdots y_n \in \mathcal{M}(Y)$ with $y_1, \cdots, y_n \in Y$. The \textit{length of $v$} is defined as $|v| := n$. The \textit{reverse of $v$} is defined as the word $\underline{v} := y_n \cdots y_1$, where the order of the letters is the reverse as that of $v$. 

Let $w_1$ and $w_2$ be words and let $w:=w_1 w_2$. The word $w_2 w_1$ is called a \textit{cyclic permutation of $w$}. Given two words $u$ and $v$ the relationship ``$u$ is a cyclic permutation of $v$" is an equivalence that we denote $u \sim v$.

Let $X$ be a set; we denote $\mathcal{F}(X)$ the free group on $X$ and we consider $\mathcal{F}(X)$ as a subset\footnote{Usually $\mathcal{F}(X)$ is considered a quotient of $\mathcal{M}(X \cup X^{-1})$, but in this paper we will not follow this habit.} of $\mathcal{M}(X \cup X^{-1})$. In particular $\mathcal{F}(X)$ is the set of \textit{reduced words on $X$}, i.e., the words of the form $x_1 \cdots x_n$ with $x_i \in X \cup X^{-1}$ and $x_{i+1} \neq x_i^{-1}$ for $i = 1, \cdots, n-1$.

We denote $\rho: \mathcal{M}(X \cup X^{-1}) \rightarrow \mathcal{F}(X)$ the function sending a word to its unique \textit{reduced form}. Given $u, v \in \mathcal{F}(X)$ the product\footnote{The product of two reduced words in $\mathcal{F}(X)$ does not coincide with the product of the same words in $\mathcal{M}(X \cup X^{-1})$. In particular $\mathcal{F}(X)$ is not a subgroup of $\mathcal{M}(X \cup X^{-1})$.} of $u$ by $v$ in $\mathcal{F}(X)$ is $\rho(uv)$, the reduced form of $uv$.

\medskip

\noindent \textbf{Convention} In this paper we adopt the following conventions:

1. With the term \textit{word} we mean a non-necessarily reduced word, i.e., an element of $\mathcal{M}(X \cup X^{-1})$.

2. Given $u$, $v_1, \dotsc, v_n \in \mathcal{M}(X \cup X^{-1})$, with the notation $u = v_1 \dots v_n$ we mean the equality in $\mathcal{M}(X \cup X^{-1})$ of $u$ with the concatenation of words $v_1 \dots v_n$ even if $u$ and all the $v_j$ belong to $\mathcal{F}(X)$. This kind of equality is called a \textit{factorization} of $v$ in the Combinatorics of Words literature (see \cite{Kar}, pag. 2 or \cite{ChofKar}, pag. 332). The equality in $\mathcal{F}(X)$ of $u$ with the reduced product of $v_1, \ldots, v_n$ will be denoted by $u = v_1 \cdot \, \ldots \, \cdot v_n$ and corresponds to the equality $\rho(u) = \rho(v_1 \ldots v_n)$ in $\mathcal{M}(X \cup X^{-1})$.

\medskip

The operations of inversion and reversion of words commute one with the other, that is given a word $v$ we have that $(\underline{v})^{-1} = \underline{(v^{-1})}$. Thus we will denote $\underline{v}^{-1}$ the inverse of the reversion of $v$ without fear of ambiguity.

We say that a reduced word is \textit{cyclically reduced} if its last letter is not the inverse of the first one, that is if all its cyclic permutations are reduced. We denote $\hat{\mathcal{F}}(X)$ the set of cyclically reduced words.

Given a word $w$ there exist unique $t \in \mathcal{F}(X)$ and $c \in \hat{\mathcal{F}}(X)$ such that $\rho(w) = t c t^{-1}$. The word $c$ is called the \textit{cyclically reduced form of $w$} and is denoted $\hat{\rho}(w)$. In particular we consider the function $\hat{\rho}: \mathcal{M}(X \cup X^{-1}) \rightarrow \hat{\mathcal{F}}(X)$ sending a word to its unique cyclically reduced form. Therefore we have that $\rho(w) = t \hat{\rho}(w) t^{-1}$ and that $\rho(w)$ is cyclically reduced if and only if $t=1$. 

Given words $u$ and $v$ we denote $u*v$ the \textit{cyclically reduced product of $u$ by $v$}, i.e., $u * v := \hat{\rho}(uv)$. This product is non-associative. Indeed let $u = xy$, $v = x^{-1}$ and $w = x$; then $(u*v)*w = yx$ while $u*(v*w) = xy$.

\begin{proposition} \label{summar} \rm Let $u, v, u_1, u_2, \cdots, u_n$ be words; then the following results hold: 	
	\begin{enumerate} [(1)]

		\item \label{reverse2} The reverse of $u_1 u_2 \cdots u_n$ is the word $\underline{u_n} \cdots \underline{u_2} \, \underline{u_1}$. 
		
		\item \label{reversesim} If $u \sim v$ then $\underline{u} \sim \underline{v}$. 
			
		\item \label{cpCanc} Let $u$ be a cyclic permutation of $\rho(v)$. Then there exists a cyclic permutation $v'$ of $v$ such that $\rho(v') = \rho(u)$. 

		\item \label{u*vNonRed} $u*v = \rho(u) * \rho(v)$. More generally if $u_1, v_1$ are words such that $\rho(u_1) = \rho(u)$ and $\rho(v_1) = \rho(v)$, then $u*v = u_1 * v_1$. 
		
		\item \label{jarem} $u*v = 1$ if and only if $\rho(u) = \rho(v^{-1})$. 

		\item \label{scope} The cyclically reduced form of $u$ is equal to the reduced form of some conjugate of $u$, that is there exists a word $\alpha$ such that $\hat{\rho}(u) = \rho(\alpha u \alpha^{-1})$. 		

		\item \label{permCon2} Let $u$ be a cyclic permutation of $v$; then the reduced form of $u$ is the reduced form of some conjugate of $v$. 
	
		\item \label{revCRP} The reverse of $u*v$ is equal to $\underline{v}*\underline{u}$ and the cancellations made to obtain $\underline{v}*\underline{u}$ from $\underline{v} \underline{u}$ are the reverse of those made to obtain $u*v$ from $uv$. 
		
		\item \label{rfrev} $\rho(\underline{u}) = \underline{\rho(u)}$ and the cancellations made to obtain $\rho(\underline{u})$ from $\underline{u}$ are the reverse of those made to obtain $\rho(u)$ from $u$. 

		\item \label{necToCpc} If $\rho(u) = v_1 v_2$ there exist words $u_1, u_2$ such that $u = u_1 u_2$ and $\rho(u_1) = v_1$, $\rho(u_2) = v_2$.

		\item \label{obviRho} $\rho(uv) = \rho(\rho(u) \rho(v))$.

	\end{enumerate} 
\end{proposition}

\begin{proof} 
	\begin{enumerate} [(1)]
		\item See Remark 1.1 of \cite{FirstArticle}.
		
		\item See Remark 1.5 of \cite{FirstArticle}.
		
		\item See Remark 1.12 of \cite{FirstArticle}.

		\item See Remark 2.14 of \cite{FirstArticle}. 
		
		\item See Remark 2.13 of \cite{FirstArticle}.

		\item See Remark 2.4 of \cite{FirstArticle}. 
		
		\item See Remarks 1.14 of \cite{FirstArticle}.
	
		\item See Remark 2.15 of \cite{FirstArticle}.	
		
		\item See Proposition 1.13 of \cite{FirstArticle}.	

		\item See Remark 1.11 of \cite{FirstArticle}.

		\item See Remark 1.10 of \cite{FirstArticle}.
	\end{enumerate} 
\end{proof}

\begin{remark} \label{rfcp-rf} \rm  Let $v$ be the reduced form of a cyclic permutation of $\rho(u)$; then $v$ is the reduced form of a cyclic permutation of $u$. Indeed there exist words $v_1, v_2$ such that $\rho(u) = v_1 v_2$ and $v = \rho(v_2 v_1)$. By (\ref{necToCpc}) of Proposition \ref{summar} there exist words $u_1, u_2$ such that $u = u_1 u_2$ and $\rho(u_1) = v_1$, $\rho(u_2) = v_2$. Thus $v = \rho(v_2 v_1) = \rho(\rho(u_2) \rho(u_1)) = \rho(u_2 u_1)$, where the last equality follows from (\ref{obviRho}) of Proposition \ref{summar}, and this proves the claim.
 \end{remark}

\begin{remark} \label{u*v=uv} \rm  Let $u, v$ be words; then $u * v = uv$ if and only if $v * u = vu$. Indeed $u * v = uv$ if and only if $uv$ is cyclically reduced, but this is true if and only if $vu$ is cyclically reduced, if and only if $v * u = vu$. \end{remark}

\begin{remark} \label{LeviLemma} \rm We have the following result, known as \textit{Levi's Lemma} (see \cite{ChofKar}, pag. 333 or \cite{Kar}, Theor. 2): \textit{let $u_1, u_2, v_1, v_2$ be words such that $u_1 u_2 = v_1 v_2$; then there exists a word $p$ such that either $u_1 = v_1 p$ and $v_2 = p u_2$ or $v_1 = u_1 p$ and $u_2 = p v_2$.} The two cases can be represented graphically in the following way,
	
\medskip
	
\hspace{2cm} 
\begin{tabular}{|c|c|}
	\hline 
	$\,\, u_1 \,$   & $u_2 \,$ \\
	\hline
\end{tabular}
	\quad  \hspace{1cm} and \hspace{1cm}
\begin{tabular}{|c|c|}
	\hline 
	$u_1$ & $\,\,u_2\,$ \\
	\hline
\end{tabular}
	
\hspace{2cm}
\begin{tabular}{|c|c|}
	\hline 
	$v_1$ & $\,\, v_2 \,\,\,$ \\
	\hline
\end{tabular}
	\quad \hspace{28.8mm}
\begin{tabular}{|c|c|}
	\hline 
	$\,\,v_1 \,\,$   & $v_2$ \\
	\hline
\end{tabular}
	
\medskip
	
\noindent and correspond to putting the bar separating $v_1$ and $v_2$ either inside $u_1$ or inside $u_2$. The case when this bar is exactly below that separating $u_1$ and $u_2$, i.e., when $u_1 = v_1$ and $u_2 = v_2$, can be considered a special case of both the cases. 
	
In general let us consider the word equation $u_1 \dots u_m = v_1 \dots v_n$, possibly with $m \neq n$. Any solution to this equation determines uniquely a way of putting $n - 1$ bars inside the $m$ spaces corresponding to $u_1, \dots, u_m$ and also a way of putting $m - 1$ bars inside the $n$ spaces corresponding to $v_1, \dots, v_n$. This is true even if some of the $u_i$ or $v_j$ are the empty word. Indeed if $u_i = 1$ or $v_j = 1$ then no bar must be contained in $u_i$ or $v_j$. 
	
We observe that a solution to the equation $u_1 \dots u_m = v_1 \dots v_n$ determines also a \textit{weak composition}\footnote{a weak composition for an integer is a composition when 0's are allowed} for $n - 1$ in $m$ parts and for $m - 1$ in $n$ parts.
	
We give the following as an example for $m = 4$ and $n = 3$:
	
\medskip

\hspace{3.5 cm}
\begin{tabular}{|c|c|c|c|}
	\hline 
	$u_1$ & $u_2$ &	$u_3$ & $\,\, u_4 \,\,$ \\
	\hline
\end{tabular}
	
\hspace{3.5 cm}
\begin{tabular}{|c|c|c|}
	\hline 
	$\,\, v_1 \,\,$ & $\,\,\,\,\,\,\, v_2 \,\,\,\,\,\,\,$ & $v_3$\\
	\hline
\end{tabular}

\medskip
	
\noindent Here we can say that there exist words $a, b, c$ such that $v_1 = u_1 a$, $u_2 = a b$, $v_2 = b u_3 c$ and $u_4 = c v_3$. This solution determines the weak compositions $(0, 1, 0, 1)$ for 2 and $(1, 2, 0)$ for 3, which are obtained by counting the number of bars inside each $u_i$ and each $v_j$ respectively.
\end{remark}

\begin{lemma} \label{shirv} Let $u$ and $v$ be reduced words such that $u \neq v^{-1}$. Then one of the following holds:
	\begin{enumerate} [1)]
		\item there exist words $u_1, a, s$ such that $u = u_1 a$, $v = a^{-1} s (u*v) s^{-1} u_1^{-1}$ and $\rho(uv) = u_1 s (u*v) s^{-1} u_1^{-1}$;
		
		\item there exist non-empty words $c_1, c_2$ and words $t, a$ such that $u*v = c_1 c_2$, $u = t c_1 a$, $v = a^{-1} c_2 t^{-1}$, $\rho(uv) = t c_1 c_2 t^{-1}$, $\rho(v u) = a^{-1} c_2 c_1 a$ and $v*u = c_2 c_1$; 
		
		\item there exist words $v_1, s, a$ such that $u = v_1^{-1} s (u*v) s^{-1} a$, $v = a^{-1} v_1$ and $\rho(uv) = v_1^{-1} s (u*v) s^{-1} v_1$. 
	\end{enumerate}
\end{lemma} 

\begin{proof} See Lemma B.2 of \cite{FirstArticle}.
\end{proof}

\begin{lemma} \label{shirv4} Let $u$ and $v$ be reduced words such that $u \neq v^{-1}$ and let $d$ be a cyclic permutation of $u*v$. Then there exists a pair of words $p$ and $q$ such that one of them is a cyclic permutation of $u$ and the other of $v$ and such that one of the following two cases holds: \begin{enumerate} [(a)]
		\item $q = p^{-1} r c_1 c_2 r^{-1}$ for words $r, c_1, c_2$ such that $p*q = c_1 c_2$ and $d = c_2 c_1$; moreover $p*q = u*v$; finally if $d = u*v$ then $c_2 = 1$;
		
		\item $p = e_2 b$ and $q = b^{-1} e_3 e_1$ for words $b, e_1, e_2, e_3$ such that $d = e_1 e_2 e_3$ and $e_2, e_3 e_1 \neq 1$; moreover either $p*q = u*v$ or $q*p = u*v$; finally if $d = u*v$ then $e_1 = 1$.		
	\end{enumerate} 
\end{lemma}  

\begin{proof} See Lemma B.4 of \cite{FirstArticle}.
\end{proof}

\begin{proposition} \label{puzo} Let $u$ and $v$ be words; then $u*v$ is a cyclic permutation of $v*u$. 
	
Moreover if $u$ and $v$ are reduced and if there exist words $\alpha$, $\beta$, $u'$, $v'$ such that $u = \alpha u' \beta$ and $v = \beta^{-1} v' \alpha^{-1}$ then the words $\beta \beta^{-1}$ and $\alpha^{-1} \alpha$ are canceled when obtaining $u*v$ from $uv$ and when obtaining $v*u$ from $vu$.
\end{proposition} 

\begin{proof} See Theorem 4.1 of \cite{FirstArticle}.
\end{proof}

\begin{corollary} \label{permCycRedForm} If $t, w$ are words then $\hat{\rho}(t w t^{-1})$ is a cyclic permutation of $\hat{\rho}(w)$. If moreover $\rho(t)\rho(w)\rho(t)^{-1}$ is reduced then $\hat{\rho}(t w t^{-1}) =\hat{\rho}(w)$. \end{corollary} 

\begin{proof} See Corollary 4.4 of \cite{FirstArticle}. \end{proof}

\begin{proposition} \label{corToMTh1} Let $u$ and $w$ be words and let us set $f := u^{-1} * w$ and $g := w * u^{-1}$. Then there exist words $u', u'', h$ such that $u'$ and $u''$ are the reduced forms of cyclic permutations of $u$ and such that
	$$\hat{\rho}(w) \sim u' * (h f h^{-1}), \hspace{0.5 cm} \hat{\rho}(w) \sim (h g h^{-1}) * u''.$$
Moreover if $h \neq 1$ then $f = g$, $u' = u''$, $u'*(h f h^{-1}) = u' h f h^{-1}$ and $(h g h^{-1}) * u'' = h g h^{-1} u''$. Finally the identities among relations involving $u, w, u^{-1}, w^{-1}$ that by Remark \ref{idFrom} follow from these equivalences are strictly basic.
\end{proposition}

\begin{proof} See Corollary 2.3 of \cite{SecondArticle}. \end{proof}

\section{The main theorem} \label{mainSec}

This section presents the main theorem of this paper which states that a generalized version of the associative property holds for the cyclically reduced product.

The proof we give is not a direct proof, i.e., we do not prove directly that for any words $u, v$ and $w$ and for any cyclic permutation $d$ of $u*w$ the equivalence (\ref{d*w}) holds. Indeed in this section we prove that (\ref{d*w}), and more generally Theorem \ref{mainTheor2}, follow from an easier to prove result, namely Lemma \ref{mainLemma}. We devote Sections \ref{lemmaSect}-\ref{prLe4} to prove the latter.

The proof of Lemma \ref{mainLemma} is long 39 pages and is divided into 43 subcases. This number of cases is big, but it is still manageable. A direct proof of Theorem \ref{mainTheor2} without using Lemma \ref{mainLemma} would have been conceptually simpler, but would have made increase the number of subcases in a way that would have made the proof practically impossible to carry out. Indeed each of the 43 subcases of Lemma \ref{mainLemma} should have been subdivided into many more sub-subcases.

The advantage of using Lemma \ref{mainLemma} is that we do not need to compute the right hand side of (\ref{d*w}), in particular $q*w'$, but only $\rho(q w')$, the reduced form of $q w'$.

Let us take for example case 1A1 of Section \ref{prLe3}; here we have that $q = v$ and $w' = w$. Since  $\rho(vw) = t (v*w) t^{-1}$ for some word $t$ and since $\rho(vw) = x^{-1} s (d*w) s^{-1} d_1^{-1}$ in this case, then we have the word equation 
	\begin{equation} \label{wordEq} x^{-1} s (d*w) s^{-1} d_1^{-1} = t (v*w) t^{-1}.	\end{equation}
By Remark \ref{LeviLemma}, (\ref{wordEq}) has a number of solutions equal to the number of weak compositions for $m -1$ into $n$ parts, where $m$ and $n$ are the number of words on the left and right hand side respectively of (\ref{wordEq}). Since $m = 5$ and $n = 3$, there are 15 solutions. Even if some solutions can be trivial, there will be certainly some that are not and this would split case 1A1 into many sub-subcases.

\smallskip

In order to prove that Lemma \ref{mainLemma} implies Theorem \ref{mainTheor2} we make use of the main theorem of \cite{SecondArticle}. The latter states that a generalized version of the associative property holds for the triple of words $u, u^{-1}, w$. In the proof of Theorem \ref{mainTheor2} we use that result for the triple of words $p, p^{-1}, \gamma (d*w) \gamma^{-1}$ where $p$ is a cyclic permutation of $u$ or of $v$ and $\gamma$ is some word.

\bigskip

\noindent \textbf{Convention} In this section and in Sections \ref{lemmaSect}-\ref{prLe4} with the expression ``identity among relations involving $r_1, \cdots, r_m$" we mean an identity among relations involving $r_1, \cdots, r_m, r_1^{-1}, \cdots, r_m^{-1}$.

This is unlike the terminology used in Appendix \ref{sectIAR} and in Proposition \ref{corToMTh1}.

\medskip

\begin{lemma} \label{anothLem} Let $p, w, \gamma$ be words. Let us set $f := p^{-1} * (\gamma (d*w) \gamma^{-1})$ and $g := (\gamma (w*d) \gamma^{-1}) * p^{-1}$. Then there exists words $p', p'', h$ such that $p'$ and $p''$ are the reduced forms of cyclic permutations of $p$ and the following two equalities hold	
		\begin{equation} \label{impLemEq1} \hat{\rho}(\gamma (d*w) \gamma^{-1}) \sim p' * (h f h^{-1})\end{equation} 
and
		\begin{equation} \label{impLemEq1'} \hat{\rho}(\gamma (w*d) \gamma^{-1}) \sim (h g h^{-1}) * p''.\end{equation} 
Moreover if $h \neq 1$ then $f = g$, $p' = p''$, $p'*(h f h^{-1}) = p' h f h^{-1}$ and $(h g h^{-1}) * p'' = h g h^{-1} p''$. Finally the identities among relations involving $p$ and $\gamma (d*w) \gamma^{-1}$ that by Remark \ref{idFrom} follow from (\ref{impLemEq1}) and that involving $p$ and $\gamma (w*d) \gamma^{-1}$ following from (\ref{impLemEq1'}) are strictly basic.		
\end{lemma}

\begin{proof} (\ref{impLemEq1}) follows by applying Proposition \ref{corToMTh1} to $p$ and $\gamma (d*w) \gamma^{-1}$; (\ref{impLemEq1'}) follows by applying Proposition \ref{corToMTh1} to $p$ and $\gamma (w*d) \gamma^{-1}$
\end{proof}

\begin{lemma} \label{lemToMT} Let $p, q, w$ be reduced non-empty words. Let us suppose that there exist words $p_0, q_0, w', d,\alpha, \beta, \gamma$ such that
	\begin{center} $p_0 \sim p$, $q_0 \sim q$, $w' \sim w$, $d \sim p_0*q_0$ \end{center}
and 
	\begin{equation} \label{impLemEq2} \rho(\alpha q w' \alpha^{-1}) = \rho(\beta p^{-1} \beta^{-1} \gamma (d*w) \gamma^{-1})	\end{equation}
and the identity among relations involving $p, q, w$ that by Remark \ref{idFrom} follows from (\ref{impLemEq2}) is basic.

Let us set $f := p^{-1} * (\gamma (d*w) \gamma^{-1})$; then $f \sim q * w'$. Moreover there exist words $p', h$ such that $p'$ is the reduced form of a cyclic permutation of $p$ and 
	\begin{equation} \label{impLemEq3} d*w \sim p' * (h f h^{-1}),\end{equation} 
with $p'*(h f h^{-1}) = p' h f h^{-1}$ if $h \neq 1$. Finally the identity among relations involving $p, q, w$ that by Remark \ref{idFrom} follows from (\ref{impLemEq3}) is basic.		
\end{lemma}

\begin{proof} Since $f = p^{-1} * (\gamma (d*w) \gamma^{-1})$, by Lemma \ref{anothLem} there exist words $p', h$ such that $p'$ is the reduced form of a cyclic permutation of $p$ and the equality
	\begin{equation} \label{lemToMTeq1} \hat{\rho}(\gamma (d*w) \gamma^{-1}) \sim p' * (h f h^{-1})\end{equation}
holds, with $p'*(h f h^{-1}) = p' h f h^{-1}$ if $h \neq 1$. Moreover the identity among relations involving $p$ and $\gamma (d*w) \gamma^{-1}$ that by Remark \ref{idFrom} follows from (\ref{lemToMTeq1}) is strictly basic.

Since $\hat{\rho}(d*w) = d*w$, then by Corollary \ref{permCycRedForm} we have that $d*w \sim \hat{\rho}(\gamma (d*w) \gamma^{-1})$, so $d*w \sim p' * (h f h^{-1})$ and (\ref{impLemEq3}) holds. By Remark \ref{basicToo}, the identity among relations involving $p, q, w$ that follows from (\ref{impLemEq3}) is basic if the identity involving $p, q, w$ following from (\ref{lemToMTeq1}) is basic\footnote{In Lemma \ref{anothLem} we have proved that the identity following from (\ref{lemToMTeq1}) is basic when it involves $p$ and $\gamma (d*w) \gamma^{-1}$, but not when it involves $p, q, w$.}. In order to finish the proof of the lemma we have to prove this fact and that $f \sim q * w'$.

We divide the proof in two parts: in the first part we use the fact that the identity among relations following from (\ref{lemToMTeq1}) and involving $p$ and $\gamma (d*w) \gamma^{-1}$ is strictly basic; in the second part we use the fact that the identity among relations following from (\ref{impLemEq2}) and involving $p, q, w$ is basic.

\medskip

\textbf{First part.} Let us set $r := d*w$. By (\ref{scope}) of Proposition \ref{summar} there exists a reduced word $\theta_2$ such that $\hat{\rho}(\gamma r \gamma^{-1}) = \rho(\theta_2 \gamma r \gamma^{-1} \theta_2^{-1})$. 

Since $p'$ is the reduced form of a cyclic permutation of $p$, by Remark \ref{rfcp-rf} there exists a reduced word $\delta$ such that $p' = \rho(\delta p \delta^{-1})$. By (\ref{scope}) of Proposition \ref{summar} there exists a reduced word $\eta$ such that 
	$$p' * (h f h^{-1}) = \rho(\eta \delta p \delta^{-1} h f h^{-1} \eta^{-1}).$$
By (\ref{scope}) and (\ref{permCon2}) of Proposition \ref{summar} and in view of last equality there exists a reduced word $\theta_1$ such that (\ref{lemToMTeq1}) is equivalent to
	$$\rho(\theta_1 \theta_2 \gamma r \gamma^{-1} \theta_2^{-1} \theta_1^{-1}) = \rho(\eta \delta p \delta^{-1} h f h^{-1} \eta^{-1}),$$
which by setting $\theta := \theta_1 \theta_2$ becomes
	\begin{equation} \label{lemToMTeq2} \rho(\theta \gamma r \gamma^{-1} \theta^{-1}) = \rho(\eta \delta p \delta^{-1} h f h^{-1} \eta^{-1}).\end{equation}
Since $f = p^{-1} * (\gamma r \gamma^{-1})$, by (\ref{permCon2}) of Proposition \ref{summar} there exists a reduced word $\zeta$ such that 
	\begin{equation} \label{lemToMTeqf} f = \rho(\zeta p^{-1} \gamma r \gamma^{-1} \zeta^{-1}),\end{equation}
so (\ref{lemToMTeq2}) becomes
	$$\rho(\theta \gamma r \gamma^{-1} \theta^{-1}) = \rho(\eta \delta p \delta^{-1} h \zeta p^{-1} \gamma r \gamma^{-1} \zeta^{-1} h^{-1} \eta^{-1}).$$
Thus the identity among relations involving $p$ and $\gamma r \gamma^{-1}$ following from (\ref{lemToMTeq1}) is
	$$\theta (\gamma r \gamma^{-1}) \theta^{-1} \equiv$$
	$$(\eta \delta) p (\delta^{-1} \eta^{-1}) \centerdot (\eta h \zeta) p^{-1} (\zeta^{-1} h^{-1} \eta^{-1}) \centerdot (\eta h \zeta) (\gamma r \gamma^{-1}) (\zeta^{-1} h^{-1} \eta^{-1}).$$
Since this identity is strictly basic, we have that $\theta = \rho(\eta \delta) = \rho(\eta h \zeta)$ thus we have that
	\begin{equation} \label{lemToMTequa} \delta = \rho(h \zeta), \hspace{5mm} \theta = \rho(\eta \delta).\end{equation}

\medskip

\textbf{Second part.} In the equality of (\ref{impLemEq2}) it is not restrictive to assume that $\beta = 1$; indeed if $\beta \neq 1$ that equality is equivalent to  $\rho(\alpha_1 q w' \alpha_1^{-1}) = \rho(p^{-1} \gamma_1 (d*w) \gamma_1^{-1})$, where $\alpha_1 = \beta^{-1} \alpha$ and $\gamma_1 = \beta^{-1} \gamma$ and the identity among relations involving $p, q, w$ that by Remark \ref{idFrom} follows from it is still basic. In this case we rename $\alpha_1$ and $\gamma_1$ as respectively $\alpha$ and $\gamma$, so this equality becomes
	\begin{equation} \label{miar2'} \rho(\alpha q w' \alpha^{-1}) = \rho(p^{-1}  \gamma (d*w) \gamma^{-1}).\end{equation}	
Our goal now is to find the identity among relations involving $p, q, w$ that follows from (\ref{miar2'}) (by hypothesis this identity is basic).

By (\ref{scope}) and (\ref{permCon2}) of Proposition \ref{summar} there exist words $\lambda, \mu, \nu, \sigma, \tau$ such that $q_0 = \rho(\lambda q \lambda^{-1})$, $w' = \rho(\mu w \mu^{-1})$, $p_0 = \rho(\tau p \tau^{-1})$, $d = \rho(\nu \tau p \tau^{-1} \lambda q \lambda^{-1} \nu^{-1})$, $d*w = \rho(\sigma \nu \tau p \tau^{-1} \lambda q \lambda^{-1} \nu^{-1} w \sigma^{-1})$, so (\ref{miar2'}) is equivalent to
	$$\rho(\alpha q \mu w \mu^{-1} \alpha^{-1}) = \rho(p^{-1} \gamma \sigma \nu \tau p \tau^{-1}  \lambda q \lambda^{-1} \nu^{-1} w \sigma^{-1} \gamma^{-1})$$
which implies the following identity among relations
	$$\alpha q \alpha^{-1} \centerdot (\alpha \mu) w (\mu^{-1} \alpha^{-1}) \equiv$$ 
	$$p^{-1} \centerdot (\gamma \sigma \nu \tau) p (\tau^{-1} \nu^{-1} \sigma^{-1} \gamma^{-1}) \centerdot (\gamma \sigma \nu \lambda) q (\lambda^{-1} \nu^{-1} \sigma^{-1} \gamma^{-1}) \centerdot (\gamma \sigma) w (\sigma^{-1} \gamma^{-1}).$$
Since this identity is basic then the following three equalities hold
	$$\rho(\gamma \sigma \nu \tau) = 1, \hspace{1cm} \alpha = \rho(\gamma \sigma \nu \lambda),  \hspace{1cm}\rho(\alpha \mu) = \rho(\gamma \sigma),$$
which imply that $\rho(\gamma \sigma \nu) = \tau^{-1}$, thus $\alpha = \rho(\tau^{-1} \lambda)$ and that $\rho(\tau^{-1} \nu^{-1}) = \rho(\gamma \sigma) = \rho(\alpha \mu)$. Moreover $\alpha = \rho(\alpha \mu \nu \lambda)$, which implies that $\rho(\lambda^{-1} \nu^{-1}) = \mu$. Therefore we have that
	\begin{equation} \label{chNa} \rho(\gamma (d*w) \gamma^{-1}) = \rho(\gamma \sigma \nu \tau p \tau^{-1} \lambda q \lambda^{-1} \nu^{-1} w \sigma^{-1} \gamma^{-1}) =
	\rho(p \alpha q \mu w \mu^{-1} \alpha^{-1})
	\end{equation}
and this implies together with (\ref{lemToMTeqf}) that
	\begin{equation} \label{lemToMTeqf2} f = \rho(\zeta p^{-1} p \alpha q \mu w \mu^{-1} \alpha^{-1} \zeta^{-1}) = \rho(\zeta \alpha q \mu w \mu^{-1} \alpha^{-1} \zeta^{-1}).	\end{equation}
From (\ref{lemToMTeqf2}) it is easy to deduce that $f \sim q * w'$. Indeed since $w' = \rho(\mu w \mu^{-1})$, by setting $a := \zeta \alpha$ we have that $f = \rho(a q w' a^{-1})$. Since $f$ is cyclically reduced and by virtue of Corollary \ref{permCycRedForm} we have that 
	$$f = \hat{\rho}(a q w' a^{-1}) \sim \hat{\rho}(q w') = q*w'.$$
In view of (\ref{chNa}) we have that
		$$\rho(\theta \gamma r \gamma^{-1} \theta^{-1}) = \rho(\theta p \alpha q \mu w \mu^{-1} \alpha^{-1} \theta^{-1}).$$
In view of (\ref{lemToMTeqf2}) and (\ref{lemToMTequa}) the right hand side of (\ref{lemToMTeq2}) is
		$$\rho(\eta \delta p \delta^{-1} h f h^{-1} \eta^{-1}) = \rho(\eta \delta p \delta^{-1} h \zeta \alpha q \mu w \mu^{-1} \alpha^{-1} \zeta^{-1} h^{-1} \eta^{-1}) =$$ $$\rho(\theta p \alpha q \mu w \mu^{-1} \alpha^{-1} \theta^{-1}).$$
The last two equalities show that the left and right hand side of (\ref{lemToMTeq2}) expressed as products of conjugates of $p, q, w$ are identical, so the identity among relations involving $p, q, w$ and following from (\ref{lemToMTeq2}) is basic. This completes the proof.
\end{proof}

\begin{remark} \label{remlemToMT} \rm Let us take the same hypotheses of Lemma \ref{lemToMT} but replacing (\ref{impLemEq2}) with the following equality:
	\begin{equation} \label{reLeEq1} \rho(\alpha w' q \alpha^{-1}) = \rho(\gamma (w*d) \gamma^{-1} \beta p^{-1} \beta^{-1}).\end{equation}
Then the thesis of Lemma \ref{lemToMT} is true if we take as $f$ the following word, $f := (\gamma (w*d) \gamma^{-1}) * p^{-1}$.

Indeed by virtue of (\ref{reverse2}), (\ref{revCRP}) and (\ref{rfrev}) of Proposition \ref{summar}, (\ref{reLeEq1}) implies the equality
	\begin{equation} \label{reLeEq1'}\rho(\underline{\alpha}^{-1} \, \underline{q} \, \underline{w'} \, \underline{\alpha}) = \rho(\underline{\beta}^{-1} \, \underline{p}^{-1}\, \underline{\beta} \, \underline{\gamma}^{-1}
	 \, (\underline{d} * \underline{w}) \, \underline{\gamma}).\end{equation}
By Remark \ref{idRev}, if the identity among relations involving $p, q, w$ that follows from (\ref{reLeEq1}) is basic, then that involving $\underline{p}, \underline{q}, \underline{w}$ that follows from (\ref{reLeEq1'}) is basic too. 

By (\ref{reversesim}) of Proposition \ref{summar} we have that $\underline{p_0} \sim \underline{p}$, $\underline{q_0} \sim \underline{q}$, $\underline{w'} \sim \underline{w}$; by (\ref{reversesim}) and (\ref{revCRP}) of Proposition \ref{summar} and Proposition \ref{puzo} we have that $\underline{d} \sim \underline{q_0} * \underline{p_0} \sim \underline{p_0} * \underline{q_0}$.

By (\ref{revCRP}) of Proposition \ref{summar} we have that $\underline{f} = \underline{p}^{-1} * (\underline{\gamma}^{-1} (\underline{d} * \underline{w}) \underline{\gamma})$, so we can apply Lemma \ref{lemToMT} and we have that $\underline{f} \sim \underline{q} * \underline{w'}$, so by (\ref{reversesim}) and (\ref{revCRP}) of Proposition \ref{summar} and Proposition \ref{puzo} we have that $f \sim q * w'$. 

Always by Lemma \ref{lemToMT} we have that there exists words $r, k$ such that $r$ is the reduced form of a cyclic permutation of $\underline{p}$ and 
	\begin{equation} \label{eqUnlin} \underline{d} * \underline{w} \sim r * (k \underline{f} k^{-1}),\end{equation}
with $r * (k \underline{f} k^{-1}) = r k \underline{f} k^{-1}$ if $k \neq 1$. Moreover the identity among relations involving $\underline{p}, \underline{q}, \underline{w}$ that by Remark \ref{idFrom} follows from (\ref{eqUnlin}) is basic.

By (\ref{revCRP}) of Proposition \ref{summar} we have that $\underline{\underline{d} * \underline{w}} = w * d$ and that $\underline{r * (k \underline{f} k^{-1})} = (\underline{k}^{-1} f \underline{k}) * \underline{r}$, so by (\ref{reversesim}) of Proposition \ref{summar} we have that 
	\begin{equation} \label{eqUnlin2} w * d \sim (\underline{k}^{-1} f \underline{k}) * \underline{r}.\end{equation}
By (\ref{reverse2}) of Proposition \ref{summar} we have that $(\underline{k}^{-1} f \underline{k}) * \underline{r} = \underline{k}^{-1} f \underline{k} \, \underline{r}$ if $\underline{k} \neq 1$. By Remark \ref{idRev} the identity among relations involving $p, q, w$ that follows from (\ref{eqUnlin2}) is basic.

By Proposition \ref{puzo} we have that $d * w \sim w * d$ and that $(\underline{k}^{-1} f \underline{k}) * \underline{r} \sim \underline{r} * (\underline{k}^{-1} g \underline{k})$.

Let us set $p' := \underline{r}$ and $h := \underline{k}$; (\ref{eqUnlin2}) implies that 
		$$d * w \sim p' * (h f h^{-1}).$$
and by Remark \ref{basicToo} the identity among relations involving $p, q, w$ that follows from it is basic. Moreover if $h \neq 1$ then $(h f h^{-1}) * p' = h f h^{-1} p'$. By Remark \ref{u*v=uv} this implies that if $h \neq 1$ then $p' * (h f h^{-1}) = p' h f h^{-1}$ and the proof is complete.
\end{remark}

\begin{theorem} \label{mainTheor2} Let $u, v, w$ be words and let $d$ be a cyclic permutation of $u*v$. Then there exist words $p, q, w', f, h$ such that	
	\begin{enumerate}[a)]
		
		\item \label{MTa)} either $p$ is the reduced form of a cyclic permutation of $u$ and $q \sim v$ or $p$ is the reduced form of a cyclic permutation of $v$ and $q \sim u$,
		
		\item $w' \sim w$,
		
		\item $f \sim q * w'$
		
	\end{enumerate}
and such that
	\begin{equation} \label{d*w} d*w \sim p*(h f h^{-1}),\end{equation} 
with $p*(h f h^{-1}) = p h f h^{-1}$ if $h \neq 1$. Moreover if $d = u*v$ or $d = v*u$ then $w' = w$; and if $d = uv$ or $d = vu$ then $w' = w$ and $q$ is equal to either $u$ or $v$. Also the identity among relations involving $u, v, w$ that by Remark \ref{idFrom} follows from (\ref{d*w}) is basic.

Now let $u, v, w$ be reduced; if moreover $d \neq 1$ and there is cancellation in the product $d*w$, then in the product $q*w'$ a non-empty word is canceled that is a cyclic permutation of a word canceled in the product $d*w$.
\end{theorem}

\begin{proof} First we assume that the first part of the claim is true when $u, v$ and $w$ are reduced and we show that this part is true also when these words are non-necessarily reduced.

Let $d$ be a cyclic permutation of $u*v$. By (\ref{u*vNonRed}) of Proposition \ref{summar} we have that $u*v = \rho(u) * \rho(v)$, so $d$ is a cyclic permutation of $\rho(u) * \rho(v)$. Then there exist words $p, q', w'', f, h$ such that
	\begin{itemize}
		
		\item [--] either $p$ is the reduced form of a cyclic permutation of $\rho(u)$ and $q' \sim \rho(v)$ or $p$ is the reduced form of a cyclic permutation of $\rho(v)$ and $q' \sim \rho(u)$,
		
		\item [--] $w'' \sim \rho(w)$,
		
		\item [--] $f \sim q' * w''$
\end{itemize}
and such that $d*w'' \sim p' * (h f h^{-1})$, with $p' * (h f h^{-1}) = p' h f h^{-1}$ if $h \neq 1$.

By (\ref{cpCanc}) of Proposition \ref{summar} and Remark \ref{rfcp-rf} there exist words $q, w'$ such that either $p$ is the reduced form of a cyclic permutation of $u$ and $q \sim v$ or $p$ is the reduced form of a cyclic permutation of $v$ and $q \sim u$, $w' \sim w$ and $\rho(q) = \rho(q')$, $\rho(w') = \rho(w'')$. By (\ref{u*vNonRed}) of Proposition \ref{summar} we have that:
	\begin{itemize} 				
		\item [--] $q' * w'' = q * w'$, so $f \sim q * w'$,
		
		\item [--] $d * w'' = d * w'$		
	\end{itemize}
and these imply (\ref{d*w}). Since the identities among relations are the same if we take the reduced forms of the words, we obtain the claim and we can assume that $u, v$ and $w$ are reduced.

\smallskip

Now we prove the claim in the case where at least one between $u, v, w$ and $d$ is equal to 1.

If $u = 1$ then $d$ is a cyclic permutation of $v$ and we set $p := u = 1$, $q := d$, $w' := w$ and $f := q*w$. Then $d*w' = q*w = 1 *(q*w) = p*f$. The identity among relations that follows from this equivalence is basic. Finally since $d = q$, the cancellations in $d*w$ and $q*w$ are equal. If $v = 1$ then the same is true by setting $p := v$, $q := d$ and $f := q*w$. 

If $w = 1$ then we set $p := u$, $q := v$, $w' := w = 1$ and $f := q*w = v$, so $d*w = d$ and since $d \sim u*v$, then $d*w \sim p*f$. It is trivial to see that the identity among relations that follows from this equivalence is basic. Moreover there are no cancellations in the product $d*w$ since $w = 1$. 

Finally if $d = 1$ then $u = v^{-1}$ by (\ref{jarem}) of Proposition \ref{summar}, so the claim follows from Proposition \ref{corToMTh1}. Moreover there are no cancellations in the product $d*w$ since $d = 1$.

\smallskip

Now let us assume that $u, v, w, d \neq 1$. By Lemma \ref{mainLemma} (to be proved in the next sections) there exist words $p$, $q$, $w'$, $\alpha$, $\beta$ and $\gamma$ such that	
\begin{enumerate}[1)]		
	\item either $p \sim u$ and $q \sim v$ or $p \sim v$ and $q \sim u$,
	
	\item $w' \sim w$
	
	\item  either 
	\begin{equation} \label{miar2} \rho(\alpha q w' \alpha^{-1}) = \rho(\beta p^{-1} \beta^{-1} \gamma (d*w) \gamma^{-1})  \end{equation}	
	or 
	\begin{equation} \label{miar'2} \rho(\alpha w' q \alpha^{-1}) = \rho(\gamma (w*d) \gamma^{-1}\beta p^{-1} \beta^{-1})  \end{equation}		
\end{enumerate}	
and the identity among relations involving $p, q, w$ that by Remark \ref{idFrom} follows from either (\ref{miar2}) or (\ref{miar'2}) is basic. Moreover if $d = u*v$ or $d = v*u$ then $w' = w$; and if $d = uv$ or $d = vu$ then $w' = w$ and $q$ is equal to either $u$ or $v$.

If $p \sim u$ and $q \sim v$ we set $p_0 := u$ and $q_0 := v$ and $d$ is a cyclic permutation of $p_0 * q_0$; if $p \sim v$ and $q \sim u$ we set $p_0 := v$ and $q_0 := u$ and $d$ is a cyclic permutation of $q_0 * p_0$, that by Proposition \ref{puzo} is a cyclic permutation of $p_0 * q_0$. Thus in any case $d \sim p_0*q_0$ and by Lemma \ref{lemToMT} and Remark \ref{remlemToMT} there exist words $p', f, h$ such that $p'$ is the reduced form of a cyclic permutation of $p$, $f \sim q * w'$ and 
	$$d*w \sim p' * (h f h^{-1}),$$ 
with $p'*(h f h^{-1}) = p' h f h^{-1}$ if $h \neq 1$. Moreover the identity among relations involving $p, q, w$ that by Remark \ref{idFrom} follows from this equality is basic. 

We have that either $p'$ is the reduced form of a cyclic permutation of $u$ and $q \sim v$ or $p'$ is the reduced form of a cyclic permutation of $v$ and $q \sim u$. By renaming $p$ the word $p'$ the first part of the claim is proved.

\smallskip

Lemma \ref{mainLemma} also implies that if there is cancellation in the product $d*w$ then then there exist reduced words $\zeta, \eta$ not both empty such that if (\ref{miar2}) holds [respectively, if (\ref{miar'2}) holds] then in the product $d*w$ [respectively, $w*d$] the words $\zeta \zeta^{-1}$ and $\eta \eta^{-1}$ are canceled and in the product $q*w'$ [respectively, $w'*q$] the word $\zeta \eta^{-1} \eta \zeta^{-1}$ is canceled. 

Let $\eta=1$; then $\zeta \neq 1$. If (\ref{miar2}) holds then in $d*w$ and $q*w'$ the non-empty word $\zeta \zeta^{-1}$ is canceled, proving the claim. If (\ref{miar'2}) holds then in $w*d$ and $w'*q*$ the non-empty word $\zeta \zeta^{-1}$ is canceled, and by Proposition \ref{puzo} the same word is canceled in $d*w$ and $q*w'$.

Let $\eta \neq 1$. If (\ref{miar2}) holds then the word $\eta \eta^{-1}$ is canceled in $d*w$ and its cyclic permutation $\eta^{-1} \eta$ is canceled in $q*w'$. If (\ref{miar'2}) holds then the word $\eta \eta^{-1}$ is canceled in $w*d$ and its cyclic permutation $\eta^{-1} \eta$ is canceled in $w'*q$. By Proposition \ref{puzo} $\eta \eta^{-1}$ is canceled in $d*w$ and $\eta^{-1} \eta$ is canceled in $q*w$.
\end{proof}

\begin{remark} \label{samProdConjRel} \rm By Remark \ref{idFrom2}, the equivalence (\ref{d*w}) determines two products of conjugates of $u, v, w$ that are equal in the free group. The second part of Theorem \ref{mainTheor2} says that these two products are indeed equal, that is they have the same conjugating elements.
	
The second part of Theorem \ref{mainTheor2} is also a generalization to the cyclically reduced product $*$ of an obvious fact valid for the reduced product $\cdot$ . Indeed, since  $\cdot$ is associative, then for words $u, v, w$ we have that 
		$$(u \cdot v) \cdot w = u \cdot (v \cdot w).$$
The identity among relations following from this equality is the trivial identity
			$$u \centerdot v \centerdot w \equiv u \centerdot v \centerdot w.$$ 
For the cyclically reduced product a generalization of the associative property holds, and this is (\ref{d*w}). The second part of Theorem \ref{mainTheor2} says that there exist words $\alpha, \beta, \gamma$ such that the identity among relations following from (\ref{d*w}) is 		
			$$\alpha u \alpha^{-1} \centerdot \beta v \beta^{-1} \centerdot \gamma w \gamma^{-1} \equiv \alpha u \alpha^{-1} \centerdot \beta v \beta^{-1} \centerdot \gamma w \gamma^{-1}.$$ 	
\end{remark}

\begin{remark} \label{nonCancBoth} \rm The third part of Theorem \ref{mainTheor2} states that if there is non-trivial cancellation in the product $d*w$ then there is also non-trivial cancellation in the product $q*w'$.

The question arises whether if there is non-trivial cancellation in both the products $u*v$ and  $d*w$ then there is also non-trivial cancellation in both $q*w'$ and $p*f$.

The following example shows that this is not the case. Let $X = \{t, x, y, z\}$ and let $u = t x$, $v = x^{-1} y$, $w = y^{-1} x z$. Let $d := u*v$; then $d * w = u*(v*w)$, so the equivalence (\ref{d*w}) is verified trivially in that case. We have that $u * v = t y$ and therefore there is non-trivial cancellation in both $u*v$ and  $d*w$.

However $v * w = z$ and there is no cancellation in the cyclically reduced product between $u*v$ and $w$.
\end{remark}

\begin{remark} \label{nonHomeoVK} \rm By Remark \ref{samProdConjRel} the two products of conjugates of $u, v, w$ determined by the left and right hand side of (\ref{d*w}) are equal term by term. Now the question arises whether, in the case where $u, v$ and $w$ are relators of a group presentation, the van Kampen diagrams associated with $d*w$ and $p*(h f h^{-1})$ are homeomorphic\footnote{A short introduction to van Kampen diagrams is given in Section \ref{sectVKdiag}.}.

This is not the case: the example in Remark \ref{nonCancBoth} is a counter-example also in this case. Indeed let $X = \{t, x, y, z\}$ and let $u = t x$, $v = x^{-1} y$, $w = y^{-1} x z$; then $(u*v)*w = u*(v*w)$. The van Kampen diagrams associated with $(u*v)*w$ and $u*(v*w)$ (see Appendix \ref{sectVKdiag} and Remark \ref{pcrFrom}) are those of Figures \ref{fig:vKLeft} and \ref{fig:vKRight} respectively.

\begin{figure}[h]
	\begin{minipage}[b]{0.4\textwidth}
		\begin{picture}(180, 100) (103, 0)

		\put(163, 40){\oval(100, 50)[r]}
		
		\put(163, 40){\oval(49, 49)}
		
		\put(163, 15){\line(0, 1){50}}	
		
		\put(163, 15){\circle{5}}
		\put(163, 15){\circle*{3}}
		
		\put(163, 65){\circle*{3}}
		
		\put(213, 40){\circle*{3}}
		
		\put(132, 26){$t$}
		\put(138.5, 41){\vector(0, 1){4}}
		
		\put(155, 51){$x$}
		\put(163, 41){\vector(0, -1){4}}
		
		\put(189, 51){$y$}
		\put(187.5, 41){\vector(0, -1){4}}
		
		\put(170, 68){$x$}
		\put(182, 65){\vector(1, 0){4}}
		
		\put(213, 26){$z$}
		\put(209, 23){\vector(-1, -1){4}}

		\end{picture}
		\caption{van Kampen diagram for $(u*v)*w$}
		\label{fig:vKLeft}
	\end{minipage}
	\hfill
	\begin{minipage}[b]{0.4\textwidth}
		\begin{picture}(180, 100) (92, 0)

		\put(187, 40){\oval(60, 60)}

		\put(137, 40){\oval(40, 40)}
		
		\put(177, 40){\oval(40, 40)}
		
		\put(157, 40){\circle*{3}}
		\put(157,40){\circle{5}}
		
		\put(197, 40){\circle*{3}}
		\put(117, 40){\circle*{3}}
		
		\put(145, 12){$t$}
		\put(138, 20){\vector(-1,0){4}}
		
		\put(113, 52){$x$}
		\put(133, 60){\vector(1,0){4}}

		\put(156, 67){$z$}
		\put(179, 70){\vector(1,0){4}}
		
		\put(164, 48){$x^{-1}$}
		\put(175, 60){\vector(1,0){4}}
		
		\put(183, 27){$y$}
		\put(178, 20){\vector(-1,0){4}}

		\end{picture}
		\vspace*{-3mm}
		\caption{van Kampen diagram for $u*(v*w)$}
		\label{fig:vKRight}
	\end{minipage}
\end{figure}

These diagrams are not homeomorphic, in particular the one on the left is topologically a disk while the one on the right is not. However both have an internal edge labeled by $y$. The last part of Theorem \ref{mainTheor2} proves that this property is general: if there is cancellation between $d$ and $w$ on the left side of (\ref{d*w}), one word that is canceled in $d*w$ is also canceled in $q*w'$ in the right side of (\ref{d*w}).

This means that even if the diagrams associated with the two sides of (\ref{d*w}) are not homeomorphic, they have some internal paths that are homeomorphic.
\end{remark}

\begin{remark} \label{} \rm We now show that the results of Theorem \ref{mainTheor2} cannot be improved. The claims we make in this remark can be verified either by long and tedious manual verification or (as the author has done) with the help of a computer program (the author has written a little project in the Java language for that purpose).

In (\ref{d*w}) the word $h$ cannot be always empty. Indeed let $u := x^{2} y^{-1} x^{3} y^{-2} x$, $v := y^{-3} x^{-1} y^{-2}$, $w := x^{-2} y^{-1} x y^{2}$ and let $d = u * v$. Then for no $p, q, w', f$ as defined in Theorem \ref{mainTheor2} is $d*w$ a cyclic permutation of $p * f$. But $d*w = u y^{-2} (v*w) y^{2}$

This is also true for $u := x^{2} y^{-1} x^{3} y^{-2} x$, $v := y^{-3} x^{-1} y^{2}$, $w := x^{-2} y x^{-1} y$, $d := x^{-1} y^{2} x^{2} y^{-1} x^{3} y^{-2} x y^{-3}$. In that case moreover the equivalence in (\ref{d*w}) is never an equality and can be realized only when $q$ is a non-trivial cyclic permutation of $v$: it can be realized either for $p = u$ and $w' = w$, with $f$ a non-trivial cyclic permutation of $q*w'$ or for $p = u$, $f = q * w'$ and $w'$ a non-trivial cyclic permutation of $w$. By swapping $u$ and $v$ we obtain an example when $q$ must be a non-trivial cyclic permutation of $u$.

\smallskip

The following is an example when on the contrary $h$ must be 1: $u := x^{2} y^{-1} x^{3} y^{-2} x$, $v := y^{-3} x^{-1} y^{2}$, $w := x^{-2} y x^{-1} y$,
$d := x^{-1} y^{2} x^{2} y^{-1} x^{3} y^{-2} x y^{-3}$. In this case $q$ must be a non-trivial cyclic permutation of $v$ and either $w'$ be a non-trivial cyclic permutation of $w$ or $f$ be a non-trivial cyclic permutation of $q*w'$. In no case is $d*w$ equal to $p*f$. If we swap $u$ and $v$ we obtain an example when $q$ must be a non-trivial cyclic permutation of $u$.
\end{remark}

\section{The main lemma} \label{lemmaSect}

In order to prove Theorem \ref{mainTheor2} we need to prove the following result.

\begin{lemma} \label{mainLemma} Let $u, v, w$ be reduced non-empty words, let $d$ be a cyclic permutation of $u*v$ and let $d \neq 1$. Then there exist words $p$, $q$, $w'$, $\alpha$, $\beta$ and $\gamma$ such that	
	
\begin{enumerate}[a)]		
	\item either $p \sim u$ and $q \sim v$ or $p \sim v$ and $q \sim u$,
	
	\item $w' \sim w$
	
	\item  either 
		\begin{equation} \label{miar} \rho(\alpha q w' \alpha^{-1}) = \rho(\beta p^{-1} \beta^{-1} \gamma (d*w) \gamma^{-1})  \end{equation}	
	or 
		\begin{equation} \label{miar'} \rho(\alpha w' q \alpha^{-1}) = \rho(\gamma (w*d) \gamma^{-1}\beta p^{-1} \beta^{-1})  \end{equation}		
\end{enumerate}	
and the identity among relations involving $p, q, w, p^{-1}, q^{-1}, w^{-1}$ that by Remark \ref{idFrom} follows from either (\ref{miar}) or (\ref{miar'}) is basic. Moreover if $d = u*v$ or $d = v*u$ then $w' = w$; and if $d = uv$ or $d = vu$ then $w' = w$ and $q$ is equal to either $u$ or $v$.
	
If there are cancellation in the product $d*w$ then there exist reduced words $\zeta, \eta$ not both empty such that if (\ref{miar}) holds [respectively, if (\ref{miar'}) holds] then in the product $d*w$ [respectively, $w*d$] the words $\zeta \zeta^{-1}$ and $\eta \eta^{-1}$ are canceled and in the product $q*w'$ [respectively, $w'*q$] the word $\zeta \eta^{-1} \eta \zeta^{-1}$ is canceled. 

In particular if $\eta = 1$ then the non-empty word $\zeta \zeta^{-1}$ is canceled in both $d*w$ and $q*w'$ [respectively, $w*d$ and $w'*q$]; if $\zeta = 1$ then in the product $d*w$ [respectively, $w*d$] the words $\eta \eta^{-1}$ is canceled and in the product $q*w'$ [respectively, $w'*q$] the word $\eta^{-1} \eta$ is canceled.
\end{lemma}

\begin{proof} We need to prove the claims in the following four cases:	
\begin{enumerate}
	\item there is no cancellation in $u*v$ and $d*w$;
	
	\item there is cancellation in $u*v$ but not in $d*w$;
	
	\item there is no cancellation in $u*v$ but there is in $d*w$;
	
	\item there is cancellation in both $u*v$ and $d*w$.
	
\end{enumerate}
These four cases will be proved in Sections \ref{prLe1}, \ref{prLe2}, \ref{prLe3} and \ref{prLe4} respectively.	
\end{proof}

\section{Proof of Lemma \ref{mainLemma}: first case} \label{prLe1}

In this section we prove Lemma \ref{mainLemma} under the hypothesis that the cyclically reduced products of $u$ by $v$ and of $d$ by $w$ are without cancellations, that is $u*v = uv$ and $d*w = dw$.

Since $d$ is a cyclic permutation of $uv$, four cases are possible: A) $d = uv$; B) $d = u_2 v u_1$ with $u = u_1 u_2$ and $u_1, u_2 \neq 1$; C) $d = vu$; D) $d = v_2 u v_1$ with $v = v_1 v_2$ and $v_1, v_2 \neq 1$. We observe that in cases B and D we have that $d \neq u*v$.

Since the claim of Lemma \ref{mainLemma} is symmetrical in $u$ and $v$ we need only to consider cases A and B because C and D follow from the latter by swapping $u$ and $v$.

\medskip

A) $d = uv$. This implies that $d*w = uvw$. We have that 
	$$\rho(vw) = \rho(u^{-1} uvw) = \rho(u^{-1} (d*w))$$
and the identity among relations that follows from it is 
	$$v \centerdot w \equiv u^{-1} \centerdot u \centerdot v \centerdot w,$$
which is basic.

The first part of the claim follows by setting $q := v$, $w' := w$, $p := u$. The second part does not apply because there are no cancellations in the product $d*w$.

\medskip

B) $d = u_2 v u_1$ with $u = u_1 u_2$ and $u_1, u_2 \neq 1$. This implies that $d = \rho(u_2 v u_1 u_2 u_2^{-1}) = \rho(u_2 v u u_2^{-1})$ and thus $d*w = \rho(u_2 v u u_2^{-1} w)$.

Let us set $u' := u_2 u_1 = \rho(u_2 u u_2^{-1})$. We have that $u'$ is a cyclic permutation of $u$ and we have that 
	$$\rho(u' w) = \rho(u_2 u u_2^{-1} w) = \rho(u_2 v^{-1} u_{2}^{-1} u_{2} v u u_{2}^{-1} w) = \rho(u_2 v^{-1} u_{2}^{-1} (d*w))$$
and the identity among relations that follows from it is
	$$u_2 u u_2^{-1} \centerdot w \equiv u_2 v^{-1} u_{2}^{-1} \centerdot
	u_2 v u_2^{-1}  \centerdot u_2 u u_2^{-1} \centerdot w,$$
which is basic.

The first part of the claim follows by setting $q := u'$, $w' := w$, $p := v$. The second part does not apply because there are no cancellations in the product $d*w$.

\section{Proof of Lemma \ref{mainLemma}: second case}	\label{prLe2}

In this section we suppose that there are cancellations in the cyclically reduced product of $u$ by $v$ but that the cyclically reduced product of $d$ by $w$ is without cancellations, that is $d*w = dw$.

By Lemma \ref{shirv4} there exists a pair of words $p$ and $q_0$ such that either $p \sim u$ and $q_0 \sim v$ or $p \sim v$ and $q_0 \sim u$ and such that one of the following two cases holds: 
\begin{enumerate} [(A)]
	\item $q_0 = p^{-1} r c_1 c_2 r^{-1}$ for words $r, c_1 c_2$ such that $p*q_0 = c_1 c_2$ and $d = c_2 c_1$; 
	
	\item $p = e_2 b$ and $q_0 = b^{-1} e_3 e_1$ for words $b, e_1, e_2, e_3$ such that $d = e_1 e_2 e_3$ and $e_2, e_3 e_1 \neq 1$. 	
\end{enumerate}

A. We have that $d*w = c_2 c_1 w$. The word $q := c_2 r^{-1} p^{-1} r c_1$ is a cyclic permutation of $q_0$ and we have that $\rho(qw) = \rho(c_2 r^{-1} p^{-1} r c_1 w)$, so 
		$$\rho(c_2^{-1} qw c_2) = \rho(r^{-1} p^{-1} r c_1 w c_2).$$
Moreover we have that $c_1 w c_2 = \rho(c_2^{-1} (d*w) c_2)$, so we have the equality
		\begin{equation} \label{ML2A} \rho(c_2^{-1} qw c_2) = \rho(r^{-1} p^{-1} r c_2^{-1} (d*w) c_2). \end{equation}
In order to prove the first part of the lemma we need to prove that the identity among relations involving $p, q, w$ that follows from it is basic. We have that 
		$$\rho(p q_0) = \rho(p p^{-1} r c_1 c_2 r^{-1}) = \rho(r c_1 c_2 r^{-1}),$$
so $\rho(c_2 r^{-1} p q_0 r  c_2^{-1}) = \rho(c_2 c_1) = d$ and since $q_0 = \rho(r c_2^{-1} q  c_2 r^{-1})$ we have that $d = \rho(c_2 r^{-1} p r c_2^{-1} q)$, therefore $d*w = \rho(c_2 r^{-1} p r c_2^{-1} q w)$ and $\rho(c_2^{-1} (d*w) c_2) = \rho(r^{-1} p r c_2^{-1} q w c_2)$ and the identity involving $p, q, w$ that follows from (\ref{ML2A}) is
		$$c_2^{-1} q c_2 \centerdot c_2^{-1} w c_2 \equiv r^{-1} p^{-1} r \centerdot r^{-1} p r \centerdot c_2^{-1} q c_2 \centerdot c_2^{-1} w c_2,$$
which is basic. The second part of the lemma does not apply here because there are no cancellations in the product $d*w$.

\medskip


B. We have that $d*w = e_1 e_2 e_3 w$. The word $q := e_1 b^{-1} e_3$ is a cyclic permutation of $q_0$ and we have that
	\begin{equation} \label{ML2B} \rho(q w) = \rho(e_1 b^{-1} e_3 w) = \rho(e_1 b^{-1} e_2^{-1} e_1^{-1} e_1 e_2 e_3 w) = \rho(e_1 p^{-1} e_1^{-1} (d*w)).	\end{equation}
In order to prove the first part of the lemma we need to prove that the identity among relations involving $p, q, w$ that follows from it is basic. We have that 
	$$\rho(p q_0) = \rho(e_2 b b^{-1} e_3 e_1) = \rho(e_2 e_3 e_1),$$	
so $\rho(e_1 p q_0 e_1^{-1}) = \rho(e_1 e_2 e_3) = d$ and since $q_0 = \rho(e_1^{-1} q e_1)$ we have that $d = \rho(e_1 p e_1^{-1} q e_1 e_1^{-1}) = \rho(e_1 p e_1^{-1} q)$, therefore $d*w = \rho(e_1 p e_1^{-1} q w)$ and the identity involving $p, q, w$ that follows from (\ref{ML2B}) is
	$$ q \centerdot w  \equiv e_1 p^{-1} e_1^{-1} \centerdot e_1 p e_1^{-1} \centerdot q \centerdot w,$$
which is basic. The second part of the lemma  does not apply here because there are no cancellations in the product $d*w$.

\section{Proof of Lemma \ref{mainLemma}: third case}	\label{prLe3}

Let $u, v, w$ be reduced non-empty words and let $d$ be a cyclic permutation of $u*v$. In this section we suppose that the cyclically reduced product of $u$ by $v$ is without cancellations, that is $u*v = uv$ but that there are cancellations in the cyclically reduced product of $d$ by $w$.

We split the proof of Lemma \ref{mainLemma} in different subcases depending on the form of $d$. Indeed by Lemma \ref{shirv} three cases are possible:

\begin{enumerate}[1.]
	\item $d = d_1 a$, $w = a^{-1} s (d*w) s^{-1} d_1^{-1}$, $\rho(dw) = d_1 s (d*w) s^{-1} d_1^{-1}$ for words $d_1, a, s$;	
	
	\item $d = t d_1 a$, $w = a^{-1} w_1 t^{-1}$, $\rho(dw) = t d_1 w_1 t^{-1}$ for words $t, d_1, a, w_1$ such that $d_1, w_1, ta \neq 1$, $d*w = d_1 w_1$ and $w*d = w_1 d_1$;  
	
	\item $d = w_1^{-1} s (d*w) s^{-1} a$, $w = a^{-1} w_1$, $\rho(dw) = w_1^{-1} s (d*w) s^{-1} w_1$ for words $w_1, s, a$.
\end{enumerate}

Since $d$ is a cyclic permutation of $uv$ then four cases are possible: \begin{enumerate}[A.]
	\item $d = uv$;
	
	\item $d = u_2 v u_1$, with $u_1, u_2 \neq 1$ and $u = u_1 u_2$;
	
	\item $d = vu$;
	
	\item $d = v_2 u v_1$ with $v = v_1 v_2$ and $v_1, v_2 \neq 1$.
\end{enumerate}

In cases B and D we have that $d \neq u*v$. We observe that we need only to treat cases A and B because C and D follow from the latter by swapping $u$ and $v$ since the claim is symmetrical in $u$ and $v$. So we will prove Lemma \ref{mainLemma} in subcases 1A, 1B, 2A, 2B, 3A, 3B.

\medskip

\textbf{1A)} $d = d_1 a = uv$, $w = a^{-1} s (d*w) s^{-1} d_1^{-1}$. The word equation $d_1 a = uv$ leads to two different possible solutions.

\medskip

\textbf{1A1)} \hspace{0.1cm}
\begin{tabular}{|c|c|}
	\hline
	\rule{0pt}{2.3ex}
	$\,\,\, u \,\,$ & $v \hspace{0.02cm}$ \\  
	\hline
\end{tabular}

\hspace{1.175cm}
\begin{tabular}{|c|c|}
	\hline 
	\rule{0pt}{2.3ex}
	$d_1$ & $\,\, a \,$ \\
	\hline
\end{tabular}	

\medskip

There exists a word $x$ such that $u = d_1 x$ and $a = x v$, so $d = d_1 x v$, thus $a^{-1} = v^{-1} x^{-1}$ and therefore $w = v^{-1} x^{-1} s (d*w) s^{-1} d_1^{-1}$.

The word $u_0 := x d_1$ is a cyclic permutation of $u$ and $u_0^{-1} = d_1^{-1} x^{-1}$. We have that 	
	$$\rho(vw) = \rho(v v^{-1} x^{-1} s (d*w) s^{-1} d_1^{-1}) = \rho(x^{-1} s (d*w) s^{-1} d_1^{-1}),$$
so
	\begin{equation} \label{ML3.1A1} \rho(d_1^{-1} v w d_1) = \rho(d_1^{-1} x^{-1} s (d*w) s^{-1}) = \rho(u_0^{-1} s (d*w) s^{-1}).\end{equation}
We have that 
	$$\rho(uvw) = \rho(dw) = \rho(d_1 x v v^{-1} x^{-1} s (d*w) s^{-1} d_1^{-1}) = \rho(d_1 s (d*w) s^{-1} d_1^{-1})$$
and therefore $\rho(s (d*w) s^{-1}) = \rho(d_1^{-1} u v w d_1)$. Moreover $u_0 = \rho(d_1^{-1} u d_1)$ and that implies that the identity involving $u, v, w$ that follows from (\ref{ML3.1A1}) is 
	$$d_1^{-1} v d_1 \centerdot d_1^{-1} w d_1  \equiv d_1^{-1} u^{-1} d_1 \centerdot d_1^{-1} u d_1 \centerdot d_1^{-1} v d_1 \centerdot d_1^{-1} w d_1,$$
which is basic. By setting $q := v$, $w' : = w$, $p := u_0$ this proves the first part of the lemma. The second part follows by observing that the non-empty word $v v^{-1}$ is canceled in $d*w$ and $q*w'$.

\medskip

\textbf{1A2)} \hspace{0.1cm}
\begin{tabular}{|c|c|}
	\hline
	\rule{0pt}{2.3ex}
	$u$ & $\,\,\, v  \,\,\,\hspace{0.02cm}$ \\  
	\hline
\end{tabular}

\hspace{1.175cm}
\begin{tabular}{|c|c|}
	\hline 
	\rule{0pt}{2.3ex}
	$\, d_1 \,$ & $\,a \,$ \\
	\hline
\end{tabular}

\medskip

There exists a word $x$ such that $d_1 = u x$, $v = x a$. Thus $d_1^{-1} = x^{-1} u^{-1}$ and therefore $w = a^{-1} s (d*w) s^{-1} x^{-1} u^{-1}$.

Since $\rho(x^{-1} v) = a$, we have that 
	$$\rho(wd) = \rho(wuv) = \rho(a^{-1} s (d*w) s^{-1} x^{-1} u^{-1} u v) =$$ $$\rho(a^{-1} s (d*w) s^{-1} x^{-1} v) = \rho(a^{-1} s (d*w) s^{-1} a).$$
Since $w*d = \hat{\rho}(wd)$, there exists a word $\alpha$ such that $\rho(wd) = \alpha (w*d) \alpha^{-1}$, so the previous equation implies that
	\begin{equation} \label{ML3.1A2pre} \alpha (w*d) \alpha^{-1} = \rho(a^{-1} s (d*w) s^{-1} a) = \rho(wuv).\end{equation}
Since $v^{-1} = a^{-1} x^{-1}$, in view of (\ref{ML3.1A2pre}) we have that
	\begin{multline} \label{ML3.1A2} \rho(wu) = \rho(a^{-1} s (d*w) s^{-1} x^{-1} u^{-1} u) = \rho(a^{-1} s (d*w) s^{-1} x^{-1}) = \\ \rho(a^{-1} s (d*w) s^{-1} a a^{-1} x^{-1}) = \rho(\alpha (w*d) \alpha^{-1} v^{-1})\end{multline}
and by (\ref{ML3.1A2pre}) the identity among relations that follows from (\ref{ML3.1A2}) is 
	$$w \centerdot u  \equiv w \centerdot u \centerdot v \centerdot v^{-1},$$
which is basic. By setting $q := u$, $w' : = w$, $p := v$ this proves the first part of the lemma. The second part follows by observing that the non-empty word $u^{-1} u$ is canceled in $w*d$ and $w'*q$.

\bigskip

\textbf{1B)} $d = d_1 a = u_2 v u_1$, $w = a^{-1} s (d*w) s^{-1} d_1^{-1}$ with $u_1, u_2 \neq 1$. The word equation $d_1 a = u_2 v u_1$ leads to two different possible solutions.

\medskip

\textbf{1B1)} \hspace{0.1cm}
\begin{tabular}{|c|c|}
	\hline
	\rule{0pt}{2.3ex}
	$u_2  v$ & $u_1$ \\  
	\hline
\end{tabular}

\hspace{11.6mm}
\begin{tabular}{|c|c|}
	\hline 
	\rule{0pt}{2.3ex}
	$d_1$ & $\,\,\, a \,\,\,\hspace{0.02cm}$ \\
	\hline
\end{tabular}	

\medskip

There exists a word $x$ such that $u_2 v = d_1 x$ and $a = x u_1$. Thus $d = d_1 x u_1$, $a^{-1} = u_1^{-1} x^{-1}$ and therefore $w = u_1^{-1} x^{-1} s (d*w) s^{-1} d_1^{-1}$. Also $uv = u_1 u_2 v = u_1 d_1 x$.

First we assume that $d_1 \neq 1$. Then since $u_1 \neq 1$ the first letter of $d$ is the first of $d_1$ and the last letter of $d$ is the last of $u_1$. Since $d$ is cyclically reduced these letters are not inverse one of the other. But these two letters are the first and last one of $w^{-1}$, which is cyclically reduced, therefore $w$ is cyclically reduced.

The word $w' := x^{-1} s (d*w) s^{-1} d_1^{-1} u_1^{-1}$ is a cyclic permutation of $w$ and is reduced. Let us set $\overline{a} := x$ and $\overline{d_1} := u_1 d_1$, so $uv = u_1 d_1 x = \overline{d_1} \overline{a}$ and $w' = \overline{a}^{-1} s (d*w) s^{-1} \overline{d_1}^{-1}$. By changing $w$ with $w'$ this case reduces to 1A.

\smallskip

Now let $d_1 = 1$. Then $d = x u_1$, $u_2 v = x$, which implies that $v^{-1} = \rho(x^{-1} u_2)$. Also we have that $w = u_1^{-1} x^{-1} s (d*w) s^{-1}$. 

We have that 
	$$\rho(d w) = \rho(x u_1 u_1^{-1} x^{-1} s (d*w) s^{-1}) = \rho(s (d*w) s^{-1}),$$
thus $s (d*w) s^{-1} = \rho(d w)$. Since $d = \rho(u_2 v u_2^{-1} u_2 u_1) = \rho(u_2 v u_2^{-1} u_2 u u_2^{-1})$ then 
	\begin{equation} \label{ML3.1B} s (d*w) s^{-1} = \rho(u_2 v u_2^{-1} u_2 u u_2^{-1} w).\end{equation}

The word $u_0 := u_2 u_1$ is a cyclic permutation of $u$ and we have that 
	\begin{multline} \label{ML3.1B1} \rho(u_0 w) = \rho(u_2 u_1 u_1^{-1} x^{-1} s (d*w) s^{-1}) = \rho(u_2 x^{-1} s (d*w) s^{-1}) = \\ \rho(u_2 x^{-1} u_2 u_2^{-1} s (d*w) s^{-1}) = \rho(u_2 v^{-1} u_2^{-1} s (d*w) s^{-1}).\end{multline}
Since $u_0 = \rho(u_2 u u_2^{-1})$ and in view of (\ref{ML3.1B}), the identity among relations involving $u, v, w$ that follows from (\ref{ML3.1B1}) is 
	$$u_2 u u_2^{-1} \centerdot w \equiv u_2 v^{-1} u_2^{-1} \centerdot u_2 v u_2^{-1} \centerdot u_2 u u_2^{-1} \centerdot w,$$
which is basic. By setting $q := u_0$, $w' : = w$, $p := v$ this proves the first part of the lemma. The second part follows by observing that the non-empty word $u_1 u_1^{-1}$ is canceled in $d*w$ and $q*w'$.


\bigskip

\textbf{1B2)} \hspace{0.1cm}
\begin{tabular}{|c|c|}
	\hline
	\rule{0pt}{2.3ex}
	$u_2  v$ & $u_1$ \\  
	\hline
\end{tabular}

\hspace{11.6mm}
\begin{tabular}{|c|c|}
	\hline 
	\rule{0pt}{2.3ex}
	$\,\,\, d_1 \,\,\,\hspace{0.02cm}$ & $a$ \\
	\hline
\end{tabular}	

\medskip

There exists a word $x$ such that $d_1 = u_2 v x$ and $u_1 = x a$. Thus $d_1^{-1} = x^{-1} v^{-1} u_2^{-1}$ and therefore $w = a^{-1} s (d*w) s^{-1} x^{-1} v^{-1} u_2^{-1}$. Also $u = u_1 u_2 = x a u_2$.

First we assume that $a \neq 1$. Then since $u_2 \neq 1$ the first letter of $d$ is the first of $u_2$ and the last letter of $d$ is the last of $a$. Since $d$ is cyclically reduced these letters are not inverse one of the other. But these two letters are the first and last one of $w^{-1}$, which is cyclically reduced, therefore $w$ is cyclically reduced.

The word $w' := v^{-1} u_2^{-1} a^{-1} s (d*w) s^{-1} x^{-1}$ is a cyclic permutation of $w$ and is reduced. Let us set $\overline{x} := a u_2$ and $\overline{d_1} := x$, so $u = x a u_2 = \overline{d_1} \, \overline{x}$ and $w' = v^{-1} \overline{x}^{-1} s (d*w) s^{-1} \overline{d_1}^{-1}$. 

By changing $w$ with $w'$ this case reduces to case 1A1.

\smallskip

Now let $a = 1$. Then $u_1 = x$ and $w = s (d*w) s^{-1} u_1^{-1} v^{-1} u_2^{-1}$. We have that 
	$$\rho(w d) = \rho(s (d*w) s^{-1} u_1^{-1} v^{-1} u_2^{-1} u_2 v u_1) = \rho(s (d*w) s^{-1}),$$
which implies that $w*d = d*w$, thus $w = s (w*d) s^{-1} u_1^{-1} v^{-1} u_2^{-1}$.

Since $d = \rho(u_2 u_1 u_1^{-1} v u_1) = \rho(u_1^{-1} u u_1 u_1^{-1} v u_1)$ then 
	\begin{equation} \label{ML3.1B'} s (w*d) s^{-1} = \rho(w u_1^{-1} u u_1 u_1^{-1} v u_1).\end{equation}
The word $u_0 := u_2 u_1$ is a cyclic permutation of $u$ and we have that 
	\begin{equation} \label{ML3.1B2} \rho(w u_0) = \rho(s (w*d) s^{-1} u_1^{-1} v^{-1} u_2^{-1} u_2 u_1) = \rho(s (w*d) s^{-1} u_1^{-1} v^{-1} u_1).\end{equation}
Since $u_0 = \rho(u_1^{-1} u u_1)$ and in view of (\ref{ML3.1B'}), the identity among relations involving $u, v, w$ that follows from (\ref{ML3.1B2}) is 
	$$w \centerdot u_1^{-1} u u_1 \equiv w \centerdot u_1^{-1} u u_1 \centerdot u_1^{-1} v u_1 \centerdot u_1^{-1} v^{-1} u_1,$$
which is basic. By setting $q := u_0$, $w' : = w$, $p := v$ this proves the first part of the lemma. The second part follows by observing that the non-empty word $u_2^{-1} u_2$ is canceled in $w*d$ and $w'*q$.

\bigskip

\textbf{2A)} $d = t d_1 a = uv$, $w = a^{-1} w_1 t^{-1}$, $d*w = d_1 w_1$, $w*d = w_1 d_1$ with $d_1, w_1, ta \neq 1$, so we have 
	$$\rho(uvw) = \rho(t d_1 a a^{-1} w_1 t^{-1}) = \rho(t d_1 w_1 t^{-1})$$
and then
	\begin{equation} \label{3.2A} d*w = d_1 w_1 = \rho(t^{-1} u v w t). 	\end{equation}

The word equation $t d_1 a = uv$ leads to two different possible solutions.
 
\medskip

\textbf{2A1)} \hspace{0.1cm}
\begin{tabular}{|c|c|}
	\hline
	\rule{0pt}{2.3ex}
	$\,\,\,\, u  \,\,\,\,$ & $v \hspace{0.2mm}$ \\  
	\hline
\end{tabular}

\hspace{1.175cm}
\begin{tabular}{|c|c|}
	\hline 
	\rule{0pt}{2.3ex}
	$t d_1$ & $\,\, a \,\,$ \\
	\hline
\end{tabular}	

\medskip

There exists a word $x$ such that $u = t d_1 x$ and $a = x v$. Thus $a^{-1} = v^{-1} x^{-1}$ and therefore $w = v^{-1} x^{-1} w_1 t^{-1}$. Moreover $d = t d_1 x v$.	

The word $u_0 := d_1 x t$ is a cyclic permutation of $u$ and $u_0^{-1} = t^{-1} x^{-1} d_1^{-1}$. We have that 
	$$\rho(vw) = \rho(v v^{-1} x^{-1} w_1 t^{-1}) = \rho(x^{-1} w_1 t^{-1})$$
so 
	$$\rho(t^{-1} v w t) = \rho(t^{-1} x^{-1} w_1) = \rho(t^{-1} x^{-1} d_1^{-1} d_1 w_1) = \rho(u_0^{-1} (d*w))$$
and since $u_0 = \rho(t^{-1} u t)$ we have that
	\begin{equation} \label{ML3.2A1} \rho(t^{-1} v w t) = \rho(t^{-1} u^{-1} t (d*w)).\end{equation}
We have that 
	$$\rho(dw) = \rho(t d_1 x v v^{-1} x^{-1} w_1 t^{-1}) = \rho(t d_1 w_1 t^{-1}) = \rho(t (d*w) t^{-1}),$$
therefore since $d = uv$ then $d*w = \rho(t^{-1} u v w t)$, which implies that the identity involving $u, v, w$ that follows from (\ref{ML3.2A1}) is 
	$$t^{-1} v t \centerdot t^{-1} w t \equiv t^{-1} u^{-1} t \centerdot t^{-1} u t \centerdot t^{-1} v t \centerdot t^{-1} w t,$$
which is basic. By setting $q := v$, $w' : = w$, $p := u$ this proves the first part of the lemma.

The second part follows by observing that the non-empty word $v v^{-1}$ is canceled in $d*w$ and $q*w$.

\bigskip

\textbf{2A2)} 
\hspace{0.1cm}
\begin{tabular}{|c|c|}
	\hline
	\rule{0pt}{2.3ex}
	$\, u \,$ & $\,\, v \,\, \hspace{0.3mm}$ \\  
	\hline
\end{tabular}

\hspace{1.175cm}
\begin{tabular}{|c|c|}
	\hline 
	\rule{0pt}{2.3ex}
	$t d_1$ & $\,a \,$ \\
	\hline
\end{tabular}	

\medskip

There exists a word $v_1$ such that $t d_1 = u v_1$ and $v = v_1 a$. We have that $\rho(t d_1 v_1^{-1}) = u$, so $\rho(d_1 v_1^{-1} t) = \rho(t^{-1} u t)$ and $\rho(t^{-1} v_1 d_1^{-1}) = \rho(t^{-1} u^{-1} t)$

First suppose that $a \neq 1$; then we have that
	$$\rho(vw) = \rho(v_1 a a^{-1} w_1 t^{-1}) = \rho(v_1 w_1 t^{-1}) =$$ 
	$$\rho(v_1 d_1^{-1} d_1 w_1 t^{-1}) = \rho(v_1 d_1^{-1} (d*w) t^{-1}),$$
so
	\begin{equation} \label{ML3.2A2a} \rho(t^{-1} v w t) = \rho(t^{-1} v_1 d_1^{-1} (d*w)) = \rho(t^{-1} u^{-1} t (d*w)).\end{equation}
By (\ref{3.2A}) the identity involving $u, v, w$ that follows from (\ref{ML3.2A2a}) is
	$$t^{-1} v t \centerdot t^{-1} w t \equiv t^{-1} u^{-1} t \centerdot t^{-1} u t \centerdot t^{-1} v t \centerdot t^{-1} w t,$$
which is basic. By setting $q := v$, $w' : = w$, $p := u$ this proves the first part of the lemma. The second part follows by noting that the non-empty word $a a^{-1}$ is canceled in $d*w$ and $q*w$.		

\smallskip

Now suppose that $a = 1$; then $w = w_1 t^{-1}$, $v = v_1$ and $d = t d_1 = u v$. Moreover $t \neq 1$ since $ta \neq 1$. We have that
	\begin{multline} \label{ML3.2A2b++}\rho(wu) = \rho(w_1 t^{-1} u) = \rho(w_1 t^{-1} u v v^{-1}) = \\
 	\rho(w_1 t^{-1} t d_1 v^{-1}) = \rho(w_1 d_1 v^{-1}) = \rho((w*d) v^{-1}).\end{multline}
We have that 
	$$\rho(wuv) = \rho(wd) = \rho(w_1 t^{-1} t d_1) = \rho(w_1 d_1) = w*d,$$
so the identity involving $u, v, w$ that follows from (\ref{ML3.2A2b++}) is
	$$w \centerdot u \equiv w \centerdot u \centerdot v \centerdot v^{-1},$$
which is basic. By setting $w' : = w$, $q := u$, $p := v$ this proves the first part of the lemma. 

For the second part we observe that $u$ and $t$ share a non-empty prefix $t_0$ and that $t_0^{-1}$ is a suffix of $w$, so $t_0^{-1} t_0$ is canceled in $w*q$. Since $d = uv$ then $t_0^{-1} t_0$ is canceled in $w*d$.

\bigskip

\textbf{2B)} $d = t d_1 a = u_2 v u_1$, $w = a^{-1} w_1 t^{-1}$, $d*w = d_1 w_1$, $w * d = w_1 d_1$, with $u_1, u_2, d_1, w_1,$ $ta \neq 1$. We have that 
	$$\rho(d w) = \rho(t d_1 a a^{-1} w_1 t^{-1}) = \rho(t d_1 w_1 t^{-1}) = \rho(t (d*w) t^{-1}),$$
and since $d = \rho(u_2 v u u_2^{-1})$ then
	\begin{equation} \label{ML3.2B} d*w = \rho(t^{-1} d w t) = \rho(t^{-1} u_2 v u u_2^{-1} w t).\end{equation}
We also have that 
	$$\rho(w d) = \rho(a^{-1} w_1 t^{-1} t d_1 a) = \rho(a^{-1} w_1 d_1 a) = \rho(a^{-1} (w*d) a),$$
and since $d = \rho(u_1^{-1} u v u_1)$ then
	\begin{equation} \label{ML3.2B'} w*d = \rho(a w d a^{-1}) = \rho(a w u_1^{-1} u v u_1  a^{-1}).\end{equation}

\medskip

We split case 2B into four subcases: 1) $d_1$ is a subword of $u_2 v$, 2) $d_1$ is a subword of $u_1$, 3) $d_1$ is a subword of $v u_1$ but neither of $v$ nor of $u_1$, 4) $d_1$ is neither a subword of $u_2 v$ nor of $v u_1$.

\medskip

\textbf{2B1)} $d_1$ is a subword of $u_2 v$.

\medskip

\begin{tabular}{|c|c|}
	\hline
	\rule{0pt}{2.3ex}
	$\,\,\,\,  u_2 v \,\,\,\, $ & $u_1 \hspace{0.3mm}$ \\  
	\hline
\end{tabular}

\begin{tabular}{|c|c|c|}
	\hline 
	\rule{0pt}{2.3ex}
	$t$ & $d_1$ & $\,\, \,a \,\,\, $ \\
	\hline
\end{tabular}	

\medskip

There exists a word $x$ such that $u_2 v = t d_1 x$ and $a = x u_1$. Therefore $d = t d_1 x u_1$, $a^{-1} = u_1^{-1} x^{-1}$ and thus $w = u_1^{-1} x^{-1} w_1 t^{-1}$. Also $uv = u_1 u_2 v = u_1 t d_1 x$.

First we assume that $t \neq 1$. Then since $u_1 \neq 1$ the first letter of $d$ is the first of $t$ and the last letter of $d$ is the last of $u_1$. Since $d$ is cyclically reduced these letters are not inverse one of the other. But these two letters are the first and last one of $w^{-1}$, which is cyclically reduced, therefore $w$ is cyclically reduced.

The word $w' := x^{-1} w_1 t^{-1} u_1^{-1}$ is a cyclic permutation of $w$ and is reduced. Let us set $\overline{a} := x$, $\overline{t} := u_1 t$, so $\overline{t}^{-1} := t^{-1} u_1^{-1}$. Thus $uv = u_1 t d_1 x = \overline{t} d_1 \overline{a}$ and $w' := \overline{a}^{-1} w_1 \overline{t}^{-1}$. By changing $w$ with $w'$ this case reduces to 2A.

\smallskip

Now let $t = 1$. This implies that $d = d_1 x u_1$, $w = u_1^{-1} x^{-1} w_1$ and that $u_2 v = d_1 x$, thus $v = \rho(u_2^{-1} d_1 x)$, which implies that $v^{-1} = \rho(x^{-1} d_1^{-1} u_2)$. 

The word $u_0 := u_2 u_1$ is a cyclic permutation of $u$ and we have that 
	\begin{multline} \label{ML3.2B1} \rho(u_0 w) = \rho(u_2 u_1 u_1^{-1} x^{-1} w_1) = \rho(u_2 x^{-1} w_1) = \rho(u_2 x^{-1} d_1^{-1} d_1 w_1) = \\ \rho(u_2 x^{-1} d_1^{-1} u_2 u_2^{-1} (d*w)) = \rho(u_2 v^{-1} u_2^{-1} (d*w)).\end{multline}	
Since $u_0 = \rho(u_2 u u_2^{-1})$ and in view of (\ref{ML3.2B}), the identity among relations that follows from (\ref{ML3.2B1}) is  
	$$u_2 u u_2^{-1} \centerdot w \equiv u_2 v^{-1} u_2^{-1} \centerdot u_2 v u_2^{-1} \centerdot u_2 u u_2^{-1} \centerdot w,$$
which is basic. By setting $q := u_0$, $w' := w$ and $p := v$ this proves the first part of the lemma. For the second part we observe that the word $u_1 u_1^{-1}$ is canceled in $d*w$ and $q*w'$ and $u_1 \neq 1$.

\medskip

\textbf{2B2)} $d_1$ is a subword of $u_1$.

\medskip

\begin{tabular}{|c|c|}
	\hline
	\rule{0pt}{2.3ex}
	$u_2 v$ & $\,\,\,\,\,\,  u_1 \,\,\,\,\,\,\hspace{0.3mm}$ \\  
	\hline
\end{tabular}

\begin{tabular}{|c|c|c|}
	\hline 
	\rule{0pt}{2.3ex}
	$\,\,\,\,\, t \,\,\,\,\,$ & $d_1$ & $a$ \\
	\hline
\end{tabular}	

\medskip

There exists a word $x$ such that $t = u_2 v x$ and $u_1 = x d_1 a$. Then $t^{-1} = x^{-1} v^{-1} u_2^{-1}$ and therefore $w = a^{-1} w_1 x^{-1} v^{-1} u_2^{-1}$. Also $u = u_1 u_2 = x d_1 a u_2$ and therefore $uv = x d_1 a u_2 v$.

First we assume that $a \neq 1$. Then since $u_2 \neq 1$ the first letter of $d$ is the first of $u_2$ and the last letter of $d$ is the last of $a$. Since $d$ is cyclically reduced these letters are not inverse one of the other. But these two letters are the first and last one of $w^{-1}$, which is cyclically reduced, therefore $w$ is cyclically reduced.

The word $w' := v^{-1} u_2^{-1} a^{-1} w_1 x^{-1}$ is a cyclic permutation of $w$ and is reduced. We set $\overline{a} := a u_2 v$, $\overline{t} := x$, so $\overline{a}^{-1} := v^{-1} u_2^{-1} a^{-1}$, $uv =  \overline{t} d_1 \overline{a}$ and $w' := \overline{a}^{-1} w_1 \overline{t}^{-1}$.

By changing $w$ with $w'$ this case reduces to 2A.

\smallskip

Now let $a = 1$. Then $u_1 = x d_1$ and $w = w_1 x^{-1} v^{-1} u_2^{-1}$. The word $u_0 := u_2 u_1$ is a cyclic permutation of $u$ and we have that 
	\begin{multline} \label{ML3.2B2} \rho(w u_0) = \rho(w_1 x^{-1} v^{-1} u_2^{-1} u_2 u_1) = \rho(w_1 x^{-1} v^{-1} u_1) = \rho(w_1 d_1 d_1^{-1} x^{-1} v^{-1} u_1) = \\ \rho((w*d) d_1^{-1} x^{-1} v^{-1} u_1) = \rho((w*d) u_1^{-1} v^{-1} u_1).\end{multline}	
Since $u_0 = \rho(u_1^{-1} u u_1)$ and in view of (\ref{ML3.2B'}), the identity among relations that follows from (\ref{ML3.2B2}) is  
	$$w \centerdot u_1^{-1} u u_1 \equiv w \centerdot u_1^{-1} u u_1 \centerdot u_1^{-1} v u_1 \centerdot u_1^{-1} v^{-1} u_1,$$
which is basic. By setting $q := u_0$, $w' := w$ and $p := v$ this proves the first part of the lemma. For the second part we observe that the word $u_2^{-1} u_2$ is canceled in $w*d$ and $w'*q$ and $u_2 \neq 1$.

\bigskip

\textbf{2B3)} $d_1$ is a subword of $v u_1$ but neither of $v$ nor of $u_1$.

\medskip

\begin{tabular}{|c|c|c|}
	\hline
	\rule{0pt}{2.3ex}
	$u_2$ & $v$ & $\,\, u_1 \,\,$ \\  
	\hline
\end{tabular}

\begin{tabular}{|c|c|c|}
	\hline 
	\rule{0pt}{2.3ex}
	$\,\,\,\, t \,\,\,\,$ & $d_1$ & $a$ \\
	\hline
\end{tabular}	

\medskip

There exist words $x, d_2, d_3$ such that $t = u_2 x$, $v = x d_2$, $d_1 = d_2 d_3$, $u_1 = d_3 a$. Thus $t^{-1} = x^{-1} u_2^{-1}$ and therefore $w = a^{-1} w_1 x^{-1} u_2^{-1}$. Also $u = u_1 u_2 = d_3 a u_2$, $d = u_2 x d_2 d_3 a$ and $d*w = d_2 d_3 w_1$.

The word $w' := u_2^{-1} a^{-1} w_1 x^{-1}$ and $v_0 := d_2 x$ are cyclic permutations of $w$ and $v$ respectively and $v_0^{-1} := x^{-1} d_2^{-1}$. We have that 
	$$\rho(uw') = \rho(d_3 a u_2 u_2^{-1} a^{-1} w_1 x^{-1}) = \rho(d_3 w_1 x^{-1}) =$$
	$$\rho(d_2^{-1} d_2 d_3 w_1 x^{-1}) = \rho(d_2^{-1} (d*w) x^{-1}),$$
so
	\begin{equation} \label{ML3.2B3} \rho(x^{-1} u w' x) = \rho(x^{-1} d_2^{-1} (d*w)) = \rho(v_0^{-1} (d*w)).\end{equation}	
Since $w' = \rho(u_2^{-1} w u_2)$ and $v_0 = \rho(x^{-1} v x)$ and in view of (\ref{ML3.2B}), the identity among relations that follows from (\ref{ML3.2B3}) is  
	$$x^{-1} u x \centerdot (x^{-1} u_2^{-1}) w (u_2 x) \equiv x^{-1} v^{-1} x \centerdot (t^{-1} u_2) v (u_2^{-1} t) \centerdot (t^{-1} u_2) u (u_2^{-1} t)  \centerdot t^{-1} w t.$$
Since $\rho(t^{-1} u_2) = x^{-1}$ and $\rho(x^{-1} u_2^{-1}) = t^{-1}$, this identity is basic and by setting $q := u$ and $p := v_0$ this proves the first part of the lemma.

The second part follows from the fact that the non-empty word $u_2^{-1} u_2$ is canceled in $q*w'$ and by Proposition \ref{puzo} it is canceled also in $d*w$.


\bigskip

\textbf{2B4)} $d_1$ is neither a subword of $u_2 v$ nor of $v u_1$.

\medskip

\begin{tabular}{|c|c|c|}
	\hline
	\rule{0pt}{2.3ex}
	$u_2$ & $v$ & $u_1$ \\  
	\hline
\end{tabular}

\begin{tabular}{|c|c|c|}
	\hline 
	\rule{0pt}{2.3ex}
	$t$ & $\,\, d_1 \,\,$ & $a \hspace{0.1mm}$ \\
	\hline
\end{tabular}

\medskip

We have that $d_1$ overlaps with $u_1$, otherwise it would be a subword of $u_2 v$. Also $d_1$ overlaps with $u_2$, otherwise it would be a subword of $v u_1$. This implies that $t$ is a prefix of $u_2$ and $a$ is a suffix of $u_1$. Therefore there exist words $d_2, d_3$ such that $u_2 = t d_2$, $d_1 = d_2 v d_3$ and $u_1 = d_3 a$. Thus $d*w = d_2 v d_3 w_1$ and $u = u_1 u_2 = d_3 a t d_2$.

The word $u' := u_2 u_1 = t d_2 d_3 a$ is a cyclic permutation of $u$. We have that 
	$$\rho(u' w) = \rho(t d_2 d_3 a a^{-1} w_1 t^{-1}) = \rho(t d_2 d_3 w_1 t^{-1}) =$$
	$$\rho(t d_2 v^{-1} d_2^{-1} d_2 v d_3 w_1 t^{-1}) = \rho(t d_2 v^{-1} d_2^{-1} (d*w) t^{-1}),$$
so
	\begin{equation} \label{ML3.2B4} \rho(t^{-1} u' w t) = \rho(d_2 v^{-1} d_2^{-1} (d*w)).\end{equation}	
Since $u' = \rho(u_2 u u_2^{-1})$ and in view of (\ref{ML3.2B}), the identity among relations that follows from (\ref{ML3.2B4}) is
	$$(t^{-1} u_2) u (u_2^{-1} t) \centerdot t^{-1} w t \equiv d_2 v^{-1} d_2^{-1} \centerdot (t^{-1} u_2) v (u_2^{-1} t) \centerdot (t^{-1} u_2) u (u_2^{-1} t)  \centerdot t^{-1} w t.$$
Since $\rho(t^{-1} u_2) = d_2$, this identity is basic and by setting $q := u'$, $w' := w$ and $p := v$ this proves the first part of the lemma. 

For the second part we observe that $ta \neq 1$ and that $q = t d_2 d_3 a$, $w' = a^{-1} w_1 t^{-1}$, $d = t d_1 a$, so by Proposition \ref{puzo} the words $a a^{-1}$ and $t^{-1} t$ are canceled in $q*w'$ and in $d*w$.



\bigskip

\textbf{3A)} $d = uv = w_1^{-1} s (d*w) s^{-1} a$, $w = a^{-1} w_1$. By setting $g := d*w$ we have that $uv = w_1^{-1} s g s^{-1} a$.

We have that
	$$\rho(uvw) = \rho(dw) = \rho(w_1^{-1} s g s^{-1} a a^{-1} w_1) = \rho(w_1^{-1} s g s^{-1} w_1),$$
so
	\begin{equation} \label{3.3A} \rho(w_1 uvw w_1^{-1}) = s g s^{-1}.\end{equation}

We also have that
	\begin{equation} \label{3.3A0.5} \rho(wuv) = \rho(wd) = \rho(a^{-1} w_1 w_1^{-1} s g s^{-1} a) = \rho(a^{-1} s g s^{-1} a),\end{equation}
which implies that 
	\begin{equation} \label{3.3A'} \rho(a wuv a^{-1}) = \rho(a wd a^{-1}) = s g s^{-1}.\end{equation}
Since $w*d = \hat{\rho}(wd)$, there exists a word $\alpha$ such that $\rho(wd) = \alpha (w*d) \alpha^{-1}$, so (\ref{3.3A0.5}) implies that
	\begin{equation} \label{3.3A1/2} \alpha (w*d) \alpha^{-1} = \rho(a^{-1} s g s^{-1} a) = \rho(wuv).\end{equation}
If $a = 1$ then (\ref{3.3A'}) implies that
	\begin{equation} \label{3.3A''} w*d = g = d*w.\end{equation}

The word equation $uv = w_1^{-1} s (d*w) s^{-1} a$ leads to three different possible solutions.

\medskip

\textbf{3A1)}  \hspace{0.1cm}
\begin{tabular}{|c|c|}
	\hline
	\rule{0pt}{2.3ex}
	$\,\,\,\,\,\,\,\,\,\,\,\,\,\,\,\,\,  u \,\,\,\,\,\,\,\,\,\,\,\,\,\,\,\, \hspace{0.45mm}$ & $v$ \\  
	\hline
\end{tabular}

\hspace{11.75mm}
\begin{tabular}{|c|c|c|}
	\hline 
	\rule{0pt}{2.3ex}
	$w_1^{-1}$ & $s g s^{-1}$ & $\,\, a \,\,\,$ \\
	\hline
\end{tabular}

\medskip

There exists a word $a_1$ such that $u = w_1^{-1} s g s^{-1} a_1$ and $a = a_1 v$. Thus $a^{-1} = v^{-1} a_1^{-1}$ and $w = v^{-1} a_1^{-1} w_1$.

The word $u_0 := s g s^{-1} a_1 w_1^{-1} = \rho(w_1 u w_1^{-1})$ is a cyclic permutation of $u$ and $u_0^{-1} = w_1 a_1^{-1} s g^{-1} s^{-1}$.

Since $a_1^{-1} =\rho(u^{-1} w_1^{-1} s g s^{-1})$ then
	$$\rho(vw) = \rho(v v^{-1} a_1^{-1} w_1) = \rho(a_1^{-1} w_1) = \rho(u^{-1} w_1^{-1} s g s^{-1} w_1),$$
therefore
	\begin{equation} \label{ML3.3A1}\rho(w_1 v w w_1^{-1}) = \rho(w_1 u^{-1} w_1^{-1} s g s^{-1}).\end{equation}
By virtue of (\ref{3.3A}) the identity involving $u, v, w$ that follows from (\ref{ML3.3A1}) is	
	$$w_1 v w_1^{-1} \centerdot w_1 w w_1^{-1} \equiv  w_1 u^{-1} w_1^{-1} \centerdot w_1 u w_1^{-1} \centerdot w_1 v w_1^{-1} \centerdot w_1 w w_1^{-1},$$
which is basic. By setting $q := v$, $w' := w$, $p := u$ this proves the first part of the lemma.	
	
Since $d = uv = w_1^{-1} s g s^{-1} a_1 v$ and $w = v^{-1} a_1^{-1} w_1$, the non-empty word $v v^{-1}$ is canceled in $d*w$ and in $q*w'$ and that proves the second part of the lemma.

\bigskip

\textbf{3A2)} \hspace{0.1cm}
\begin{tabular}{|c|c|}
	\hline
	\rule{0pt}{2.3ex}
	$\,\,\,\,\,\,\,\,\,\,u\,\,\,\,\,\,\,\,\,\,$ & $\,\,\,\, v \,\,\,\,\hspace{0.5mm}$ \\  
	\hline
\end{tabular}

\hspace{11.8mm}
\begin{tabular}{|c|c|c|}
	\hline 
	\rule{0pt}{2.3ex}
	$w_1^{-1}$ & $s g s^{-1}$ & $a$ \\
	\hline
\end{tabular}

\medskip

There exist words $x_1, x_2$ such that $u = w_1^{-1} x_1$, $s g s^{-1} = x_1 x_2$, $v = x_2 a$. The word $u_0 := x_1 w_1^{-1}$ is a cyclic permutation of $u$ and $u_0^{-1} = w_1 x_1^{-1}$.

Let $a \neq 1$; then 
	$$\rho(vw) = \rho(x_2 a a^{-1} w_1) = \rho(x_2 w_1) = \rho(x_1^{-1} x_1 x_2 w_1) = \rho(x_1^{-1} s g s^{-1} w_1),$$
so
	\begin{equation} \label{ML3.3A2a} \rho(w_1 v w  w_1^{-1}) = \rho(w_1 x_1^{-1} s g s^{-1}) = \rho(u_0^{-1} s g s^{-1}).\end{equation}
Since $u_0 = \rho(w_1 u w_1^{-1})$ and by virtue of (\ref{3.3A}) the identity involving $u, v, w$ that follows from (\ref{ML3.3A2a}) is	
	$$w_1 v w_1^{-1} \centerdot w_1 w w_1^{-1} \equiv  w_1 u^{-1} w_1^{-1} \centerdot w_1 u w_1^{-1} \centerdot w_1 v w_1^{-1} \centerdot w_1 w w_1^{-1},$$
which is basic. By setting $q := v$, $w' := w$, $p := u_0$ this proves the first part of the lemma.

Since $d = uv = w_1^{-1} x_1 x_2 a$, the non-empty word $a a^{-1}$ is canceled in $d*w$ and in $q*w'$, which proves the second part of the lemma.

\smallskip

Now let $a = 1$; then $w = w_1$, $u = w^{-1} x_1$, $v = x_2$, $d = w^{-1} s g s^{-1}$. By (\ref{3.3A''}) we have that $w*d = g = d*w$.

We also have that
	\begin{multline} \label{ML3.3A2b} \rho(wu) = \rho(w w^{-1} x_1) = \rho(x_1) = \rho(x_1 x_2 x_2^{-1}) = \\ \rho(s g s^{-1} v^{-1}) = \rho(s (w*d) s^{-1} v^{-1})\end{multline}
and from (\ref{3.3A'}) the identity among relations that follows from (\ref{ML3.3A2b}) is  
	$$w \centerdot u \equiv  w \centerdot u \centerdot v  \centerdot v^{-1},$$
which is basic. By setting $q := u$, $w' := w$, $p := v$ this proves the first part of the lemma.

Since $d = w^{-1} s (d*w) s^{-1}$ and $q = w^{-1} x_1$, the non-empty word $w w^{-1}$ is canceled in $w*d$ and also in $w'*q$, proving the second part of the lemma.

\bigskip

\textbf{3A3)} \hspace{0.1cm}
\begin{tabular}{|c|c|}
	\hline
	\rule{0pt}{2.3ex}
	$u$ & $\,\,\,\,\,\,\,\,\,\,\,\,\,\,  v \,\,\,\,\,\,\,\,\,\,\,\,\,\,\hspace{0.5mm}$ \\  
	\hline
\end{tabular}

\hspace{11.8mm}
\begin{tabular}{|c|c|c|}
	\hline 
	\rule{0pt}{2.3ex}
	$w_1^{-1}$ & $s g s^{-1}$ & $a$ \\
	\hline
\end{tabular}

\medskip

There exists a word $v_1$ such that $w_1^{-1} = u v_1$ and $v = v_1 s  g s^{-1} a$. Thus $w_1 = v_1^{-1} u^{-1}$ and therefore $w = a^{-1} v_1^{-1} u^{-1}$.

By (\ref{3.3A1/2}) we have that
	\begin{multline} \label{ML3.3A3} \rho(w u) = \rho(a^{-1} v_1^{-1} u^{-1} u) = \rho(a^{-1} v_1^{-1}) = \\ \rho(a^{-1} s g s^{-1} a a^{-1} s g^{-1} s^{-1} v_1^{-1}) = \rho(\alpha (w*d) \alpha^{-1} v^{-1}),\end{multline} 
and from (\ref{3.3A1/2}) the identity among relations that follows from (\ref{ML3.3A3}) is  
	$$w \centerdot u \equiv  w \centerdot u \centerdot v  \centerdot v^{-1},$$
which is basic. By setting $q := u$, $w' := w$, $p := v$ this proves the first part of the lemma. The second part follows from the fact that the non-empty word $u^{-1} u$ is canceled in $w*d$ and $w'*q$.


\bigskip

\textbf{3B)} By setting $g := d*w$ we have that $d = u_2 v u_1 = w_1^{-1} s g s^{-1} a$, $w = a^{-1} w_1$, $\rho(dw) = w_1^{-1} s g s^{-1} w_1$

Since $d = \rho(u_2 v u u_2^{-1})$ then 
	\begin{multline} \label{sgs33B} s g s^{-1} = \rho(w_1 d w w_1^{-1}) = \rho(w_1 u_2 v u u_2^{-1} w w_1^{-1}) =\\ \rho((w_1 u_2) v u (u_2^{-1} w_1^{-1}) w_1 w w_1^{-1})\end{multline}

We also have that $\rho(wd) = \rho(a^{-1} w_1 w_1^{-1} s g s^{-1} a) = \rho(a^{-1} s g s^{-1} a)$, thus $s g s^{-1} = \rho(a wd a^{-1})$ and since $d = \rho(u_1^{-1} u v u_1)$ then 
	\begin{equation} \label{sgs33B'} s g s^{-1} = \rho(a w d a^{-1}) = \rho(a w u_1^{-1} u v u_1 a^{-1}).\end{equation}

We split case 3B into three subcases.

\medskip

\textbf{3B1)} $u_1$ is a suffix of $a$.

\begin{tabular}{|c|c|}
	\hline
	\rule{0pt}{2.3ex}
	$\,\,\,\,\,\,\,\,\,\,  u_2 v \,\,\,\,\,\,\,\,\,\,$ & $u_1$ \\  
	\hline
\end{tabular}

\begin{tabular}{|c|c|}
	\hline 
	\rule{0pt}{2.3ex}
	$w_1^{-1} s g s^{-1}$ & $\,\,\, a \,\,\,\hspace{0.3mm}$ \\
	\hline
\end{tabular}

\medskip

Since $u_1$ is a suffix of $a$ then $w_1^{-1} s g s^{-1}$ is a prefix of $u_2 v$, therefore there exists a word $a_1$ such that $u_2 v = w_1^{-1} s g s^{-1} a_1$ and $a = a_1 u_1$. Thus $a^{-1} = u_1^{-1} a_1^{-1}$ and therefore $w = u_1^{-1} a_1^{-1} w_1$.

First we assume that $w_1^{-1} \neq 1$. Then since $u_1 \neq 1$ the first letter of $d$ is the first of $w_1^{-1}$ and the last letter of $d$ is the last of $u_1$. Since $d$ is cyclically reduced these letters are not inverse one of the other. But these two letters are the first and last one of $w^{-1}$, which is cyclically reduced, therefore $w$ is cyclically reduced.

We have that $u v = u_1 u_2 v = u_1 w_1^{-1} s g s^{-1} a_1$. The word $w' := a_1^{-1} w_1 u_1^{-1}$ is a cyclic permutation of $w$ and is reduced. We set $\overline{w_1} := w_1 u_1^{-1}$ and $\overline{a} := a_1$, so $\overline{w_1}^{-1} = u_1 w_1^{-1}$, therefore $u v = \overline{w_1}^{-1} s g s^{-1} \overline{a}$ and $w' = \overline{a}^{-1} \overline{w_1}$.

By changing $w$ with $w'$ this case reduces to 3A.

\smallskip

Now let $w_1^{-1} = 1$, that is $w_1 = 1$. This implies that $w = u_1^{-1} a_1^{-1}$ and $u_2 v = s g s^{-1} a_1$ therefore $v = \rho(u_2^{-1} s g s^{-1} a_1)$.

The word $u_0 := u_2 u_1$ is a cyclic permutation of $u$ and since $\rho(u_2 a_1^{-1} s g^{-1} s^{-1}) = \rho(u_2 v^{-1} u_2^{-1})$ we have that 
	\begin{multline} \label{ML3.3B1} \rho(u_0 w) = \rho(u_2 u_1 u_1^{-1} a_1^{-1}) = \rho(u_2 a_1^{-1}) = \\ \rho(u_2 a_1^{-1} s g^{-1} s^{-1} s g s^{-1}) = \rho(u_2 v^{-1} u_2^{-1} s g s^{-1}).\end{multline}
Since $u_0 = \rho(u_2 u u_2^{-1})$ and in view of (\ref{sgs33B}) the identity among relations that follows from (\ref{ML3.3B1}) is 
	$$u_2 u u_2^{-1} \centerdot w \equiv u_2 v^{-1} u_2^{-1} \centerdot u_2 v u_2^{-1} \centerdot u_2 u u_2^{-1} \centerdot w,$$
which is basic. By setting $q := u_0$, $w' := w$, $p := v$ this proves the first part of the lemma.

The second part of the lemma follows by noting that the word $u_1 u_1^{-1}$ is canceled in $d*w$ and $q'*w$ and that $u_1 \neq 1$.

\medskip

\textbf{3B2)} $u_2$ is a prefix of $w_1^{-1}$.

\begin{tabular}{|c|c|}
	\hline
	\rule{0pt}{2.3ex}
	$u_2$ & $\,\,\,\,\,\,\,\,\,\, v u_1 \,\,\,\,\,\,\,\,\,\,$ \\  
	\hline
\end{tabular}

\begin{tabular}{|c|c|}
	\hline 
	\rule{0pt}{2.3ex}
	$w_1^{-1}$ & $\,\,\, s g s^{-1} a \,\,\,\hspace{0.3mm}$ \\
	\hline
\end{tabular}

\medskip

Since $u_2$ is a prefix of $w_1^{-1}$ then $s g s^{-1} a$ is a suffix of $v u_1$, therefore there exists a word $v_1$ such that $w_1^{-1} = u_2 v_1$ and $v u_1 = v_1 s g s^{-1} a$.

First we assume that $a \neq 1$. Then since $u_2 \neq 1$ the first letter of $d$ is the first of $u_2$ and the last letter of $d$ is the last of $a$. Since $d$ is cyclically reduced these letters are not inverse one of the other. But these two letters are the first and last one of $w^{-1}$, which is cyclically reduced, therefore $w$ is cyclically reduced.

We have that $v u = v u_1 u_2 = v_1 s g s^{-1} a u_2$ and that $w_1 = v_1^{-1} u_2^{-1}$, therefore $w = a^{-1} v_1^{-1} u_2^{-1}$. The word $w' := u_2^{-1} a^{-1} v_1^{-1}$ is a cyclic permutation of $w$. We set $\overline{a} := a u_2$ and $\overline{w_1} := v_1^{-1}$, so $v u = \overline{w_1}^{-1} s g s^{-1} \overline{a}$ and $w' = \overline{a}^{-1} \overline{w_1}$.

By changing $w$ with $w'$ and swapping $u$ and $v$ this case reduces to 3A.

\smallskip

Now let $a = 1$. Then $w = w_1$, $w^{-1} = u_2 v_1$, thus $w = v_1^{-1} u_2^{-1}$; moreover $v u_1 = v_1 s g s^{-1}$, therefore $v = \rho(v_1 s g s^{-1} u_1^{-1})$.

The word $u_0 := u_2 u_1$ is a cyclic permutation of $u$ and since $\rho(s g^{-1} s^{-1} v_1^{-1} u_1) = \rho(u_1^{-1} v^{-1} u_1)$ we have that 
	\begin{multline} \label{ML3.3B2} \rho(w u_0) = \rho(v_1^{-1} u_2^{-1} u_2 u_1) = \rho(v_1^{-1} u_1) = \\ \rho(s g s^{-1} s g^{-1} s^{-1} v_1^{-1} u_1) = \rho(s g s^{-1} u_1^{-1} v^{-1} u_1).\end{multline}
Since $u_0 = \rho(u_1^{-1} u u_1)$ and in view of (\ref{sgs33B'}) the identity among relations that follows from (\ref{ML3.3B2}) is 
	$$w \centerdot u_1^{-1} u u_1 \equiv w \centerdot u_1^{-1} u u_1 \centerdot u_1^{-1} v u_1 \centerdot u_1^{-1} v^{-1} u_1,$$
which is basic. By setting $q := u_0$, $w' := w$, $p := v$ this proves the first part of the lemma. The second part of the lemma follows by noting that the word $u_2^{-1} u_2$ is canceled in $w*d$ and $w*q'$ and that $u_2 \neq 1$.

\medskip

\textbf{3B3)} $u_2$ is not a prefix of $w_1^{-1}$ and $u_1$ is not a suffix of $a$.

\medskip

\begin{tabular}{|c|c|c|}
	\hline
	\rule{0pt}{2.3ex}
	$\,\,\,\, u_2 \,\,\,$ & $\,\,\,\, v \,\,\,$ & $u_1$ \\  
	\hline
\end{tabular}

\begin{tabular}{|c|c|c|}
	\hline 
	\rule{0pt}{2.3ex}
	$w_1^{-1}$ & $s g s^{-1}$ & $a \hspace{0.3mm}$ \\
	\hline
\end{tabular}

\medskip

If $u_2$ is not a prefix of $w_1^{-1}$, then $w_1^{-1}$ is a prefix of $u_2$; if $u_1$ is not a suffix of $a$ then $a$ is a suffix of $u_1$. Therefore there exist words $x, y$ such that $u_2 = w_1^{-1} x$, $s g s^{-1} = x v y$, $u_1 = y a$, thus $u = u_1 u_2 = y a w_1^{-1} x$. The words $u_0 := x y a w_1^{-1}$ and $w' := w_1 a^{-1}$ are cyclic permutations of $u$ and $w$ respectively and we have
	\begin{equation} \label{ML3.3B3} \rho(u_0 w') = \rho(x y a w_1^{-1} w_1 a^{-1}) = \rho(x y) = \rho(x v^{-1} x^{-1} x v y) =  \rho(x v^{-1} x^{-1} s g s^{-1}).\end{equation}
Since $u_0 = \rho(x u x^{-1})$ and $w' = \rho(w_1 w w_1^{-1})$, and in view of (\ref{sgs33B}), the identity among relations that follows from (\ref{ML3.3B3}) is  
	$$x u x^{-1} \centerdot w_1 w w_1^{-1} \equiv x v^{-1} x^{-1} \centerdot (w_1 u_2) v (u_2^{-1} w_1^{-1}) \centerdot (w_1 u_2) u (u_2^{-1} w_1^{-1}) \centerdot w_1 w w_1^{-1}.$$
Since $\rho(w_1 u_2) = x$ that identity is basic. By setting $q := u_0$, $w' := w$, $p := v$ this proves the first part of the lemma.

The second part of the lemma follows by noting that $d = w_1^{-1} s g s^{-1} a$ and $w = a^{-1} w_1$, so the words $a a^{-1}$ and $w_1 w_1^{-1}$ are canceled in $d*w$ by virtue of Proposition \ref{puzo}. Moreover the word $a w_1^{-1} w_1 a^{-1}$ is canceled in $q*w'$ and $a w_1^{-1} \neq 1$ since $a w_1^{-1}$ is a cyclic permutation of $w^{-1}$ and $w \neq 1$.

\section{Proof of Lemma \ref{mainLemma}: fourth case}	\label{prLe4}

Let $u, v, w$ be reduced non-empty words and let $d$ be a cyclic permutation of $u*v$. In this section we suppose that there are cancellations in both $u*v$ and $d*w$.

We split the proof of Lemma \ref{mainLemma} in different subcases depending on the form of $d$. Indeed by Lemma \ref{shirv} three cases are possible:

\begin{enumerate}[1)]
	\item $d = d_1 a$, $w = a^{-1} s (d*w) s^{-1} d_1^{-1}$, $\rho(dw) = d_1 s (d*w) s^{-1} d_1^{-1}$ for words $d_1, a, s$;	
	
	\item $d = t d_1 a$, $w = a^{-1} w_1 t^{-1}$, $\rho(dw) = t d_1 w_1 t^{-1}$, $\rho(wd) = a^{-1} w_1 d_1 a$, for words $t, d_1, a, w_1$ such that $d_1, w_1, ta \neq 1$ and $d*w = d_1 w_1$, $w*d = w_1 d_1$;  
	
	\item $d = w_1^{-1} s (d*w) s^{-1} a$, $w = a^{-1} w_1$, $\rho(dw) = w_1^{-1} s (d*w) s^{-1} w_1$ for words $w_1, s, a$.
\end{enumerate}

By Lemma \ref{shirv4} there exists a pair of words $p$ and $q_0$ such that either $p \sim u$ and $q_0 \sim v$ or $p \sim v$ and $q_0 \sim u$ and such that one of the following two cases holds:  
\begin{enumerate} [(A)]	
	\item $q_0 = p^{-1} r c_1 c_2 r^{-1}$ for words $r, c_1, c_2$ such that $p*q_0 = c_1 c_2$ and $d = c_2 c_1$; moreover $p*q_0 = u*v$; finally if $d = u*v$ then $c_2 = 1$;
		
	\item $p = e_2 b$ and $q_0 = b^{-1} e_3 e_1$ for words $b, e_1, e_2, e_3$ such that $d = e_1 e_2 e_3$ and $b, e_2, e_3 e_1 \neq 1$; finally if $d = u*v$ then $e_1 = 1$.
\end{enumerate}

If 1) holds then
	\begin{equation} \label{4.1} s (d*w) s^{-1} = \rho(d_1^{-1} d w d_1)\end{equation}
and 
	\begin{equation} \label{4.1-0.5} \rho(wd) = \rho(a^{-1} s (d*w) s^{-1} a), \end{equation}
which implies that
	\begin{equation} \label{4.1'} s (d*w) s^{-1} = \rho(a w d a^{-1}).\end{equation}	
Since $w*d = \hat{\rho}(wd)$, there exists a word $\alpha$ such that $\rho(wd) = \alpha (w*d) \alpha^{-1}$, so (\ref{4.1-0.5}) and (\ref{4.1'}) imply that
	\begin{equation} \label{4.1-1.5} \alpha (w*d) \alpha^{-1} = \rho(a^{-1} s (d*w) s^{-1} a) = \rho(w d).\end{equation}

If 2) holds then
	\begin{equation} \label{4.2} d*w = \rho(t^{-1} d w t)\end{equation}
and 
	\begin{equation} \label{4.2'} w*d = \rho(a w d a^{-1}).\end{equation}

If 3) holds then
	\begin{equation} \label{4.3} s (d*w) s^{-1} = \rho(w_1 d w w_1^{-1}).\end{equation}
Also we have that 
	$$\rho(wd) = \rho(a^{-1} w_1 w_1^{-1} s (d*w) s^{-1} a) = \rho(a^{-1} s (d*w) s^{-1} a).$$
Since $w*d = \hat{\rho}(wd)$, there exists a word $\alpha$ such that $\rho(wd) = \alpha (w*d) \alpha^{-1}$, and this implies that 
	\begin{equation} \label{4.4} \alpha (w*d) \alpha^{-1} = \rho(a^{-1} s (d*w) s^{-1} a) = \rho(wd).\end{equation}

Let (A) hold; then $\rho(p q_0) = \rho(p p^{-1} r c_1 c_2 r^{-1}) = \rho(r c_1 c_2 r^{-1})$ and thus $c_1 c_2 = \rho(r^{-1} p q_0 r)$, which implies that
	\begin{equation} \label{4.A} d = \rho(c_1^{-1} r^{-1} p q_0 r c_1) = \rho(c_2 r^{-1} p q_0 r c_2^{-1}).\end{equation}
Also we have that $\rho(q_0 p) = \rho(p^{-1} r c_1 c_2 r^{-1} p)$ and thus $c_1 c_2 = \rho((r^{-1} p) q_0 p (p^{-1} r))$, which implies that
	\begin{equation} \label{4.A'} d = \rho((c_1^{-1} r^{-1} p) q_0 p (p^{-1} r c_1)) = \rho((c_2 r^{-1} p) q_0 p (p^{-1} r c_2^{-1})).\end{equation}
		
Let (B) hold; then $e_2 e_3 e_1 = \rho(p q_0)$, which implies that
	\begin{equation} \label{4.B} d = \rho(e_1 p q_0 e_1^{-1}) = \rho(e_3^{-1} e_2^{-1} p q_0 e_2 e_3).\end{equation}
Also we have that $\rho(q_0 p) = \rho(b^{-1} e_3 e_1 e_2 b)$, thus $e_3 e_1 e_2 = \rho(b q_0 p b^{-1})$, which implies that
	\begin{equation} \label{4.B'} d = \rho(e_3^{-1} b q_0 p b^{-1} e_3) = \rho(e_1 e_2 b q_0 p b^{-1} e_2^{-1} e_1^{-1}).\end{equation}

We will prove Lemma \ref{shirv} in subcases 1A, 1B, 2A, 2B, 3A, 3B.

\bigskip

\textbf{1A)} $d = d_1 a = c_2 c_1$, $w = a^{-1} s (d*w) s^{-1} d_1^{-1}$. If $d = u*v$ then $c_2 = 1$.

\smallskip

By (\ref{4.1}) and (\ref{4.A}) we have that 
	\begin{equation} \label{ML4.1A} s (d*w) s^{-1} = \rho(d_1^{-1} c_1^{-1} r^{-1} p q_0 r c_1 w d_1)\end{equation}
and that
	\begin{equation} \label{ML4.1A'} s (d*w) s^{-1} = \rho(d_1^{-1} c_2 r^{-1} p q_0 r c_2^{-1} w d_1)\end{equation}
	
By (\ref{4.1'}) and (\ref{4.A'}) we have that
	\begin{equation} \label{ML4.1A''} s (d*w) s^{-1} = \rho(a w (c_2 r^{-1} p) q_0 p (p^{-1} r c_2^{-1}) a^{-1}) \end{equation}

\smallskip

The word equation $d_1 a = c_2 c_1$ leads to two different possible solutions.

\bigskip

\textbf{1A1)}  \hspace{0.1cm}
\begin{tabular}{|c|c|}
	\hline
	\rule{0pt}{2.3ex}
	$\,\, c_2 \,\,$ & $c_1$ \\  
	\hline
\end{tabular}

\hspace{11.75mm}
\begin{tabular}{|c|c|}
	\hline 
	\rule{0pt}{2.3ex}
	$d_1$ & $\,\,\, a \,\,\hspace{0.4mm}$ \\
	\hline
\end{tabular}

\medskip

There exists a word $x$ such that $c_2 = d_1 x$ and $a = x c_1$, so $a^{-1} = c_1^{-1} x^{-1}$ and $w = c_1^{-1} x^{-1} s (d*w) s^{-1} d_1^{-1}$, $q_0 = p^{-1} r c_1 d_1 x r^{-1}$, $d = d_1 x c_1$.

Since $d = d_1 x c_1$ and $d \neq 1$ then it is not possible that both $d_1 c_1$ and $x$ be equal to 1.

First suppose that $d_1 c_1\neq 1$. The words $q := x r^{-1} p^{-1} r c_1 d_1$ and $w' := d_1^{-1} c_1^{-1} x^{-1} s (d*w) s^{-1} = \rho(d_1^{-1} w d_1)$ are cyclic permutations of $q_0$ and $w$ respectively and we have
	\begin{multline} \label{ML4.1A1a} \rho(q w') = \rho(x r^{-1} p^{-1} r c_1 d_1 d_1^{-1} c_1^{-1} x^{-1} s (d*w) s^{-1}) = \\ \rho(x r^{-1} p^{-1} r x^{-1} s (d*w) s^{-1}).\end{multline}
By (\ref{ML4.1A'}) and since $\rho(d_1^{-1} c_2) = x$, $d_1^{-1} = \rho(x c_2^{-1})$, $q_0 = \rho(r x^{-1} q x r^{-1})$ we have that 
	$$s (d*w) s^{-1} = \rho(x r^{-1} p r x^{-1} q x r^{-1} r c_2^{-1} w d_1) = \rho(x r^{-1} p r x^{-1} q d_1^{-1} w d_1),$$
so the identity among relations that follows from (\ref{ML4.1A1a}) is	
	$$q \centerdot d_1^{-1} w d_1 \equiv (x r^{-1}) p^{-1} (r x^{-1}) \centerdot (x r^{-1}) p (r x^{-1}) \centerdot q \centerdot d_1^{-1} w d_1,$$
which is basic, proving the first part of the lemma. If $d = u*v$ then $c_2 = 1$, thus $d_1 = 1$ and then $w' = w$.

The second part follows by noting that $d = d_1 x c_1$ and $w = c_1^{-1} x^{-1} s (d*w) s^{-1} d_1^{-1}$, the words $c_1 c_1^{-1}$ and $d_1^{-1} d_1$ are canceled in $d*w$ by Proposition \ref{puzo}, the word $c_1 d_1 d_1^{-1} c_1^{-1}$ is canceled in $q*w'$ and that $d_1 c_1\neq 1$.

\smallskip

Now suppose that $d_1 c_1 = 1$. Then $d = x$ and $q_0 = p^{-1} r d r^{-1}$, $w = d^{-1} s (d*w) s^{-1}$. The word $q := r^{-1} p^{-1} r d$ is a cyclic permutation of $q_0$ and 
	\begin{equation} \label{ML4.1A1b} \rho(qw) = \rho(r^{-1} p^{-1} r d d^{-1} s (d*w) s^{-1}) = \rho(r^{-1} p^{-1} r s (d*w) s^{-1}).\end{equation}
By (\ref{ML4.1A}) and since $q_0 = \rho(r q r^{-1})$ we have that 	
	$$s (d*w) s^{-1} = \rho(r^{-1} p q_0 r w) = \rho(r^{-1} p r q r^{-1} r w) = \rho(r^{-1} p r q w),$$	
so the identity among relations that follows from (\ref{ML4.1A1b}) is	
	$$q \centerdot  w  \equiv r^{-1} p^{-1} r  \centerdot r^{-1} p r  \centerdot q \centerdot w,$$
which is basic. By setting $w' := w$ this proves the first part of the lemma. The second part follows from the fact that the non-empty word $d d^{-1}$ is canceled in both $d*w$ and $q*w'$.

\bigskip

\textbf{1A2)}  \hspace{0.1cm}
\begin{tabular}{|c|c|}
	\hline
	\rule{0pt}{2.3ex}
	$c_2$ & $\, c_1 \,$ \\  
	\hline
\end{tabular}

\hspace{11.75mm}
\begin{tabular}{|c|c|}
	\hline 
	\rule{0pt}{2.3ex}
	$\,\, d_1 \,\hspace{0.4mm}$ & $a$ \\
	\hline
\end{tabular}

\medskip

There exists a word $x$ such that $d_1 = c_2 x$ and $c_1 = x a$, so $d_1^{-1} = x^{-1} c_2^{-1}$ and then $w = a^{-1} s (d*w) s^{-1} x^{-1} c_2^{-1}$, $q_0 = p^{-1} r x a c_2 r^{-1}$, $d = c_2 x a$.

Since $d = c_2 x a$ and $d \neq 1$ then it is not possible that both $c_2 a$ and $x$ be equal to 1.

First suppose that $c_2 a \neq 1$. The words $q := r^{-1} p^{-1} r x a c_2$ and $w' := c_2^{-1} a^{-1} s (d*w) s^{-1} x^{-1} = \rho(c_2^{-1} w c_2)$ are cyclic permutations of $q_0$ and $w$ respectively and we have
	\begin{multline} \label{ML4.1A2a}  \rho(qw') = \rho(r^{-1} p^{-1} r x a c_2 c_2^{-1} a^{-1} s (d*w) s^{-1} x^{-1}) = \\ 
	\rho(r^{-1} p^{-1} r x s (d*w) s^{-1} x^{-1}).\end{multline}
By (\ref{ML4.1A'}) and since $\rho(x d_1^{-1} c_2) = 1$, $q_0 = \rho(r q r^{-1})$, $\rho(d_1 x^{-1}) = c_2$ we have that 
	$$\rho(x s (d*w) s^{-1} x^{-1}) = \rho(x d_1^{-1} c_2 r^{-1} p q_0 r c_2^{-1} w d_1 x^{-1}) =$$
	$$\rho(r^{-1} p r q r^{-1} r c_2^{-1} w c_2) = \rho(r^{-1} p r q c_2^{-1} w c_2),$$	
so the identity among relations that follows from (\ref{ML4.1A2a}) is	
	$$q \centerdot c_2^{-1} w c_2 \equiv r^{-1} p^{-1} r \centerdot r^{-1} p r \centerdot q \centerdot c_2^{-1} w c_2,$$
which is basic, proving the first part of the lemma. If $d = u*v$ then $c_2 = 1$, thus $w' = w$.

The second part of the lemma follows by noting that $d = c_2 x a$ and $w = a^{-1} s (d*w) s^{-1} x^{-1} c_2^{-1}$, the words $a a^{-1}$ and $c_2^{-1} c_2$ are canceled in $d*w$ by virtue of Proposition \ref{puzo}, the word $a c_2 c_2^{-1} a^{-1}$ is canceled in $q*w'$ and that $a c_2 \neq 1$.

\smallskip
	
Now let $c_2 a = 1$; then $x \neq 1$, $d = d_1 = x$ , $q_0 = p^{-1} r d r^{-1}$ and $w = s (d*w) s^{-1} d^{-1}$. 

The word $q := d r^{-1} p^{-1} r$ is a cyclic permutation of $q_0$ and we have that
	$$\rho(wq) = \rho(s (d*w) s^{-1} d^{-1} d r^{-1} p^{-1} r) = \rho(s (d*w) s^{-1} r^{-1} p^{-1} r).$$
	
By (\ref{4.1'}) we have that $\rho(wd) = s (d*w) s^{-1}$, therefore $w*d = d*w$, so the last equality implies	
	\begin{equation} \label{ML4.1A2b} \rho(wq) = \rho(s (w*d) s^{-1} r^{-1} p^{-1} r).\end{equation}
	
By (\ref{ML4.1A''}) and since $q_0 = \rho((p^{-1} r) q (r^{-1} p))$we have that
		$$s (d*w) s^{-1} = \rho(w (r^{-1} p) q_0 p (p^{-1} r)) =$$
		$$\rho(w (r^{-1} p p^{-1} r) q (r^{-1} p) p (p^{-1} r)) = \rho(w q r^{-1} p r)$$
so the identity among relations that follows from (\ref{ML4.1A2b}) is	
	$$w \centerdot q  \equiv w \centerdot q \centerdot r^{-1} p r \centerdot r^{-1} p^{-1} r,$$
which is basic. By setting $w' := w$ this proves the first part of the lemma. The second part of the lemma follows from the fact that the non-empty word $d^{-1} d$ is canceled in both $w*$ and $w*q$. 

\medskip	

\bigskip

\textbf{1B)} $d = d_1 a = e_1 e_2 e_3$, $w = a^{-1} s (d*w) s^{-1} d_1^{-1}$, $\rho(dw) = d_1 s (d*w) s^{-1} d_1^{-1}$, $p = e_2 b$, $q_0 = b^{-1} e_3 e_1$, with $b, e_2, e_3 e_1 \neq 1$. If $d = u*v$ then $e_1 = 1$.

\smallskip

By (\ref{4.1}) and (\ref{4.B}) we have that
	\begin{equation} \label{ML4.1B} s (d*w) s^{-1} = \rho(d_1^{-1} e_1 p q_0 e_1^{-1} w d_1)\end{equation}
and that
	\begin{equation} \label{ML4.1B'} s (d*w) s^{-1} = \rho(d_1^{-1} e_3^{-1} e_2^{-1} p q_0 e_2 e_3 w d_1)\end{equation}
	
By (\ref{4.1}) and (\ref{4.B'}) we have that
	\begin{equation} \label{ML4.1B''} s (d*w) s^{-1} = \rho(d_1^{-1} e_1 e_2 b q_0 p b^{-1} e_2^{-1} e_1^{-1} w d_1).\end{equation}

By (\ref{4.1-1.5}) and (\ref{4.B}) we have that
	\begin{equation} \label{ML4.1B'''} \alpha (w*d) \alpha^{-1} = \rho(w e_1 p q_0 e_1^{-1}).\end{equation}


\smallskip

The word equation $d_1 a = e_1 e_2 e_3$ leads to three different possible solutions.

\bigskip

\textbf{1B1)}  \hspace{0.1cm}
\begin{tabular}{|c|c|c|}
	\hline
	\rule{0pt}{2.3ex}
	$e_1$ & $e_2$ & $e_3$ \\  
	\hline
\end{tabular}

\hspace{11.6mm}
\begin{tabular}{|c|c|}
	\hline 
	\rule{0pt}{2.3ex}
	$\,\,\,\,\,\,\, d_1 \,\,\,\,\,\,\hspace{0.6mm}$ & $a$ \\
	\hline
\end{tabular}

\medskip

There exists a word $x$ such that $d_1 = e_1 e_2 x$ and $e_3 = x a$. Thus $d_1^{-1} = x^{-1} e_2^{-1} e_1^{-1}$ and $w = a^{-1} s (d*w) s^{-1} x^{-1} e_2^{-1} e_1^{-1}$, $q_0 = b^{-1} x a e_1$. 

The words $w' := e_1^{-1} a^{-1} s (d*w) s^{-1} x^{-1} e_2^{-1} = \rho(e_1^{-1} w e_1)$ and $q := e_1 b^{-1} x a$ are cyclic permutations of $w$ and $q_0$ respectively and $q^{-1} = a^{-1} x^{-1} b e_1^{-1}$.

We have that
		$$\rho(w'p) = \rho(e_1^{-1} a^{-1} s (d*w) s^{-1} x^{-1} e_2^{-1} e_2 b) = \rho(e_1^{-1} a^{-1} s (d*w) s^{-1} x^{-1} b),$$
thus by (\ref{4.1-1.5}) we have that
	\begin{multline} \label{ML4.1B1} \rho(e_1 w'p e_1^{-1}) = \rho(a^{-1} s (d*w) s^{-1} x^{-1} b e_1^{-1}) = \\ \rho(a^{-1} s (d*w) s^{-1} a a^{-1} x^{-1} b e_1^{-1}) = \rho(\alpha (w*d) \alpha^{-1} q^{-1}) .\end{multline}
Since $q_0 = \rho(e_1^{-1} q e_1)$, by (\ref{ML4.1B'''}) we have that 
		$$\alpha (w*d) \alpha^{-1} = \rho(w e_1 p q_0 e_1^{-1}) = \rho(w e_1 p e_1^{-1} q e_1 e_1^{-1}) = \rho(w e_1 p e_1^{-1} q).$$
Since $\rho(e_1 w' e_1^{-1}) = w$, the identity among relations that follows from (\ref{ML4.1B1}) is	
	$$w \centerdot e_1 p e_1^{-1} \equiv w \centerdot e_1 p e_1^{-1} \centerdot q \centerdot q^{-1},$$
which is basic. By setting swapping $p$ and $q$ this proves the first part of the lemma. If $d = u*w$ then $e_1 = 1$, which implies that $w' = w$.

To prove the second part we note that the non-empty word $e_2^{-1} e_2$ is canceled in $w'*q$ and in $w*d$.

\bigskip

\textbf{1B2)} \hspace{0.1cm}
\begin{tabular}{|c|c|c|}
	\hline
	\rule{0pt}{2.3ex}
	$e_1$ & $e_2$ & $e_3$ \\  
	\hline
\end{tabular}

\hspace{11.6mm}
\begin{tabular}{|c|c|}
	\hline 
	\rule{0pt}{2.3ex}
	$\,\,\, d_1 \,\,\hspace{0.6mm}$ & $\,\,\,\, a \,\,\,\,$ \\
	\hline
\end{tabular}

\medskip

There exists a word $x$ such that $d_1 = e_1 x$, $e_2 = xy$, $a = y e_3$. Thus $d_1^{-1} = x^{-1} e_1^{-1}$, $a^{-1} = e_3^{-1} y^{-1}$ and then $w = e_3^{-1} y^{-1} s (d*w) s^{-1} x^{-1} e_1^{-1}$, $p = xyb$.

The words $w' := e_1^{-1} e_3^{-1} y^{-1} s (d*w) s^{-1} x^{-1} = \rho(e_1^{-1} w e_1)$ and $p_0 := ybx = \rho(x^{-1} p x)$ are cyclic permutations of $w$ and $p$ respectively. Moreover we have $p_0^{-1} = x^{-1} b^{-1} y^{-1}$ and setting $q := q_0$ we have
	$$\rho(qw') = \rho(b^{-1} e_3 e_1 e_1^{-1} e_3^{-1} y^{-1} s (d*w) s^{-1} x^{-1}) =$$
	$$\rho(qw') = \rho(b^{-1} y^{-1} s (d*w) s^{-1} x^{-1}),$$
so
	\begin{multline} \label{ML4.1B2} \rho(x^{-1} qw' x) = \rho(x^{-1} b^{-1} y^{-1} s (d*w) s^{-1}) = \\ \rho(p_0^{-1} s (d*w) s^{-1}) = \rho(x^{-1} p^{-1} x s (d*w) s^{-1}). \end{multline}
By (\ref{ML4.1B}) we have that 	
	$$s (d*w) s^{-1} = \rho(d_1^{-1} e_1 p q_0 e_1^{-1} w d_1) = \rho(x^{-1} e_1^{-1} e_1 p q e_1^{-1} w e_1 x) =$$
	$$\rho(x^{-1} p q e_1^{-1} w e_1 x) = \rho(x^{-1} p q x (x^{-1} e_1^{-1}) w (e_1 x))$$	
and since $\rho(x^{-1} w' x) = \rho((x^{-1} e_1^{-1}) w (e_1 x))$ the identity among relations that follows from (\ref{ML4.1B2}) is	
	$$x^{-1} q x \centerdot (x^{-1} e_1^{-1}) w (e_1 x) \equiv x^{-1} p^{-1} x \centerdot x^{-1} p x \centerdot x^{-1} q x \centerdot (x^{-1} e_1^{-1}) w (e_1 x),$$
which is basic, proving the first part of the lemma. If $d = u*w$ then $e_1 = 1$, which implies that $w' = w$.

The second part follows by noting that $d = e_1 x y e_3$ and $w = e_3^{-1} y^{-1} s (d*w) s^{-1} x^{-1} e_1^{-1}$, the words $e_3 e_3^{-1}$ and $e_1^{-1} e_1$ are canceled in $d*w$ by virtue of Proposition \ref{puzo}, the word $e_3 e_1 e_1^{-1} e_3^{-1}$ is canceled in $q*w'$ and that $e_3 e_1 \neq 1$.


\bigskip

\textbf{1B3)}  \hspace{0.1cm}
\begin{tabular}{|c|c|c|}
	\hline
	\rule{0pt}{2.3ex}
	$\,\, e_1 \,\,$ & $e_2$ & $e_3$ \\  
	\hline
\end{tabular}

\hspace{11.6mm}
\begin{tabular}{|c|c|}
	\hline 
	\rule{0pt}{2.3ex}
	$d_1$ & $\,\,\, \,\,\,\,\,\, a \,\,\,\,\,\, \,\,\hspace{0.6mm}$ \\
	\hline
\end{tabular}

\medskip

There exists a word $x$ such that $e_1 = d_1 x$ and $a = x e_2 e_3$, thus $a^{-1} = e_3^{-1} e_2^{-1}  x^{-1}$ and $w = e_3^{-1} e_2^{-1} x^{-1} s (d*w) s^{-1} d_1^{-1}$, $q_0 = b^{-1} e_3 d_1 x$.

Since $e_3 d_1 x = e_3 e_1 \neq 1$, then it is not possible that both $e_3 d_1$ and $x$ be equal to 1.

First let us assume that $e_3 d_1 \neq 1$. The words $q := x b^{-1} e_3 d_1$ and $w' := d_1^{-1} e_3^{-1} e_2^{-1} x^{-1} s (d*w) s^{-1} = \rho(d_1^{-1} w d_1)$ are cyclic permutations of $q_0$ and $w$ respectively. We have that 
	\begin{multline} \label{ML4.1B3} \rho(q w') = \rho(x b^{-1} e_3 d_1 d_1^{-1} e_3^{-1} e_2^{-1} x^{-1} s (d*w) s^{-1}) = \\
	\rho(x b^{-1} e_2^{-1} x^{-1} s (d*w) s^{-1}) = \rho(x p^{-1} x^{-1} s (d*w) s^{-1}).\end{multline}
Since $\rho(d_1^{-1} e_1) = x$, $\rho(x e_1^{-1}) = d_1^{-1}$ and $q_0 = \rho(x^{-1} q x)$, by (\ref{ML4.1B}) we have that 
	$$s (d*w) s^{-1} = \rho(d_1^{-1} e_1 p q_0 e_1^{-1} w d_1) =$$ 
	$$\rho(x p x^{-1} q x e_1^{-1} w d_1)  = \rho(x p x^{-1} q d_1^{-1} w d_1)$$
which implies that the identity among relations that follows from (\ref{ML4.1B3}) is	
	$$q \centerdot d_1^{-1} w d_1 \equiv x^{-1} p^{-1} x \centerdot x^{-1} p x \centerdot q \centerdot d_1^{-1} w d_1,$$
which is basic, proving the first part of the lemma. If $d = u*w$ then $e_1 = 1$ thus $d_1 = 1$, which implies that $w' = w$.

For the second part we note that the non-empty word $e_3 d_1 d_1^{-1} e_3^{-1}$ is canceled in $q*w'$ and $e_3 d_1 \neq 1$. Since $d = d_1 x e_2 e_3$ and $w = e_3^{-1} e_2^{-1} x^{-1} s (d*w) s^{-1} d_1^{-1}$, then $e_3 e_3^{-1}$ and $d_1^{-1} d_1$ are canceled in $d*w$ by Proposition \ref{puzo}.

\smallskip

Now let us assume that $e_3 d_1 = 1$, that is $x \neq 1$. We have that $w = e_2^{-1} x^{-1} s (d*w) s^{-1}$, $q_0 = b^{-1} x$. The word $w' := x^{-1} s (d*w) s^{-1} e_2^{-1} = \rho(e_2 w e_2^{-1})$ is a cyclic permutation of $w$ and by setting $q := q_0$ we have that 
	\begin{multline} \label{ML4.1B3'} \rho(q w') = \rho(b^{-1} x x^{-1} s (d*w) s^{-1} e_2^{-1}) = \rho(b^{-1} s (d*w) s^{-1} e_2^{-1}) = \\
	\rho(b^{-1} e_2^{-1} e_2 s (d*w) s^{-1} e_2^{-1}) = \rho(p^{-1} e_2 s (d*w) s^{-1} e_2^{-1}).\end{multline}
By (\ref{ML4.1B'}) we have that
	$$\rho(e_2 s (d*w) s^{-1} e_2^{-1}) = \rho(e_2 e_2^{-1} p q_0 e_2 w e_2^{-1}) = \rho(p q e_2 w e_2^{-1}),$$
so the identity among relations that follows from (\ref{ML4.1B3'}) is
	$$q \centerdot e_2 w e_2^{-1} \equiv p^{-1} \centerdot p \centerdot q \centerdot e_2 w e_2^{-1},$$
which is basic, proving the first part of the lemma. We observe that $d \neq u*w$, because if $d = u*w$ then $e_1 = 1$ and since $e_3 = 1$ then we would have that $e_3 e_1 = 1$, which is impossible.

For the second part we note that the non-empty word $x x^{-1}$ is canceled in $q*w'$ and it is also canceled in $d*w$ by Proposition \ref{puzo} since $d = x e_2$ and $w = e_2^{-1} x^{-1} s (d*w) s^{-1}$.



\medskip

\bigskip

\textbf{2A)} $d = t d_1 a = c_2 c_1$, $w = a^{-1} w_1 t^{-1}$, $\rho(dw) = t d_1 w_1 t^{-1}$, $d*w = d_1 w_1$, $w*d = w_1 d_1$, $q_0 = p^{-1} r c_1 c_2 r^{-1}$, with $d_1, w_1, ta \neq 1$. If $d = u*v$ then $c_2 = 1$.

By (\ref{4.2}) and (\ref{4.A}) we have that 
	\begin{equation} \label{ML4.2A} d*w = \rho(t^{-1} c_1^{-1} r^{-1} p q_0 r c_1 w t)\end{equation}
and that
	\begin{equation} \label{ML4.2A'} d*w = \rho(t^{-1} c_2 r^{-1} p q_0 r c_2^{-1} w t).\end{equation}

By (\ref{4.2'}) and (\ref{4.A'}) we have that
	\begin{equation} \label{ML4.2A''} w*d = \rho\big(a w (c_2 r^{-1} p) q_0 p (p^{-1} r c_2^{-1}) a^{-1}\big)\end{equation}

\smallskip

The word equation $c_2 c_1 = t d_1 a$ leads to three different possible solutions.

\bigskip

\textbf{2A1)} \hspace{0.1cm}
\begin{tabular}{|c|c|}
	\hline
	\rule{0pt}{2.3ex}
	$\,\,\,\,\,\, c_2 \,\,\,\,\,\,$ & $c_1$ \\  
	\hline
\end{tabular}

\hspace{11.8mm}
\begin{tabular}{|c|c|c|}
	\hline 
	\rule{0pt}{2.3ex}
	$t$ & $d_1$ & $\,\, a \,\,\hspace{0.5mm}$ \\
	\hline
\end{tabular}

\medskip

There exists a word $x$ such that $c_2 = t d_1 x$, $a = x c_1$, so $a^{-1} = c_1^{-1} x^{-1}$ and $q_0 = p^{-1} r c_1 t d_1 x r^{-1}$, $w = c_1^{-1} x^{-1} w_1 t^{-1}$. We have that $ta = t x c_1 \neq 1$, so if $c_1 t = 1$ then $x \neq 1$.

First suppose that $c_1 t \neq 1$. The words $q := d_1 x r^{-1} p^{-1} r c_1 t$ and $w' := t^{-1} c_1^{-1} x^{-1} w_1 = \rho(t^{-1} w t)$ are cyclic permutations of $q_0$ and $w$ respectively and we have that
	$$\rho(qw') = \rho(d_1 x r^{-1} p^{-1} r c_1 t t^{-1} c_1^{-1} x^{-1} w_1) = \rho(d_1 x r^{-1} p^{-1} r x^{-1} w_1),$$
therefore since $w_1 d_1 = \rho(d_1^{-1} (d*w) d_1)$ we have that 
	\begin{multline} \label{ML4.2A1}\rho(d_1^{-1} qw' d_1) = \rho(x r^{-1} p^{-1} r x^{-1} w_1 d_1) = \\ \rho(x r^{-1} p^{-1} r x^{-1} d_1^{-1} (d*w) d_1).\end{multline}
Since $\rho(d_1^{-1} t^{-1} c_2) = x$, $\rho(x c_2^{-1}) = \rho(d_1^{-1} t^{-1})$ and $q_0 = \rho((r x^{-1} d_1^{-1}) q (d_1 x r^{-1}))$, by (\ref{ML4.2A'}) we have that
	$$\rho(d_1^{-1} (d*w) d_1) = \rho(d_1^{-1} t^{-1} c_2 r^{-1} p q_0 r c_2^{-1} w t d_1) =$$
	$$\rho(x r^{-1} p r x^{-1} d_1^{-1} q d_1 x c_2^{-1} w t d_1) = \rho((x r^{-1}) p (r x^{-1}) d_1^{-1} q d_1 (d_1^{-1} t^{-1}) w (t d_1)),$$
so the identity among relations that follows from (\ref{ML4.2A1}) is
	$$d_1^{-1} q d_1 \centerdot (d_1^{-1} t^{-1}) w (t d_1) \equiv$$
	$$(x r^{-1}) p^{-1} (r x^{-1}) \centerdot (x r^{-1}) p (r x^{-1}) \centerdot d_1^{-1} q d_1 \centerdot (d_1^{-1} t^{-1}) w (t d_1)),$$
which is basic, proving the first part of the lemma. If $d = u*v$ then $c_2 = 1$, thus $t = 1$ which implies that $w' = w$.

For the second part we note that $d = t d_1 x c_1$ and $w = c_1^{-1} x^{-1} w_1 t^{-1}$, the words $c_1 c_1^{-1}$ and $t^{-1} t$ are canceled in $d*w$ by virtue of Proposition \ref{puzo}, the word $c_1 t t^{-1} c_1^{-1}$ is canceled in $q*w'$ and $c_1 t \neq 1$.	
	
\smallskip

Now suppose that $c_1, t = 1$; then $q_0 = p^{-1} r d_1 x r^{-1}$ and $w = x^{-1} w_1$. The word $q := r^{-1} p^{-1} r d_1 x$ is a cyclic permutation of $q_0$ and we have
	\begin{multline} \label{ML4.2A1'} \rho(qw) = \rho(r^{-1} p^{-1} r d_1 x x^{-1} w_1) = \\ \rho(r^{-1} p^{-1} r d_1 w_1) = \rho(r^{-1} p^{-1} r (d*w)).\end{multline}
Since $q_0 = \rho(r q r^{-1})$, by (\ref{ML4.2A}) we have that
	$$d*w = \rho(r^{-1} p q_0 r w) = \rho(r^{-1} p r q r^{-1} r w) = \rho(r^{-1} p r q w)$$
so the identity among relations that follows from (\ref{ML4.2A1'}) is
	$$q \centerdot w \equiv r^{-1} p^{-1} r \centerdot r^{-1} p r \centerdot q \centerdot w,$$	
which is basic. By setting $w' := w$ this proves the first part of the lemma. The second part follows from the fact that $d = d_1 x$, so the non-empty word $x x^{-1}$ is canceled in both $d*w$ and $q*w'$.

\bigskip

\textbf{2A2)} \hspace{0.1cm}
\begin{tabular}{|c|c|}
	\hline
	\rule{0pt}{2.3ex}
	$\,\, c_2 \,\,$ & $\,\, c_1 \,\,$ \\  
	\hline
\end{tabular}

\hspace{11.8mm}
\begin{tabular}{|c|c|c|}
	\hline 
	\rule{0pt}{2.3ex}
	$t$ & $d_1$ & $a \hspace{0.5mm}$ \\
	\hline
\end{tabular}

\medskip

There exist words $d_2, d_3$ such that $c_2 = t d_2$, $d_1 = d_2 d_3$, $c_1 = d_3 a$, so $q_0 = p^{-1} r d_3 a t d_2 r^{-1}$, $d*w = d_2 d_3 w_1$.

The words $q := d_2 r^{-1} p^{-1} r d_3 a t$ and $w' := t^{-1} a^{-1} w_1 = \rho(t^{-1} w t)$ are cyclic permutations of $q_0$ and $w$ respectively and
	$$\rho(qw') = \rho(d_2 r^{-1} p^{-1} r d_3 a t t^{-1} a^{-1} w_1) = \rho(d_2 r^{-1} p^{-1} r d_3 w_1),$$
so since $d_3 w_1 d_2 = \rho(d_2^{-1} (d*w) d_2)$ we have that
	\begin{equation} \label{ML4.2A2} \rho(d_2^{-1} q w' d_2) = \rho(r^{-1} p^{-1} r d_3 w_1 d_2) = \rho(r^{-1} p^{-1} r d_2^{-1} (d*w) d_2).\end{equation}
Since $\rho(d_2^{-1} t^{-1} c_2) = 1$, $c_2 = t d_2$ and $q_0 = \rho(r d_2^{-1} q d_2 r^{-1})$, by (\ref{ML4.2A'}) we have that 
	$$\rho(d_2^{-1} (d*w) d_2) = \rho(d_2^{-1} t^{-1} c_2 r^{-1} p q_0 r c_2^{-1} w t d_2) =$$
	$$\rho(r^{-1} p r d_2^{-1} q d_2 r^{-1} r c_2^{-1} w c_2) = \rho(r^{-1} p r d_2^{-1} q d_2 c_2^{-1} w c_2),$$
so the identity among relations that follows from (\ref{ML4.2A2}) is
	$$d_2^{-1} q d_2 \centerdot (d_2^{-1} t^{-1}) w (t d_2) \equiv r^{-1} p^{-1} r  \centerdot r^{-1} p r \centerdot d_2^{-1} q d_2 \centerdot c_2^{-1} w c_2,$$	
which is basic because $c_2 = t d_2$, proving the first part of the lemma. If $d = u*v$ then $c_2 = 1$, thus $t = 1$ which implies that $w' = w$.

The second part follows by noting that since $d = t d_1 a$ and $w = a^{-1} w_1 t^{-1}$, the words $a a^{-1}$ and $t^{-1} t$ are canceled in $d*w$ by virtue of Proposition \ref{puzo}, the word $a t t^{-1} a^{-1}$ is canceled in $q*w'$ and $ta \neq 1$.

\bigskip

\medskip

\textbf{2A3)} \hspace{0.1cm}
\begin{tabular}{|c|c|}
	\hline
	\rule{0pt}{2.3ex}
	$c_2$ & $\,\,\,\,\,\,\, c_1 \,\,\,\,\,\,\,$ \\  
	\hline
\end{tabular}

\hspace{11.8mm}
\begin{tabular}{|c|c|c|}
	\hline 
	\rule{0pt}{2.3ex}
	$\,\,\, t \,\,\,$ & $d_1$ & $a \hspace{0.5mm}$ \\
	\hline
\end{tabular}

\medskip

There exists a word $x$ such that $t = c_2 x$, $c_1 = x d_1 a$, so $t^{-1} = x^{-1} c_2^{-1}$ and $q_0 = p^{-1} r x d_1 a c_2 r^{-1}$, $w = a^{-1} w_1 x^{-1} c_2^{-1}$.

We have that $ta = c_2 x a \neq 1$, so if $a c_2 = 1$ then $x \neq 1$.

First suppose that $a c_2 \neq 1$. The words $q := r^{-1} p^{-1} r x d_1 a c_2$ and $w' := c_2^{-1} a^{-1} w_1 x^{-1} = \rho(c_2^{-1} w c_2)$ are cyclic permutations of $q_0$ and $w$ respectively and
	\begin{multline} \label{ML4.2A3}\rho(qw') = \rho(r^{-1} p^{-1} r x d_1 a c_2 c_2^{-1} a^{-1} w_1 x^{-1}) = \\ \rho(r^{-1} p^{-1} r x d_1 w_1 x^{-1}) = \rho(r^{-1} p^{-1} r x (d*w) x^{-1}).\end{multline}
Since $t^{-1} = x^{-1} c_2^{-1}$, $\rho(x t^{-1} c_2) = 1$ and $q_0 = \rho(r q r^{-1})$, by (\ref{ML4.2A'}) we have that
	$$\rho(x (d*w) x^{-1}) = \rho(x t^{-1} c_2 r^{-1} p q_0 r c_2^{-1} w t x^{-1}) =
	\rho(r^{-1} p r q c_2^{-1} w c_2),$$	
so the identity among relations that follows from (\ref{ML4.2A3}) is
	$$q  \centerdot c_2^{-1} w c_2 \equiv r^{-1} p^{-1} r \centerdot r^{-1} p r \centerdot q  \centerdot c_2^{-1} w c_2,$$	
which is basic, proving the first part of the lemma. If $d = u*v$ then $c_2 = 1$, which implies that $w' = w$.

The second part follows by noting that $d = c_2 x d_1 a$ and $w = a^{-1} w_1 x^{-1} c_2^{-1}$, so the words $a a^{-1}$ and $c_2^{-1} c_2$ are canceled in $d*w$ by virtue of Proposition \ref{puzo}, the word $a c_2 c_2^{-1} a^{-1}$ is canceled in $q*w'$ and $a c_2 \neq 1$.

\smallskip

Now suppose that $a c_2 = 1$; then $q_0 = p^{-1} r x d_1 r^{-1}$ and $w = w_1 x^{-1}$. The word $q := x d_1 r^{-1} p^{-1} r$ is a cyclic permutations of $q_0$ and we have that
	\begin{multline} \label{ML4.2A3'} \rho(w q) = \rho(w_1 x^{-1} x d_1 r^{-1} p^{-1} r) = \\ \rho(w_1 d_1 r^{-1} p^{-1} r) = \rho((w*d) \, r^{-1} p^{-1} r).\end{multline}
By (\ref{ML4.2A''}) we have that
	$$w*d = \rho\big(w (r^{-1} p) q_0 p (p^{-1} r)\big)$$
and since $q_0 = \rho((p^{-1} r) q (r^{-1} p))$ then 
	$$w*d = \rho\big(w (r^{-1} p) (p^{-1} r) q (r^{-1} p) p (p^{-1} r)\big) =$$
	$$\rho\big(w q (r^{-1} p) p (p^{-1} r)\big) = \rho(w q r^{-1} p r),$$
so the identity among relations that follows from (\ref{ML4.2A3'}) is
	$$w  \centerdot q \equiv w \centerdot q \centerdot r^{-1} p r  \centerdot r^{-1} p^{-1} r,$$	
which is basic. By setting $w' := w$ this proves the first part of the lemma. The second part follows from the fact that $w = w_1 x^{-1}$ and $d = x d_1$, so the words $x^{-1} x$ is canceled in $w*d$ and in $w'*q$.

\bigskip

\textbf{2B)} $d = t d_1 a = e_1 e_2 e_3$, $w = a^{-1} w_1 t^{-1}$, $d*w = d_1 w_1$,  $w*d = w_1 d_1$, $p = e_2 b$, $q_0 = b^{-1} e_3 e_1$, with $d_1, w_1, ta, b, e_2, e_3 e_1\neq 1$. If $d = u*v$ then $e_1 = 1$.

\smallskip

From (\ref{4.2}) and (\ref{4.B}) we have that
	\begin{equation} \label{ML4.2B} d*w = \rho(t^{-1} e_1 p q_0 e_1^{-1} w t).\end{equation}
From (\ref{4.2}) and (\ref{4.B'}) we have also that
	\begin{equation} \label{ML4.2B'} d*w = \rho(t^{-1} e_3^{-1} b q_0 p b^{-1} e_3 w t) \end{equation}
and that
	\begin{equation} \label{ML4.2B''} d*w = \rho(t^{-1} e_1 e_2 b q_0 p b^{-1} e_2^{-1} e_1^{-1} w t). \end{equation}
From (\ref{4.2'}) and (\ref{4.B}) we have that 
	\begin{equation} \label{ML4.2Bv} w*d = \rho(a w e_1 p q_0 e_1^{-1} a^{-1}) \end{equation}
and that
	\begin{equation} \label{ML4.2B'''} w*d = \rho(a w e_3^{-1} e_2^{-1} p q_0 e_2 e_3 a^{-1}).\end{equation}

The word equation $e_1 e_2 e_3 = t d_1 a$ leads to six different possible solutions.

\bigskip

\textbf{2B1)}  \hspace{0.1cm}
\begin{tabular}{|c|c|c|}
	\hline
	\rule{0pt}{2.3ex}
	$\,\,\,\,\,\, e_1 \,\,\,\,\,\,$ & $e_2$ & $e_3$ \\  
	\hline
\end{tabular}

\hspace{11.6mm}
\begin{tabular}{|c|c|c|}
	\hline 
	\rule{0pt}{2.3ex}
	$t$ & $d_1$ & $\,\,\,\,\,\,\,\, a \,\,\,\,\,\,\,\,\, \hspace{0.1mm}$ \\
	\hline
\end{tabular}

\medskip

There exists a word $x$ such that $e_1 = t d_1 x$, $a = x e_2 e_3$, so $a^{-1} = e_3^{-1} e_2^{-1} x^{-1}$ and $w = e_3^{-1} e_2^{-1} x^{-1} w_1 t^{-1}$, $q_0 = b^{-1} e_3 t d_1 x$, therefore $q_0^{-1} = x^{-1} d_1^{-1} t^{-1} e_3^{-1} b$. 

The word  $w' := x^{-1} w_1 t^{-1} e_3^{-1} e_2^{-1} = \rho(e_2 e_3 w e_3^{-1} e_2^{-1})$ is a cyclic permutation of $w$ and we have that
	\begin{multline} \label{ML4.2B1} \rho(w' p) = \rho(x^{-1} w_1 t^{-1} e_3^{-1} e_2^{-1} e_2 b) = 
	\rho(x^{-1} w_1 t^{-1} e_3^{-1} b) = \\
	\rho(x^{-1} w_1 d_1 x x^{-1} d_1^{-1} t^{-1} e_3^{-1} b) = \rho(x^{-1} (w*d) x \, q_0^{-1}).\end{multline}
By (\ref{ML4.2B'''}) we have that
	$$\rho(x^{-1} (w*d) x) = \rho(x^{-1} a w e_3^{-1} e_2^{-1} p q_0 e_2 e_3 a^{-1} x) = $$
	$$\rho(x^{-1} x e_2 e_3 w e_3^{-1} e_2^{-1} p q_0 e_2 e_3 e_3^{-1} e_2^{-1} x^{-1} x) = \rho(e_2 e_3 w e_3^{-1} e_2^{-1} p q_0),$$
so the identity among relations that follows from (\ref{ML4.2B1}) is
	$$(e_2 e_3) w (e_3^{-1} e_2^{-1})  \centerdot p \equiv (e_2 e_3) w (e_3^{-1} e_2^{-1})  \centerdot p \centerdot q_0 \centerdot q_0^{-1},$$	
which is basic. By setting $q := p$ and $p := q_0$ this proves the first part of the lemma.	

For the second part we observe that the non-empty word $e_2^{-1} e_2$ is canceled in $w'*q$ and that since $w = e_3^{-1} e_2^{-1} x^{-1} w_1 t^{-1}$ and $d = t d_1 x e_2 e_3$, the word  $e_2 e_2^{-1}$ is canceled in $w*d$ by virtue of Proposition \ref{puzo}. 
	
\smallskip

Now let $d = u*v$; then $e_1 = 1$, thus $t, d_1, x = 1$. This implies that $e_3 \neq 1$ and that $w = e_3^{-1} e_2^{-1} w_1$, $q_0 = b^{-1} e_3$, $d*w = w_1$, therefore
	\begin{multline} \label{ML4.2B1'} \rho(q_0 w) = \rho(b^{-1} e_3 e_3^{-1} e_2^{-1} w_1) = \rho(b^{-1} e_2^{-1} w_1) = \rho(p^{-1}(d*w)).\end{multline}
By (\ref{ML4.2B}) we have that $d*w = \rho(p q_0 w )$, so the identity among relations that follows from (\ref{ML4.2B1'}) is
	$$q_0  \centerdot w \equiv p^{-1} \centerdot p \centerdot q_0 \centerdot w,$$	
which is basic. By setting $w' := w$ this proves the first part of the lemma.	

For the second part we observe that the non-empty word $e_3 e_3^{-1}$ is canceled in $d*w$ and in $q*w'$.

\bigskip

\textbf{2B2)}  \hspace{0.1cm}
\begin{tabular}{|c|c|c|}
	\hline
	\rule{0pt}{2.3ex}
	$\,\, e_1 \,\,$ & $e_2$ & $e_3$ \\  
	\hline
\end{tabular}

\hspace{11.6mm}
\begin{tabular}{|c|c|c|}
	\hline 
	\rule{0pt}{2.3ex}
	$t$ & $d_1$ & $\,\,\,\,\, a \,\,\,\,\hspace{0.1mm}$ \\
	\hline
\end{tabular}

\medskip

There exist words $d_2, d_3, a_1$ such that $e_1 = t d_2$, $d_1 = d_2 d_3$, $e_2 = d_3 a_1$, $a = a_1 e_3$. Thus $a^{-1} = e_3^{-1} a_1^{-1}$ and $p = d_3 a_1 b$, $q_0 = b^{-1} e_3 t d_2$, $w = e_3^{-1} a_1^{-1} w_1 t^{-1}$, $p^{-1} = b^{-1} a_1^{-1} d_3^{-1}$, $d*w = d_2 d_3 w_1$.

We have that $ta = t a_1 e_3$ and since $ta \neq 1$ it is not possible that both $a_1$ and $e_3 t$ be equal to 1.

First we suppose that $e_3 t \neq 1$. The words $q := d_2 b^{-1} e_3 t$ and $w' := t^{-1} e_3^{-1} a_1^{-1} w_1 = \rho(t^{-1} w t)$ are cyclic permutations of $q_0$ and $w$ respectively. 

We have that 
	\begin{multline} \label{ML4.2B2} \rho(qw') = \rho(d_2 b^{-1} e_3 t t^{-1} e_3^{-1} a_1^{-1} w_1) = \rho(d_2 b^{-1} a_1^{-1} w_1) = \\
	\rho(d_2 b^{-1} a_1^{-1} d_3^{-1} d_2^{-1} d_2 d_3 w_1) = \rho(d_2 p^{-1} d_2^{-1} (d*w)).\end{multline}
Since $\rho(t^{-1} e_1) = d_2$, $\rho(d_2 e_1^{-1}) = t^{-1}$ and $q_0 = \rho(d_2^{-1} q d_2)$, by (\ref{ML4.2B}) we have that 
		$$d*w = \rho(t^{-1} e_1 p q_0 e_1^{-1} w t) = \rho(d_2 p d_2^{-1} q d_2 e_1^{-1} w t) = \rho(d_2 p d_2^{-1} q t^{-1} w t),$$
so the identity among relations that follows from (\ref{ML4.2B2}) is
	$$q  \centerdot t^{-1} w t \equiv d_2 p^{-1} d_2^{-1} \centerdot d_2 p d_2^{-1} \centerdot q  \centerdot t^{-1} w t,$$	
which is basic, proving the first part of the lemma. If $d = u*v$ then $e_1 = 1$, thus $t = 1$ which implies that $w' = w$.

The second part follows by noting that $d = t d_1 a_1 e_3$, $w = e_3^{-1} a_1^{-1} w_1 t^{-1}$, the words $e_3 e_3^{-1}$ and $t^{-1} t$ are canceled in $d*w$ by virtue of Proposition \ref{puzo}, the word $e_3 t t^{-1} e_3^{-1}$ is canceled in $q*w'$ and $e_3 t \neq 1$.	
	
\smallskip

Now suppose that $e_3 t = 1$, that is $a_1 \neq 1$. Then $p = d_3 a_1 b$, $q_0 = b^{-1} d_2$, $w = a_1^{-1} w_1$ and we have that $q_0^{-1} = d_2^{-1} b$. The word $p' := b d_3 a_1 = \rho(b p b^{-1})$ is a cyclic permutation of $p$ and we have
	$$\rho(p' w) = \rho(b d_3 a_1 a_1^{-1} w_1) = \rho(b d_3 w_1),$$
therefore since $d_3 w_1 d_2 = d_2^{-1} (d*w) d_2$ then
	\begin{equation} \label{ML4.2B2'} \rho(d_2^{-1} p' w d_2) = \rho(d_2^{-1} b d_3 w_1 d_2) = \rho(q_0^{-1} d_2^{-1} (d*w) d_2).\end{equation}
Since $q_0 = b^{-1} d_2$, by (\ref{ML4.2B'}) we have that
	$$\rho(d_2^{-1} (d*w) d_2) = \rho(d_2^{-1} b q_0 p b^{-1} w d_2) = \rho(d_2^{-1} b q_0 b^{-1} d_2 d_2^{-1} b p b^{-1} d_2 d_2^{-1} w d_2) =$$	
	$$\rho(q_0^{-1} q_0 q_0 q_0^{-1} p q_0 d_2^{-1} w d_2) = \rho(q_0 q_0^{-1} p q_0 d_2^{-1} w d_2),$$
so the identity among relations that follows from (\ref{ML4.2B2'}) is
	$$q_0^{-1} p q_0  \centerdot d_2^{-1} w d_2 \equiv q_0^{-1} \centerdot q_0 \centerdot q_0^{-1} p q_0  \centerdot d_2^{-1} w d_2,$$	
which is basic. By setting $q := p'$, $p := q_0$ and $w' := w$ this proves the first part of the lemma.

For the second part, we observe that the non-empty word $a_1 a_1^{-1}$ is canceled in $q*w'$ and since $d = t d_1 a_1$ and $w = a_1^{-1} w_1$, it is also canceled in $d*w$.

\bigskip

\textbf{2B3)}  \hspace{0.1cm}
\begin{tabular}{|c|c|c|}
	\hline
	\rule{0pt}{2.3ex}
	$\,\, e_1 \,\,$ & $e_2$ & $\,\, e_3 \,$ \\  
	\hline
\end{tabular}

\hspace{11.6mm}
\begin{tabular}{|c|c|c|}
	\hline 
	\rule{0pt}{2.3ex}
	$t$ & $\,\,\,\,\,\, d_1 \,\,\,\,\,\,\hspace{0.1mm}$ & $a$ \\
	\hline
\end{tabular}

\medskip 

There exist words $d_2, d_3$ such that $e_1 = t d_2$, $d_1 = d_2 e_2 d_3$, $e_3 = d_3 a$, so $q_0 = b^{-1} d_3 a t d_2$ and $d*w = d_2 e_2 d_3 w_1$. Moreover $p = e_2 b$ and thus $p^{-1} = b^{-1} e_2^{-1}$.

The words $q := d_2 b^{-1} d_3 a t$ and $w' := t^{-1} a^{-1} w_1 = \rho(t^{-1} w t)$ are cyclic permutations of $q_0$ and $w$ respectively and we have  
	\begin{multline} \label{ML4.2B3} \rho(qw') = \rho(d_2 b^{-1} d_3 a t t^{-1} a^{-1} w_1) = \rho(d_2 b^{-1} d_3 w_1) = \\
	\rho(d_2 b^{-1} e_2^{-1} d_2^{-1} d_2 e_2 d_3 w_1) = \rho(d_2 p^{-1} d_2^{-1} (d*w)).\end{multline}
Since $\rho(t^{-1} e_1) = d_2$, $\rho(d_2 e_1^{-1}) = t^{-1}$ and $q_0 = \rho(d_2^{-1} q d_2)$, by (\ref{ML4.2B}) we have that 
	$$d*w = \rho(t^{-1} e_1 p q_0 e_1^{-1} w t) = \rho(d_2 p d_2^{-1} q d_2 e_1^{-1} w t) = \rho(d_2 p d_2^{-1} q t^{-1} w t),$$
so the identity among relations that follows from (\ref{ML4.2B3}) is
	$$q  \centerdot t^{-1} w t \equiv d_2 p^{-1} d_2^{-1} \centerdot d_2 p d_2^{-1} \centerdot q  \centerdot t^{-1} w t,$$	
which is basic. By setting $w' := w$, this proves the first part of the lemma. If $d = u*v$ then $e_1 = 1$, thus $t = 1$ which implies that $w' = w$.

The second part follows by noting that $a t \neq 1$, the non-empty word $a t t^{-1} a^{-1}$ is canceled in $q*w'$ and since $d = t d_1 a$ and $w = a^{-1} w_1 t^{-1}$, the words $a a^{-1}$ and $t^{-1} t$ are canceled in $d*w$ by virtue of Proposition \ref{puzo}.

\bigskip

\textbf{2B4)}  \hspace{0.1cm}
\begin{tabular}{|c|c|c|}
	\hline
	\rule{0pt}{2.3ex}
	$\,\, e_1 \,\,$ & $\,\,\, e_2 \,\,\,$ & $e_3$ \\  
	\hline
\end{tabular}

\hspace{11.6mm}
\begin{tabular}{|c|c|c|}
	\hline 
	\rule{0pt}{2.3ex}
	$\,\,\,\,\, t \,\,\,\,\,$ & $d_1$ & $\,\,\, a \,\,\hspace{0.1mm}$ \\
	\hline
\end{tabular}

\medskip 

There exist words $t_1, a_1$ such that $t = e_1 t_1$, $e_2 = t_1 d_1 a_1$, $a = a_1 e_3$. Thus $t^{-1} = t_1^{-1} e_1^{-1}$, $a^{-1} = e_3^{-1} a_1^{-1}$ and then $p = t_1 d_1 a_1 b$, $w = e_3^{-1} a_1^{-1} w_1 t_1^{-1} e_1^{-1}$.

The words $p_0 := d_1 a_1 b t_1 = \rho(t_1^{-1} p t_1)$ and $w' := e_1^{-1} e_3^{-1} a_1^{-1} w_1 t_1^{-1} = \rho(e_1^{-1} w e_1)$ are cyclic permutations of $p$ and $w$ respectively and $p_0^{-1} = t_1^{-1} b^{-1} a_1^{-1} d_1^{-1}$. By setting $q := q_0$ we have that
	$$\rho(qw') = \rho(b^{-1} e_3 e_1 e_1^{-1} e_3^{-1} a_1^{-1} w_1 t_1^{-1}) = \rho(b^{-1} a_1^{-1} w_1 t_1^{-1}) =$$
	$$\rho(b^{-1} a_1^{-1} d_1^{-1} d_1 w_1 t_1^{-1}) = \rho(b^{-1} a_1^{-1} d_1^{-1} (d*w) t_1^{-1}),$$
so	
	\begin{equation} \label{ML4.2B4} \rho(t_1^{-1} qw' t_1) = \rho(t_1^{-1} b^{-1} a_1^{-1} d_1^{-1} (d*w)) = \rho(p_0^{-1} (d*w)).\end{equation}
Since $\rho(t^{-1} e_1) = t_1^{-1}$ and $t = e_1 t_1$, by (\ref{ML4.2B}) we have that 
	$$d*w = \rho(t^{-1} e_1 p q e_1^{-1} w e_1 t_1) = \rho(t_1^{-1} p t_1 \, t_1^{-1} q t_1 \, t_1^{-1} e_1^{-1} w e_1 t_1),$$
so the identity among relations that follows from (\ref{ML4.2B4}) is
	$$t_1^{-1} q t_1  \centerdot (t_1^{-1} e_1^{-1}) w (e_1 t_1) \equiv t_1^{-1} p^{-1} t \centerdot t_1^{-1} p t \centerdot t_1^{-1} q t_1  \centerdot (t_1^{-1} e_1^{-1}) w (e_1 t_1),$$		
which is basic, proving the first part of the lemma. If $d = u*v$ then $e_1 = 1$, which implies that $w' = w$.

The second part follows by noting that $e_3 e_1 \neq 1$, the word $e_3 e_1 e_1^{-1} e_3^{-1}$ is canceled in $q*w'$, $d = e_1 t_1 d_1 a_1 e_3$, $w = e_3^{-1} a_1^{-1} w_1 t_1^{-1} e_1^{-1}$, so the words $e_3 e_3^{-1}$ and $e_1^{-1} e_1$ are canceled in $d*w$ by virtue of Proposition \ref{puzo}.

\bigskip

\textbf{2B5)} \hspace{0.1cm}
\begin{tabular}{|c|c|c|}
	\hline
	\rule{0pt}{2.3ex}
	$\,\, e_1 \,\,$ & $\,\,\, e_2 \,\,\,$ & $\, e_3 \,$ \\  
	\hline
\end{tabular}

\hspace{11.6mm}
\begin{tabular}{|c|c|c|}
	\hline 
	\rule{0pt}{2.3ex}
	$\,\,\,\,\, t \,\,\,\,\,$ & $\,\,\,\, d_1 \,\,\,\hspace{0.1mm}$ & $a$ \\
	\hline
\end{tabular}

\medskip 

There exist words $x, d_2, d_3$ such that $t = e_1 x$, $e_2 = x d_2$, $d_1 = d_2 d_3$, $e_3 = d_3 a$. So $t^{-1} = x^{-1} e_1^{-1}$ and $p = x d_2 b$, $q_0 = b^{-1} d_3 a e_1$, $d = e_1 x d_2 d_3 a$, $w = a^{-1} w_1 x^{-1} e_1^{-1}$, $d*w = d_2 d_3 w_1$.

Since $ta = e_1 x a$ and $ta \neq 1$ then it is not possible that both $e_1 a$ and $x$ be equal to 1.

First suppose that $e_1 a \neq 1$, which implies that $d \neq u*v$. The words $w' := e_1^{-1} a^{-1} w_1 x^{-1} = \rho(e_1^{-1} w e_1)$ and $p_0 := d_2 b x = \rho(x^{-1} p x)$ are cyclic permutations of $w$ and $p$ respectively and $p_0^{-1} = x^{-1} b^{-1} d_2^{-1}$. By setting $q := q_0$ we have
	$$\rho(qw') = \rho(b^{-1} d_3 a e_1 e_1^{-1} a^{-1} w_1 x^{-1}) = \rho(b^{-1} d_3 w_1 x^{-1}) =$$
	$$\rho(b^{-1} d_2^{-1} d_2 d_3 w_1 x^{-1}) = \rho(b^{-1} d_2^{-1} (d*w) x^{-1}),$$
so	
	\begin{equation} \label{ML4.2B5} \rho(x^{-1} q w' x) = \rho(x^{-1} b^{-1} d_2^{-1} (d*w)) = \rho(p_0^{-1} (d*w)).\end{equation}
Since $\rho(t^{-1} e_1) = x^{-1}$ and $x^{-1} e_1^{-1} = t^{-1}$, by (\ref{ML4.2B}) we have that
	$$d*w = \rho(t^{-1} e_1 p q_0 e_1^{-1} w t) = \rho(x^{-1} p q e_1^{-1} w t) =$$
	$$\rho(x^{-1} p x \, x^{-1} q  x \, x^{-1} e_1^{-1} w t)
	 = \rho(x^{-1} p x \, x^{-1} q  x \, t^{-1} w t)$$
and since $\rho(x^{-1} w' x) = \rho(t^{-1} w t)$ the identity among relations that follows from (\ref{ML4.2B5}) is
	$$x^{-1} q x \centerdot t^{-1} w t \equiv x^{-1} p^{-1} x \centerdot x^{-1} p x \centerdot x^{-1} q x  \centerdot t^{-1} w t,$$	
which is basic, proving the first part of the lemma. If $d = u*v$ then $e_1 = 1$, which implies that $w' = w$.

The second part follows by noting that the words $a a^{-1}$ and $e_1^{-1} e_1$ are canceled in $d*w$ by virtue of Proposition \ref{puzo}, the word $a e_1 e_1^{-1} a^{-1}$ is canceled in $q*w'$ and $a e_1 \neq 1$. 

\smallskip

Now suppose that $e_1 a = 1$. This implies that $t = x$, $e_2 = x d_2$, $p = $, $q_0 = b^{-1} d_3$, $d = x d_2 d_3$ and $w = w_1 x^{-1}$. 

 We have that
	\begin{multline} \label{ML4.2B5'} \rho(w p) = \rho(w_1 x^{-1} x d_2 b) = \rho(w_1 d_2 b) = \rho(w_1 d_1 d_1^{-1} d_2 b) =  
		\\ \rho((w*d) d_3^{-1} d_2^{-1} d_2 b) = \rho((w*d) d_3^{-1} b) = \rho((w*d) q_0^{-1}).\end{multline}
By (\ref{ML4.2Bv}) we have that $w*d = \rho(w p q_0)$, so the identity among relations that follows from (\ref{ML4.2B5'}) is
	$$w \centerdot p \equiv w \centerdot p \centerdot q_0 \centerdot q_0^{-1},$$	
which is basic. By setting $q := p$, $p := q_0$, $w' :=w$ this proves the first part of the lemma. To prove the second part we observe that the non-empty word $x^{-1} x$ is canceled $w'*q$ and in $w*d$.

\bigskip

\textbf{2B6)} \hspace{0.1cm}
\begin{tabular}{|c|c|c|}
	\hline
	\rule{0pt}{2.3ex}
	$e_1$ & $e_2$ & $\,\,\,\,\,\,\, e_3 \,\,\,\,\,\,\,$ \\  
	\hline
\end{tabular}

\hspace{11.6mm}
\begin{tabular}{|c|c|c|}
	\hline 
	\rule{0pt}{2.3ex}
	$\,\,\,\,\,\,\,\,\,\, t \,\,\,\,\,\,\,\,\,\hspace{0.1mm}$ & $d_1$ & $a$ \\
	\hline
\end{tabular}

\medskip

There exists a word $x$ such that $t = e_1 e_2 x$, $e_3 = x d_1 a$, so $t^{-1} = x^{-1} e_2^{-1} e_1^{-1}$ and $q_0 = b^{-1} x d_1 a e_1$, $d = e_1 e_2 x d_1 a$, $w = a^{-1} w_1 x^{-1} e_2^{-1} e_1^{-1}$.


The word $w' := e_1^{-1} a^{-1} w_1 x^{-1} e_2^{-1} = \rho(e_1^{-1} w e_1)$ is a cyclic permutations of $w$ and we have that
	\begin{multline} \label{ML4.2B6} \rho(w' p) = \rho(e_1^{-1} a^{-1} w_1 x^{-1} e_2^{-1} e_2 b) = \\ 
	\rho(e_1^{-1} a^{-1} w_1 x^{-1} b) = \rho(e_1^{-1} a^{-1} w_1 d_1 d_1^{-1} x^{-1} b) = \\ \rho(e_1^{-1} a^{-1} (w*d) a e_1 e_1^{-1} a^{-1} d_1^{-1} x^{-1} b) = \rho(e_1^{-1} a^{-1} (w*d) a e_1 q_0^{-1}).\end{multline}
By (\ref{ML4.2Bv}) we have that 
	$$\rho(e_1^{-1} a^{-1} (w*d) a e_1) = \rho(e_1^{-1} a^{-1} a w e_1 p q_0 e_1^{-1} a^{-1} a e_1) = \rho(e_1^{-1} w e_1 p q_0),$$
ao the identity among relations that follows from (\ref{ML4.2B6}) is
	$$e_1^{-1} w e_1 \centerdot p \equiv e_1^{-1} w e_1 \centerdot p \centerdot q_0  \centerdot q_0^{-1},$$	
which is basic. By setting $q := p$, $p := q_0$ this proves the first part of the lemma. If $d = u*v$ then $e_1 = 1$ which implies that $w' = w$.

For the second part we observe that the non-empty word $e_2^{-1} e_2$ is canceled in $w'*q$ and in $w*d$.

\bigskip

\textbf{3A)} $d = w_1^{-1} s (d*w) s^{-1} a = c_2 c_1$, $w = a^{-1} w_1$, $q_0 = p^{-1} r c_1 c_2 r^{-1}$. If $d = u*v$ then $c_2 = 1$.

From (\ref{4.3}) and (\ref{4.A}) we have that
	\begin{equation} \label{ML4.3A} s (d*w) s^{-1} = \rho(w_1 c_1^{-1} r^{-1} p q_0 r c_1 w w_1^{-1})\end{equation}
and that
	\begin{equation} \label{ML4.3A'} s (d*w) s^{-1} = \rho(w_1 c_2 r^{-1} p q_0 r c_2^{-1} w w_1^{-1}).\end{equation}	
From (\ref{4.3}) and (\ref{4.A'}) we have that
\begin{equation} \label{ML4.3A''} s (d*w) s^{-1} = \rho(w_1 (c_2 r^{-1} p) q_0 p (p^{-1} r c_2^{-1}) w w_1^{-1}).\end{equation}


\medskip

The word equation $c_2 c_1 = w_1^{-1} s (d*w) s^{-1} a$ leads to three different possible solutions.

\bigskip

\textbf{3A1)} \hspace{0.1cm}
\begin{tabular}{|c|c|}
	\hline
	\rule{0pt}{2.3ex}
	$\,\,\,\,\,\,\,\,\,\,\,\,\,\,\,\,\,\,\,\,\,\,\, c_2 \,\,\,\,\,\,\,\,\,\,\,\,\,\,\,\,\,\,\,\,\,\,$ & $c_1$ \\  
	\hline
\end{tabular}

\hspace{11.8mm}
\begin{tabular}{|c|c|c|}
	\hline 
	\rule{0pt}{2.7ex}
	$w_1^{-1}$ & $s (d*w) s^{-1}$ & $\,\,\, a \,\,\hspace{0.1mm}$ \\
	\hline
\end{tabular}

\medskip

There exists a word $a_1$ such that $c_2 = w_1^{-1} s (d*w) s^{-1} a_1$, $a = a_1 c_1$, so $a^{-1} = c_1^{-1} a_1^{-1}$ and $q_0 = p^{-1} r c_1 w_1^{-1} s (d*w) s^{-1} a_1 r^{-1}$, $d = w_1^{-1} s (d*w) s^{-1} a_1 c_1$, $w = c_1^{-1} a_1^{-1} w_1$. Since $d*w \neq 1$ and since $d*w$ is a subword of $c_2$ then it is not possible that $d = u*v$.


First we suppose that $c_1^{-1} w_1 \neq 1$. The words $q := s (d*w) s^{-1} a_1 r^{-1} p^{-1} r c_1 w_1^{-1}$ and $w' := w_1 c_1^{-1} a_1^{-1} = \rho(w_1 w w_1^{-1})$ are cyclic permutations of $q_0$ and $w$ respectively.

Since $\rho(s (d*w) s^{-1} a_1) = w_1 c_2$ then 
	$$\rho(qw') = \rho(s (d*w) s^{-1} a_1 r^{-1} p^{-1} r c_1 w_1^{-1} w_1 c_1^{-1} a_1^{-1})=$$
	$$\rho(w_1 c_2 r^{-1} p^{-1} r a_1^{-1}) = \rho\big((w_1 c_2 r^{-1}) p^{-1} (r c_2^{-1} w_1^{-1}) w_1 c_2 a_1^{-1}\big) =$$
	$$\rho\big((w_1 c_2 r^{-1}) p^{-1} (r c_2^{-1} w_1^{-1}) s (d*w) s^{-1} a_1 a_1^{-1}\big) =$$
	$$\rho\big((w_1 c_2 r^{-1}) p^{-1} (r c_2^{-1} w_1^{-1}) s (d*w) s^{-1}\big),$$
thus
	\begin{equation} \label{ML4.3A1} \rho(w_1^{-1} qw' w_1) = \rho\big((c_2 r^{-1}) p^{-1} (r c_2^{-1}) w_1^{-1} s (d*w) s^{-1} w_1\big).\end{equation}
Since $q_0 = \rho\big((s (d*w) s^{-1} a_1 r^{-1})^{-1} q (s (d*w) s^{-1} a_1 r^{-1})\big)= \rho\big((r c_2^{-1} w_1^{-1}) q (w_1 c_2 r^{-1})\big)$, by (\ref{ML4.3A''}) we have that
	$$w_1^{-1} s (d*w) s^{-1} w_1 = \rho\big(w_1^{-1} w_1 (c_2 r^{-1} p) q_0 p (p^{-1} r c_2^{-1}) w w_1^{-1} w_1\big) =$$
	$$\rho\big((c_2 r^{-1} p) (r c_2^{-1} w_1^{-1}) q (w_1 c_2 r^{-1}) p (p^{-1} r c_2^{-1}) w\big) =$$
	$$\rho\big((c_2 r^{-1}) p (r c_2^{-1}) \, w_1^{-1} q w_1 \, w\big),$$	
so the identity among relations that follows from (\ref{ML4.3A1}) is
	$$w_1^{-1} q w_1 \centerdot w \equiv (c_2 r^{-1}) p^{-1} (r c_2^{-1}) \centerdot (c_2 r^{-1}) p (r c_2^{-1} w_1) \centerdot w_1^{-1} q w_1 \centerdot w,$$	
which is basic, proving the first part of the lemma. The second part follows by noting that the words $c_1 c_1^{-1}$ and $w_1 w_1^{-1}$ are canceled in $d*w$ by virtue of Proposition \ref{puzo}, the word $c_1 w_1^{-1} w_1 c_1^{-1}$ is canceled in $q*w'$ and $c_1 w_1^{-1} \neq 1$.

\smallskip

Now suppose that $c_1^{-1} w_1 = 1$. Then $w = a_1^{-1}$, $q_0 = p^{-1} r s (d*w) s^{-1} w^{-1} r^{-1}$ and $d = s (d*w) s^{-1} w^{-1}$. The word $q := r^{-1} p^{-1} r s (d*w) s^{-1} w^{-1}$ is a cyclic permutation of $q_0$ and we have that
	\begin{equation} \label{ML4.3A1'} \rho(qw) = \rho(r^{-1} p^{-1} r s (d*w) s^{-1} w^{-1} w) = \rho(r^{-1} p^{-1} r s (d*w) s^{-1}).\end{equation}
Since $q_0 = \rho(r q r^{-1})$, by (\ref{ML4.3A}) we have that
	$$s (d*w) s^{-1} = \rho(r^{-1} p q_0 r w) = \rho(r^{-1} p r q r^{-1} r w) = \rho(r^{-1} p r q w),$$
so the identity among relations that follows from (\ref{ML4.3A1'}) is
	$$q \centerdot w \equiv r^{-1} p^{-1} r \centerdot r^{-1} p r \centerdot q \centerdot w,$$	
which is basic, proving the first part of the lemma. The second part follows by noting that the non-empty word $w^{-1} w$ is canceled in $d*w$ and $q*w$.

\bigskip

\textbf{3A2)} \hspace{0.1cm}
\begin{tabular}{|c|c|}
	\hline
	\rule{0pt}{2.3ex}
	$\hspace{8 mm} c_2 \hspace{7 mm}$ & $\hspace{5.25 mm} c_1 \hspace{5.25 mm}$ \\  
	\hline
\end{tabular}

\hspace{11.8mm}
\begin{tabular}{|c|c|c|}
	\hline 
	\rule{0pt}{2.7ex}
	$w_1^{-1}$ & $s (d*w) s^{-1}$ & $a$ \\
	\hline
\end{tabular}

\medskip

There exist words $x, y$ such that $c_2 = w_1^{-1} x$, $s (d*w) s^{-1} = x y$, $c_1 = y a$, so $q_0 = p^{-1} r y a w_1^{-1} x r^{-1}$ and $d = w_1^{-1} xya$. 

The words $q := x r^{-1} p^{-1} r y a w_1^{-1}$ and $w' := w_1 a^{-1} = \rho(w_1 w w_1^{-1})$ are cyclic permutations of $q_0$ and $w$ respectively and
	$$\rho(qw') = \rho(x r^{-1} p^{-1} r y a w_1^{-1} w_1 a^{-1}) = \rho(x r^{-1} p^{-1} r y),$$
so since $yx = \rho(x^{-1} s (d*w) s^{-1} x)$ we have that 
	\begin{equation} \label{ML4.3A2} \rho(x^{-1} qw' x) = \rho(r^{-1} p^{-1} r y x) = \rho(r^{-1} p^{-1} r \, x^{-1} s (d*w) s^{-1} x).\end{equation}	
Since $q_0 = \rho(r x^{-1} q x r^{-1})$, $\rho(x^{-1} w_1 c_2) = 1$ and $c_2 = w_1^{-1} x$, by (\ref{ML4.3A'}) we have that
	$$\rho(x^{-1} s (d*w) s^{-1} x) = \rho(x^{-1} w_1 c_2 r^{-1} p q_0 r c_2^{-1} w w_1^{-1} x) = \rho(r^{-1} p r x^{-1} q x c_2^{-1} w c_2).$$	
Since $\rho(x^{-1} w' x) = \rho(c_2^{-1} w c_2)$, the identity among relations that follows from (\ref{ML4.3A2}) is
	$$x^{-1} q x \centerdot c_2^{-1} w c_2 \equiv r^{-1} p^{-1} r \centerdot r^{-1} p r \centerdot x^{-1} q x \centerdot c_2^{-1} w c_2,$$	
which is basic, proving the first part of the lemma. If $d = u*v$ then $c_2 = 1$, therefore $w_1 = 1$ which implies that $w' = w$.

The second part follows by noting that the words $a a^{-1}$ and $w_1 w_1^{-1}$ are canceled in $d*w$ by virtue of Proposition \ref{puzo}, the word $a w_1^{-1} w_1 a^{-1}$ is canceled in $q*w'$ and $a w_1^{-1} \neq 1$ since $w = a^{-1} w_1$ and $w \neq 1$.

\bigskip

\textbf{3A3)} \hspace{0.1cm}
\begin{tabular}{|c|c|}
	\hline
	\rule{0pt}{2.3ex}
	$c_2$ & $\hspace{12.75 mm} c_1  \hspace{12.75 mm}$ \\  
	\hline
\end{tabular}

\hspace{11.8mm}
\begin{tabular}{|c|c|c|}
	\hline 
	\rule{0pt}{2.7ex}
	$w_1^{-1}$ & $s (d*w) s^{-1}$ & $a$ \\
	\hline
\end{tabular}

\medskip

There exists a word $x$ such that $w_1^{-1} = c_2 x$, $c_1 = x s (d*w)  s^{-1} a$, thus $w_1 = x^{-1} c_2^{-1}$, $q_0 = p^{-1} r x s (d*w) s^{-1} a c_2 r^{-1}$, $d = c_2 x s (d*w) s^{-1} a$, $w = a^{-1} x^{-1} c_2^{-1}$.

First we suppose that $a c_2 \neq 1$. The words $q := r^{-1} p^{-1} r x s (d*w) s^{-1} a c_2$ and $w' := c_2^{-1} a^{-1} x^{-1} = \rho(c_2^{-1} w c_2)$ are cyclic permutations of $q_0$ and $w$ respectively and
	\begin{multline} \label{ML4.3A3} \rho(qw') = \rho(r^{-1} p^{-1} r x s (d*w) s^{-1} a c_2 c_2^{-1} a^{-1} x^{-1}) =\\
	\rho(r^{-1} p^{-1} r x s (d*w) s^{-1} x^{-1}).\end{multline} 	
Since $\rho(x w_1 c_2) = 1$, $q_0 = \rho(r q r^{-1})$ and $\rho(w_1^{-1} x^{-1}) = c_2$, by (\ref{ML4.3A'}) we have that  
	$$\rho(x s (d*w) s^{-1} x^{-1}) = \rho(x w_1 c_2 r^{-1} p q_0 r c_2^{-1} w w_1^{-1} x^{-1}) = \rho(r^{-1} p r q c_2^{-1} w c_2),$$
so the identity among relations that follows from (\ref{ML4.3A3}) is
	$$q \centerdot c_2^{-1} w c_2 \equiv r^{-1} p^{-1} r \centerdot r^{-1} p r \centerdot q \centerdot c_2^{-1} w c_2,$$	
which is basic, proving the first part of the lemma. If $d = u*v$ then $c_2 = 1$, which implies that $w' = w$.

The second part follows by noting that the words $a a^{-1}$ and $c_2^{-1} c_2$ are canceled in $d*w$ by virtue of Proposition \ref{puzo}, the word $a c_2 c_2^{-1} a^{-1}$ is canceled in $q*w'$ and $a c_2 \neq 1$.

\smallskip

Now we suppose that $a c_2 = 1$. Then $w = w_1 = x^{-1}$, $d = w^{-1} s (d*w) s^{-1}$ and $q_0 = p^{-1} r w^{-1} s (d*w) s^{-1} r^{-1}$. We have that $\rho(wd) = s (d*w) s^{-1}$, so $w*d = d*w$. The word $q := w^{-1} s (d*w) s^{-1} r^{-1} p^{-1} r$ is a cyclic permutation of $q_0$ and 
	\begin{equation} \label{ML4.3A3'} \rho(w q) = \rho(w w^{-1} s (d*w) s^{-1} r^{-1} p^{-1} r) = \rho(s (w*d) s^{-1} r^{-1} p^{-1} r).\end{equation}	
Since $q_0 = \rho(p^{-1} r q r^{-1} p)$ by (\ref{ML4.3A''}) we have that
	 $$s (w*d) s^{-1} = \rho(w_1 (r^{-1} p) q_0 p (p^{-1} r) w w_1^{-1}) =$$
	 $$\rho(w (r^{-1} p) p^{-1} r q r^{-1} p p (p^{-1} r) w w^{-1}) = \rho(w q r^{-1} p r),$$
so the identity among relations that follows from (\ref{ML4.3A3'}) is
	$$w \centerdot q \equiv w \centerdot q \centerdot r^{-1} p r \centerdot r^{-1} p^{-1} r,$$	
which is basic. By setting $w' := w$ we prove the first part of the lemma. The second part follows by noting that non-empty word $w w^{-1}$ is canceled in $w*d$ and $q*w'$.

\bigskip

\textbf{3B)} $d = w_1^{-1} s (d*w) s^{-1} a = e_1 e_2 e_3$, $w = a^{-1} w_1$, $p = e_2 b$, $q_0 = b^{-1} e_3 e_1$ for words $b, e_1, e_2, e_3$ such that $b, e_2, e_3 e_1 \neq 1$. If $d = u*v$ then $e_1 = 1$.

From (\ref{4.3}) and (\ref{4.B}) we have that
	\begin{equation} \label{ML4.3B} s (d*w) s^{-1} = \rho(w_1 e_1 p q_0 e_1^{-1} w w_1^{-1})\end{equation}	
and that
	\begin{equation} \label{ML4.3B'} s (d*w) s^{-1} = \rho(w_1 e_3^{-1} e_2^{-1} p q_0 e_2 e_3 w w_1^{-1}).\end{equation}	
From (\ref{4.3}) and (\ref{4.B'}) we have that
	\begin{equation} \label{ML4.3B''} s (d*w) s^{-1} = \rho(w_1 e_3^{-1} b q_0 p b^{-1} e_3 w w_1^{-1})\end{equation}	
and that
	\begin{equation} \label{ML4.3B'''} s (d*w) s^{-1} = \rho(w_1 e_1 e_2 b q_0 p b^{-1} e_2^{-1} e_1^{-1} w w_1^{-1}).\end{equation}	
From (\ref{4.4}) and (\ref{4.B}) we have that
	\begin{equation} \label{ML4.3B'v}\alpha (w*d) \alpha^{-1} = \rho(a^{-1} s (d*w) s^{-1} a) = \rho(w e_1 p q_0 e_1^{-1}).\end{equation}

\medskip

The word equation $e_1 e_2 e_3 = w_1^{-1} s (d*w) s^{-1} a$ leads to six different possible solutions.

\bigskip

\textbf{3B1)} \hspace{0.1cm}
\begin{tabular}{|c|c|c|}
	\hline
	\rule{0pt}{2.3ex}
	$e_1$ & $e_2$ & $\hspace{12 mm} e_3  \hspace{12 mm}$ \\  
	\hline
\end{tabular}

\hspace{11.6mm}
\begin{tabular}{|c|c|c|}
	\hline 
	\rule{0pt}{2.7ex}
	$\hspace{3.25 mm} w_1^{-1} \hspace{3.25 mm}$ & $s (d*w) s^{-1}$ & $a$ \\
	\hline
\end{tabular}

\medskip

There exists a word $x$ such that $w_1^{-1} = e_1 e_2 x$ and $e_3 = x s (d*w) s^{-1} a$, so $d = e_1 e_2 x s (d*w) s^{-1} a$. This implies that $q_0 = b^{-1} x s (d*w) s^{-1} a e_1$ and $w_1 = x^{-1} e_2^{-1} e_1^{-1}$, thus $w = a^{-1} x^{-1} e_2^{-1} e_1^{-1}$. 

The word $w' := e_1^{-1} a^{-1} x^{-1} e_2^{-1} = \rho(e_1^{-1} w e_1)$ is a cyclic permutation of $w$ and by (\ref{ML4.3B'v}) we have that
	\begin{multline} \label{ML4.3B1+++} \rho(w' p) = \rho(e_1^{-1} a^{-1} x^{-1} e_2^{-1} e_2 b) = \rho(e_1^{-1} a^{-1} x^{-1} b) = \\
	\rho(e_1^{-1} a^{-1} x^{-1} b) = \rho(e_1^{-1} a^{-1} s (d*w) s^{-1} a e_1 e_1^{-1} a^{-1} x^{-1} b) = 
	\\ \rho(e_1^{-1} a^{-1} s (d*w) s^{-1} a e_1 q_0^{-1}) = \rho(e_1^{-1} \alpha (w*d) \alpha^{-1} e_1 q_0^{-1}).\end{multline}
By (\ref{ML4.3B'v}) we have that 
	$$\rho(e_1^{-1} \alpha (w*d) \alpha^{-1} e_1) = \rho(e_1^{-1} w e_1 p q_0 e_1^{-1} e_1) = \rho(e_1^{-1} w e_1 p q_0),$$
so the identity among relations that follows from (\ref{ML4.3B1+++}) is
	$$e_1^{-1} w e_1 \centerdot p \equiv e_1^{-1} w e_1 \centerdot p \centerdot q_0 \centerdot q_0^{-1},$$	
which is basic. By setting $q := p$ and $p := q_0$ this proves the first part of the lemma. If $d = u*v$ then $e_1 = 1$, which implies that $w' = w$.

For the second part we observe that the non-empty word $e_2^{-1} e_2$ is canceled in $w'*q$ and $w*d$.

\bigskip

\textbf{3B2)} \hspace{0.1cm}
\begin{tabular}{|c|c|c|}
	\hline
	\rule{0pt}{2.3ex}
	$e_1$ & $e_2$ & $\hspace{8.8 mm} e_3  \hspace{8.8 mm}$ \\  
	\hline
\end{tabular}

\hspace{11.6mm}
\begin{tabular}{|c|c|c|}
	\hline 
	\rule{0pt}{2.7ex}
	$w_1^{-1}$ & $s (d*w) s^{-1}$ & $a$ \\
	\hline
\end{tabular}

\medskip

There exist words $x, y, z$ such that $w_1^{-1} = e_1 x$, $e_2 = xy$, $s (d*w) s^{-1} = yz$, $e_3 = z a$, so $d = e_1 x s (d*w) s^{-1} a$, $p = xyb$ and $q_0 = b^{-1} z a e_1$. This implies that $p^{-1} = b^{-1} y^{-1} x^{-1}$, $w_1 = x^{-1} e_1^{-1}$ and thus $w = a^{-1} x^{-1} e_1^{-1}$. 


First let us assume that $a e_1 \neq 1$. The word $w' := e_1^{-1} a^{-1} x^{-1} = \rho(e_1^{-1} w e_1)$ is a cyclic permutation of $w$ and by setting $q := q_0$ we have
	\begin{multline} \label{ML4.3B2} \rho(qw') = \rho(b^{-1} z a e_1 e_1^{-1} a^{-1} x^{-1}) = \rho(b^{-1} z x^{-1}) = \\
	\rho(b^{-1} y^{-1} x^{-1} x y z x^{-1}) = \rho(p^{-1} x s (d*w) s^{-1} x^{-1}). \end{multline}	
Since $\rho(x w_1 e_1) = 1$ and $\rho(w_1^{-1} x^{-1}) = e_1$, by (\ref{ML4.3B}) we have that
	$$\rho(x s (d*w) s^{-1} x^{-1}) = \rho(x w_1 e_1 p q_0 e_1^{-1} w w_1^{-1} x^{-1}) = \rho(p q e_1^{-1} w e_1),$$
so the identity among relations that follows from (\ref{ML4.3B2}) is
	$$q \centerdot e_1^{-1} w e_1 \equiv p^{-1} \centerdot p \centerdot q \centerdot e_1^{-1} w e_1,$$	
which is basic, proving the first part of the lemma. If $d = u*v$ then $e_1 = 1$, which implies that $w' = w$.

The second part follows by noting that the words $a a^{-1}$ and $e_1^{-1} e_1$ are canceled in $d*w$ by virtue of Proposition \ref{puzo}, the word $a e_1 e_1^{-1} a^{-1}$ is canceled in $q*w'$ and $a e_1 \neq 1$. 

\smallskip
	
Now let $a e_1 = 1$. This implies that $q_0 = b^{-1} z$, $w = w_1 = x^{-1}$, $p = w^{-1} y b$ and $e_2 = w^{-1} y$.

The words $p_1 := y b w^{-1} = \rho(y b p b^{-1} y^{-1})$ and $q_1 := z b^{-1}$ are cyclic permutations of $p$ and $q_0$ respectively and $q_1^{-1} := b z^{-1}$. We have that 
	\begin{multline} \label{ML4.3B2'} \rho(p_1 w) = \rho(y b w^{-1} w) = \rho(yb) = \rho(y b z^{-1} y^{-1} y z) = \\
	\rho(y q_1^{-1} y^{-1} s (d*w) s^{-1}).\end{multline}
Since $\rho(w e_2) = y$, $q_0 = \rho(b^{-1} q_1 b)$, $w = w_1$ and $\rho(y e_2^{-1}) = w$, by (\ref{ML4.3B'''}) we have that
	$$s (d*w) s^{-1} = \rho(w e_2 b q_0 p b^{-1} e_2^{-1} w w_1^{-1}) = \rho(y b b^{-1} q_1 b p b^{-1} e_2^{-1}) =$$
	$$\rho(y q_1 y^{-1} (y b) p (b^{-1} y^{-1}) y e_2^{-1}) = \rho(y q_1 y^{-1} p_1 w),$$
so the identity among relations that follows from (\ref{ML4.3B2'}) is
	$$p_1 \centerdot w \equiv y q_1^{-1} y^{-1} \centerdot y q_1 y^{-1} \centerdot p_1 \centerdot w,$$	
which is basic. By setting $q := p_1$ and $p := q_1$ this proves the first part of the lemma. For the second part we observe that the non-empty word $w^{-1} w$ is canceled in $q*w'$ and since $d = w s (d*w) s^{-1}$ it also canceled in $d*w$ by virtue of Proposition \ref{puzo}. 


\bigskip

\textbf{3B3)} \hspace{0.1cm}
\begin{tabular}{|c|c|c|}
	\hline
	\rule{0pt}{2.3ex}
	$e_1$ & $\hspace{10.5 mm} e_2 \hspace{10.5 mm}$ & $e_3$ \\  
	\hline
\end{tabular}

\hspace{11.6mm}
\begin{tabular}{|c|c|c|}
	\hline 
	\rule{0pt}{2.7ex}
	$w_1^{-1}$ & $s (d*w) s^{-1}$ & $\hspace{1.7 mm} a \hspace{1.7 mm}$ \\
	\hline
\end{tabular}

\medskip

There exist words $x, y$ such that $w_1^{-1} = e_1 x$, $e_2 = x s (d*w) s^{-1} y$, $a = y e_3$, so $d = e_1 x s (d*w) s^{-1} y e_3$, $p = x s (d*w) s^{-1} y b$ and $p^{-1} = b^{-1} y^{-1} s (d*w)^{-1} s^{-1} x^{-1}$. 

We have that $a^{-1} = e_3^{-1} y^{-1}$ and $w_1 = x^{-1} e_1^{-1}$, so $w = e_3^{-1} y^{-1} x^{-1} e_1^{-1}$. 

The word $w' := e_1^{-1} e_3^{-1} y^{-1} x^{-1} = \rho(e_1^{-1} w e_1)$ is a cyclic permutation of $w$ and by setting $q := q_0$ we have
	\begin{multline} \label{ML4.3B3} \rho(qw') = \rho(b^{-1} e_3 e_1 e_1^{-1} e_3^{-1} y^{-1} x^{-1}) = \rho(b^{-1} y^{-1} x^{-1}) = \\
	\rho(b^{-1} y^{-1} s (d*w)^{-1} s^{-1} x^{-1} x s (d*w) s^{-1} x^{-1}) = 
	\rho(p^{-1} x s (d*w) s^{-1} x^{-1}).\end{multline}
Since $\rho(x w_1 e_1) = 1$ and $\rho(w_1^{-1} x^{-1}) = e_1$, by (\ref{ML4.3B}) we have that
	$$\rho(x s (d*w) s^{-1} x^{-1}) = \rho(x w_1 e_1 p q_0 e_1^{-1} w w_1^{-1} x^{-1}) 
	 = \rho(p q e_1^{-1} w e_1),$$
so the identity among relations that follows from (\ref{ML4.3B3}) is
	$$q \centerdot e_1^{-1} w e_1 \equiv p^{-1} \centerdot p \centerdot q \centerdot e_1^{-1} w e_1,$$	
which is basic, proving the first part of the lemma. If $d = u*v$ then $e_1 = 1$, which implies that $w' = w$.

The second part follows by noting that since $d = e_1 x s (d*w) s^{-1} y e_3$ and $w = e_3^{-1} y^{-1} x^{-1} e_1^{-1}$, the words $e_3 e_3^{-1}$ and $e_1^{-1} e_1$ are canceled in $d*w$ by virtue of Proposition \ref{puzo}, the word $e_3 e_1 e_1^{-1} e_3^{-1}$ is canceled in $q*w'$ and $e_3 e_1 \neq 1$.

\bigskip

\textbf{3B4)} \hspace{0.1cm}
\begin{tabular}{|c|c|c|}
	\hline
	\rule{0pt}{2.3ex}
	$\hspace{4 mm} e_1 \hspace{4 mm}$ & $\hspace{2 mm} e_2 \hspace{2 mm}$ & $\hspace{2.75 mm} e_3 \hspace{2.75 mm}$ \\  
	\hline
\end{tabular}

\hspace{11.6mm}
\begin{tabular}{|c|c|c|}
	\hline 
	\rule{0pt}{2.7ex}
	$w_1^{-1}$ & $s (d*w) s^{-1}$ & $a$ \\
	\hline
\end{tabular}

\medskip

There exists words $x, y$ such that $e_1 = w_1^{-1} x$, $s (d*w) s^{-1} = x e_2 y$, $e_3 = y a$, which implies that $q_0 = b^{-1} y a w_1^{-1} x$. The words $q := x b^{-1} y a w_1^{-1}$ and $w' := w_1 a^{-1} = \rho(w_1 w w_1^{-1})$ are cyclic permutations of $q_0$ and $w$ respectively and we have that
	\begin{multline} \label{ML4.3B4} \rho(q w') = \rho(x b^{-1} y a w_1^{-1} w_1 a^{-1}) = \rho(x b^{-1} y) = \\
	\rho(x b^{-1} e_2^{-1} x^{-1} x e_2 y) = \rho(x p^{-1} x^{-1} s (d*w) s^{-1}).\end{multline}
Since $\rho(w_1 e_1) = x$, $q_0 = \rho(x^{-1} q x)$ and $\rho(x e_1^{-1}) = w_1$, by (\ref{ML4.3B}) we have that
	$$s (d*w) s^{-1} = \rho(w_1 e_1 p q_0 e_1^{-1} w w_1^{-1}) = \rho(x p x^{-1} q x e_1^{-1} w w_1^{-1}) = \rho(x p x^{-1} q w_1 w w_1^{-1}),$$
so the identity among relations that follows from (\ref{ML4.3B4}) is
	$$q \centerdot w_1 w w_1^{-1} \equiv x p^{-1} x^{-1} \centerdot x p x^{-1} \centerdot q \centerdot w_1 w w_1^{-1},$$	
which is basic, proving the first part of the lemma. If $d = u*v$ then $e_1 = 1$, thus $w_1 = 1$ which implies that $w' = w$.

The second part follows by noting that since $d = w_1^{-1} s (d*w) s^{-1} a$ and $w = a^{-1} w_1$, the words $a a^{-1}$ and $w_1 w_1^{-1}$ are canceled in $d*w$ by virtue of Proposition \ref{puzo}, the word $a w_1^{-1} w_1 a^{-1}$ is canceled in $q*w'$ and $a w_1 \neq 1$.

\bigskip

\textbf{3B5)} \hspace{0.1cm}
\begin{tabular}{|c|c|c|}
	\hline
	\rule{0pt}{2.3ex}
	$\hspace{4 mm} e_1 \hspace{4 mm}$ & $\hspace{7 mm} e_2  \hspace{7 mm}$ & $e_3$ \\  
	\hline
\end{tabular}

\hspace{11.6mm}
\begin{tabular}{|c|c|c|}
	\hline 
	\rule{0pt}{2.7ex}
	$w_1^{-1}$ & $s (d*w) s^{-1}$ & $\hspace{2.2 mm} a \hspace{2.2 mm}$ \\
	\hline
\end{tabular}

\medskip

There exist words $x, y, a_1$ such that $e_1 = w_1^{-1} x$, $s (d*w) s^{-1} = xy$, $e_2 = y a_1$, $a = a_1 e_3$, thus $p = y a_1 b$, $q_0 = b^{-1} e_3 w_1^{-1} x$, $d = w_1^{-1} x y a_1 e_3$. We have that $p^{-1} = b^{-1} a_1^{-1} y^{-1}$ and that $a^{-1} = e_3^{-1} a_1^{-1}$, therefore $w = e_3^{-1} a_1^{-1} w_1$.



First let us assume that $e_3^{-1} w_1 \neq 1$. The words $q := x b^{-1} e_3 w_1^{-1}$ and $w' := w_1 e_3^{-1} a_1^{-1} = \rho(w_1 w w_1^{-1})$ are cyclic permutations of $q_0$ and $w$ respectively and we have that
	\begin{multline} \label{ML4.3B5} \rho(qw') = \rho(x b^{-1} e_3 w_1^{-1} w_1 e_3^{-1} a_1^{-1}) = \rho(x b^{-1} a_1^{-1}) = \\
 	\rho(x b^{-1} a_1^{-1} y^{-1} x^{-1} x y) = \rho(x p^{-1} x^{-1} s (d*w) s^{-1}).\end{multline}
Since $\rho(w_1 e_1) = x$, $q_0 = \rho(x^{-1} q x)$ and $\rho(x e_1^{-1}) = w_1$, by (\ref{ML4.3B}) we have that
	$$s (d*w) s^{-1} = \rho(w_1 e_1 p q_0 e_1^{-1} w w_1^{-1}) = \rho(x p x^{-1} q x e_1^{-1} w w_1^{-1}) = \rho(x p x^{-1} q w_1 w w_1^{-1}),$$
so the identity among relations that follows from (\ref{ML4.3B5}) is
	$$q \centerdot w_1 w w_1^{-1} \equiv x p^{-1} x^{-1} \centerdot x p x^{-1} \centerdot q \centerdot w_1 w w_1^{-1},$$	
which is basic, proving the first part of the lemma. If $d = u*v$ then $e_1 = 1$, thus $w_1 = 1$ which implies that $w' = w$.

The second part follows by noting that since $d = w_1^{-1} x y a_1 e_3$ and $w = e_3^{-1} a_1^{-1} w_1$, the words $a_1 a_1^{-1}$ and $w_1 w_1^{-1}$ are canceled in $d*w$ by virtue of Proposition \ref{puzo}, the word $e_3 w_1^{-1} w_1 e_3^{-1}$ is canceled in $q*w'$ and $e_3 w_1 \neq 1$.

\smallskip

Now let $e_3^{-1} w_1 = 1$. This implies that $q_0 = b^{-1} x$ and $w = a_1^{-1}$, therefore $p = y w^{-1} b$. The words $p_1 := b y w^{-1} = \rho(b p b^{-1})$ and $q_1 := x b^{-1}$ are cyclic permutations of $p$ and $q_0$ respectively and $q_1^{-1} = b x^{-1}$. We have that
	\begin{equation} \label{ML4.3B5'} \rho(p_1 w) = \rho(b y w^{-1} w) = \rho(b y) = \rho(b x^{-1} x y) = \rho(q_1^{-1} s (d*w) s^{-1}).\end{equation}
Since $q_0 = \rho(b^{-1} q_1 b)$, by (\ref{ML4.3B''}) we have that
	$$s (d*w) s^{-1} = \rho(b q_0 p b^{-1} w) = \rho(b b^{-1} q_1 b p b^{-1} w) = \rho(q_1 p_1 w),$$
so the identity among relations that follows from (\ref{ML4.3B5'}) is
	$$p_1 \centerdot w \equiv q_1^{-1} \centerdot q_1 \centerdot p_1 \centerdot w,$$	
which is basic. By setting $w' := w$, $q := p_1$ and $p := q_1$, this proves the first part of the lemma. For the second part we observe that the non-empty word $w^{-1} w$ is canceled in $d*w$ and $q*w'$. 

\bigskip

\textbf{3B6)} \hspace{0.1cm}
\begin{tabular}{|c|c|c|}
	\hline
	\rule{0pt}{2.3ex}
	$\hspace{15 mm} e_1 \hspace{15 mm}$ & $e_2$ & $e_3$ \\  
	\hline
\end{tabular}

\hspace{11.6mm}
\begin{tabular}{|c|c|c|}
	\hline 
	\rule{0pt}{2.7ex}
	$w_1^{-1}$ & $s (d*w) s^{-1}$ & $\hspace{6.2 mm} a \hspace{6.2 mm}$ \\
	\hline
\end{tabular}

\medskip

There exists a word $a_1$ such that $e_1 = w_1^{-1} s (d*w) s^{-1} a_1$, $a = a_1 e_2 e_3$, thus $q_0 = b^{-1} e_3 w_1^{-1} s (d*w) s^{-1} a_1$. We have that $a^{-1} = e_3^{-1} e_2^{-1} a_1^{-1}$, therefore $w = e_3^{-1} e_2^{-1} a_1^{-1} w_1$. Since $d*w \neq 1$ and since $d*w$ is a subword of $e_1$ then it is not possible that $d = u*v$.

The words $p_1 := b e_2 = \rho(b p b^{-1})$, $w' := e_2^{-1} a_1^{-1} w_1 e_3^{-1} = \rho(e_3 w e_3^{-1})$ and $q_1 := w_1^{-1} s (d*w) s^{-1} a_1 b^{-1} e_3$ are cyclic permutations of $p$, $w$ and $q_0$ respectively and $q_1^{-1} = e_3^{-1} b a_1^{-1} s (d*w)^{-1} s^{-1} w_1$. We have that
	$$\rho(p_1 w') = \rho(b e_2 e_2^{-1} a_1^{-1} w_1 e_3^{-1}) = \rho(b a_1^{-1} w_1 e_3^{-1}),$$ 
so
	\begin{multline} \label{ML4.3B6} \rho(e_3^{-1} p_1 w' e_3) = \rho(e_3^{-1} b a_1^{-1} w_1) = \\
	\rho(e_3^{-1} b a_1^{-1} s (d*w)^{-1} s^{-1} w_1 w_1^{-1} s (d*w) s^{-1} w_1) = \\
	\rho(q_1^{-1} w_1^{-1} s (d*w) s^{-1} w_1).\end{multline}
Since $q_0 = \rho(b^{-1} e_3 q_1 e_3^{-1} b)$, by (\ref{ML4.3B''}) we have that
	$$\rho(w_1^{-1} s (d*w) s^{-1} w_1) = \rho(w_1^{-1} w_1 e_3^{-1} b q_0 p b^{-1} e_3 w w_1^{-1} w_1) =$$ 
	$$\rho(e_3^{-1} b b^{-1} e_3 q_1 e_3^{-1} b p b^{-1} e_3 w) = \rho(q_1 e_3^{-1} p_1 e_3 w). $$ 
Since $\rho(e_3^{-1} w' e_3) = w$, the identity among relations that follows from (\ref{ML4.3B6}) is
	$$e_3^{-1} p_1 e_3 \centerdot w \equiv q_1^{-1} \centerdot q_1 \centerdot e_3^{-1} p_1 e_3 \centerdot w,$$	
which is basic. By setting $q:= p_1$ and $p:= q_1$ this proves the first part of the lemma. For the second part we observe that the non-empty word $e_2 e_2^{-1}$ is canceled in $q*w'$ and since $d = e_1 e_2 e_3$ and $w = e_3^{-1} e_2^{-1} a_1^{-1} w_1$ it is canceled in $d*w$.


\appendix

\section{Van Kampen diagrams} \label{sectVKdiag}

In this section we briefly introduce some notions of van Kampen diagrams. The material presented here is quite standard, except the observation made in Remark \ref{remVK} that we have not been able to find it in the numerous references on van Kampen diagrams.

This section is only needed in Remark \ref{nonHomeoVK} and the reader can skip it if s/he is not interested in the interpretation of Theorem \ref{mainTheor2} in terms of diagrams.

\smallskip

A van Kampen diagram (see for example \cite{Short} for a thorough introduction) is a planar 2-cell complex that can be associated with any relator of a group presentation. A relator is the reduced form of a product of conjugates of basic relators, and with the non-reduced product of these basic relators can be associated a van Kampen diagram in the form of a ``bouquet of lollipops". 

Indeed let $r_1, \cdots, r_n$ be basic relators of a group presentation and let $a_1, \cdots, a_n$ be words. With the word $w := a_1 r_1 a_1^{-1} \cdots a_n r_n a_n^{-1}$ can be associated the van Kampen diagram of Figure \ref{fig:lollipop}.

\begin{figure}[h]
	\begin{picture}(80, 140) 
		\put(170, 19){\line(-1, 1){38}}
		\put(170, 19){\line(1, 5){11}}
		\put(170, 19){\line(1, 0){58}}

		\put(147, 25){$a_1$}
		\put(162, 40){$a_2$}
		\put(183, 11){$a_n$}

		\put(170, 19){\circle*{3}}
		\put(170, 19){\circle{5}}

		\put(132, 57){\circle*{3}}
		\put(181, 74){\circle*{3}}
		\put(228, 19){\circle*{3}}

		\put(119, 72){\circle{38}}
		\put(185, 94){\circle{38}}
		\put(248, 20){\circle{38}}

		\put(112, 96){$r_1$}
		\put(180, 118){$r_2$}
		\put(243, 44){$r_n$}

		\put(105, 72){\vector(0, 1){4}}
		\put(120, 72){\oval(30, 30)[bl]}

		\put(167.5, 94){\vector(0, 1){4}}
		\put(180, 94){\oval(25, 25)[bl]}

		\put(242, 33.5){\vector(1, 0){4}}
		\put(243, 21){\oval(25, 25)[tl]}

		\put(152, 38){\vector(-1, 1){4}}
		\put(151, 37){\vector(-1, 1){4}}
		\put(150, 39){\vector(-1, 1){4}}

		\put(176.4, 50){\vector(0, 1){4}}
		\put(177.4, 50){\vector(0, 1){4}}
		\put(176.9, 51.5){\vector(0, 1){4}}

		\put(202, 18.9){\vector(1, 0){4}}
		\put(200.9, 19.3){\vector(1, 0){4}}
		\put(200.9, 18.5){\vector(1, 0){4}}

		\put(176, 23){$\iddots$}
		\put(190, 34){$\iddots$}
		\put(204, 45){$\iddots$}
		\put(218, 56){$\iddots$}
	\end{picture}
	
		\caption{A bouquet of lollipops}
		\label{fig:lollipop}
\end{figure}
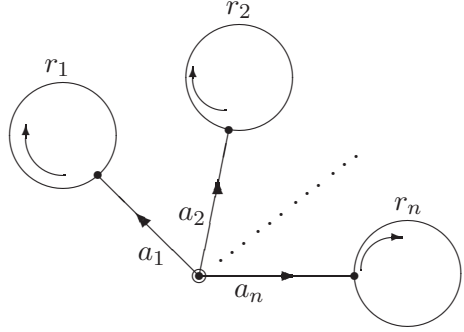

Each letter of $w$ corresponds to an edge of the van Kampen diagram and every time there is a cancellation in $w$, there is a folding of the two edges corresponding to the two canceled letters. This folding does not change other elements of the cell complex (see the proof of Prop. 2.3, pag. 166, in \cite{Short}). Sometimes in this process two or more faces making up a 2-sphere are discarded in order to keep the cell complex planar: this reduces the number of faces (see \cite{Short}, Fig. 2.6).

When all the cancellations have been made, the obtained word is $\rho(w)$, which is reduced, and is the word that can be read in the boundary of the van Kampen diagram. We say that $\rho(w)$ is the \textit{boundary label} of the diagram. That boundary forms a cycle (i.e., its initial and final points coincide) and starts with the edge labeled by the first letter of $w$, so the initial point of that edge is the initial point of the van Kampen diagram.

For example let $a_1 = x_1 x_2$, $r_1 = x_3$, $a_2 = x_1 x_4$,  $r_2 = x_5 x_6$, $a_3 = x_1 x_4$,  $r_3 = x_6^{-1} x_7$, that is 
		$$w := a_1 r_1 a_1^{-1} a_2 r_2 a_2^{-1} a_3 r_3 a_3^{-1} =$$ 
			$$(x_1 x_2) x_3 (x_2^{-1} x_1^{-1}) \, (x_1 x_4) x_5 x_6 (x_4^{-1} x_1^{-1}) \, (x_1 x_4) x_6^{-1} x_7 (x_4^{-1} x_1^{-1}).$$
Then $\rho(w) = x_1 x_2 x_3 x_2^{-1} x_4 x_5 x_7 x_4^{-1} x_1^{-1}$ and the van Kampen diagram corresponding to $\rho(w)$ obtained with the procedure described above is that of Figure \ref{fig:vKRed}.

\begin{figure}[h]
	\begin{picture}(80, 100) 
		\put(198, 9){\line(-1, 1){38}}
		\put(178, 29){\line(1, 1){48}}

		\put(147, 62){\circle{38}}
		\put(212, 63){\circle{38}}

		\put(179, 11){$x_1$}
		\put(159, 31){$x_2$}
		\put(113, 66){$x_3$}
		\put(192, 37){$x_4$}
		\put(179, 66){$x_5$}

		\put(198, 9){\circle*{3}}
		\put(198, 9){\circle{5}}

		\put(160, 47){\circle*{3}}
		\put(178, 29){\circle*{3}}
		\put(198, 49){\circle*{3}}
		\put(226, 77){\circle*{3}}

		\put(207, 72){$x_6$}
		\put(235, 66){$x_7$}

		\put(138, 84.5){\vector(1, 0){4}}
		\put(138, 72){\oval(25, 25)[tl]}

		\put(203, 84.5){\vector(1, 0){4}}
		\put(203, 72){\oval(25, 25)[tl]}

		\put(225, 43.5){\vector(-1, 0){4}}
		\put(224, 56){\oval(25, 25)[br]}

		\put(187.1, 18.9){\vector(-1, 1){4}}
		\put(188.2, 20){\vector(-1, 1){4}}
		\put(186.1, 21){\vector(-1, 1){4}}

		\put(168.9, 37){\vector(-1, 1){4}}
		\put(170, 38){\vector(-1, 1){4}}
		\put(167.9, 39){\vector(-1, 1){4}}

		\put(213.2, 64.2){\vector(-1, -1){4}}
		\put(215, 65){\vector(-1, -1){4}}
		\put(214, 66){\vector(-1, -1){4}}

		\put(186, 37){\vector(-1, -1){4}}
		\put(186.3, 37.6){\vector(-1, -1){4}}
		\put(186.7, 37.3){\vector(-1, -1){4}}
	\end{picture}
	
		\caption{The van Kampen diagram of a reduced relator}
		\label{fig:vKRed}
\end{figure}
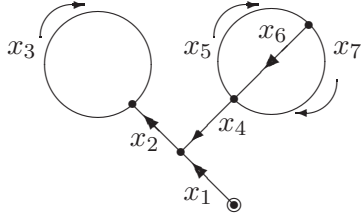

As it can be seen, the word $\rho(w)$ can be read from the boundary cycle of the diagram. The edges not belonging to faces, called \textit{spines}, are traversed twice, in both ways; by convention the label of an edge is the opposite when the edge is traversed in the direction opposite to the depicted arrow. The edges belonging to faces and not internal are traversed once. The internal edges of the faces are not part of the boundary.

The boundary labels of the faces are $r_1, r_2, r_3$, with the convention that the label of an internal edge is the opposite when traversed in the direction opposite to the depicted arrow.

\begin{remark} \label{remVK} \rm While the reduced form of $w$ is the same no matter the sequence of cancellations that are done, that is not true for the associated van Kampen diagram. That is two van Kampen diagrams obtained from the same bouquet of lollipops of Figure \ref{fig:lollipop} by means of two different sequences of cancellations of $w$ could be non-homeomorphic. Actually they could even not have the same number of faces.

This is not the case for the example of Figure \ref{fig:vKRed}, because the cancellations are independent, i.e., any sequence of cancellations from $w$ to $\rho(w)$ contains the same cancellations possibly in a different order. By what seen above, a folding does not change other elements of the complex, so in that case two diagrams corresponding to two different sequences of cancellations are homeomorphic. 

But let $r_1 = x y$, $r_2 = y^{-1} x^{-1}$, $r_3 = x z$ and $a_1 = a_2 = a_3 = 1$. Then $w = x y y^{-1} x^{-1} x z$. After canceling $y y^{-1}$ we obtain the word $w_1 = x x^{-1} x z$. Now if we cancel the first two letters we obtain a van Kampen diagram with a single face, since the faces corresponding to $r_1$ and $r_2$ make up a 2-sphere which is discarded. Instead if we cancel the second and third letter we obtain a van Kampen diagram with three faces. The two van Kampen diagrams have the same boundary label but are not homeomorphic even if they correspond to the reduced form of the same bouquet of lollipops.

In Figures \ref{fig:vKLeft} and \ref{fig:vKRight} of Remark \ref{nonHomeoVK} we show an example of two van Kampen diagrams that correspond to the same bouquet of lollipops, that have the same faces but that are not homeomorphic. 
\end{remark}

\begin{remark} \label{shifVK} \rm Let $\delta$ be a van Kampen diagram that has no spines, let $u$ be its boundary label and let $u_1, u_2$ be words such that $u = u_1 u_2$. Let us given a spine labeled by either $u_1^{-1}$ or $u_2$ and let us join the endpoint of this spine with the initial point of $\delta$. Then the new diagram obtained has label either $u_1^{-1} u_1 u_2 u_1$ or $u_2 u_1 u_2 u_2^{-1}$.

In the first case by folding all the edges corresponding to $u_1^{-1} u_1$ we obtain a diagram homeomorphic to $\delta$, thus without spines, and labeled by $u_2 u_1$, that is the same as $\delta$ but with initial point shifted by $|u_1|$ letters of the boundary cycle. In the second case by folding all the edges corresponding to $u_2 u_2^{-1}$ we still obtain a diagram that is the same as $\delta$ and with initial point shifted by $|u_1|$ letters of the boundary cycle.
\end{remark}

\section{Identities among relations} \label{sectIAR}

This section deals with identities among relations. Some of the material in this section can also be found in Appendix A of \cite{FirstArticle} and Appendix A of \cite{SecondArticle}, but we have included here in order to make the paper self-contained. 

\smallskip


Let $a_1, \cdots, a_m, r_1, \cdots, r_m, b_1, \cdots, b_n, s_1, \cdots, s_n$ be words such that the equality 
	$$\rho(a_1 r_1 a_1^{-1} \cdots a_m r_m a_m^{-1}) = \rho(b_1 s_1 b_1^{-1} \cdots b_n s_n b_n^{-1})$$ 
holds. Then we say that we have an \textit{identity among relations involving $r_1, \cdots, r_m, s_1^{-1}, \cdots, s_n^{-1}$} and we denote it 
	\begin{equation} \label{idamre2} a_1 r_1 a_1^{-1} \centerdot \cdots \centerdot a_m r_m a_m^{-1} \equiv b_1 s_1 b_1^{-1} \centerdot \cdots \centerdot b_n s_n b_n^{-1}.	\end{equation} 
If $n = 0$, that is the right hand side is 1, then we say that the identity is in \textit{normal form}.

Identities among relations are special types of word equations. They arise in the context of group presentations, but we will use them without involving an explicit group presentation. In particular an identity among relations involving $r_1, \cdots, r_m$ is an identity among relations for any group presentation having $r_1, \cdots, r_m$ as relators. The last claim is obvious if the $r_i$ are basic relators. If some of the $r_i$ are non-basic relators, then the claim follows from Remark \ref{iarRepl}.

\begin{remark} \label{iarRepl} \rm Let us suppose that (\ref{idamre2}) holds and that for some $i$ we have that the reduced form of $r_i$ is equal to the reduced form of $c_1 t_1 c_1^{-1} \cdots c_k t_k c_k^{-1}$ for some words $c_1, t_1, \cdots c_k, t_k$. Then by replacing in (\ref{idamre2}) the term $a_i r_i a_i^{-1}$ with $d_1 t_1 d_1^{-1} \centerdot \cdots \centerdot d_k t_k d_k^{-1}$, where $d_j = a_i c_j$, we obtain an identity among relations involving $s_1, \cdots, s_n, t_1, \cdots, t_k$ and all the $r_h$ except for $h = i$.
\end{remark}

\begin{definition} \label{}  \rm We say that the identities 
	$$a_2 r_2 a_2^{-1} \centerdot \cdots \centerdot a_m r_m a_m^{-1} \equiv a_1 r_1^{-1} a_1^{-1} \centerdot b_1 s_1 b_1^{-1} \centerdot \cdots \centerdot b_n s_n b_n^{-1},$$
	$$a_1 r_1 a_1^{-1} \centerdot \cdots \centerdot a_{m-1} r_{m-1} a_{m-1}^{-1} \equiv b_1 s_1 b_1^{-1} \centerdot \cdots \centerdot b_n s_n b_n^{-1} \centerdot a_m r_m a_m^{-1},$$
	$$b_1 s_1^{-1} b_1^{-1} \centerdot a_1 r_1 a_1^{-1} \centerdot \cdots \centerdot a_m r_m a_m^{-1} \equiv b_2 s_2 b_2^{-1} \centerdot \cdots \centerdot b_n s_n b_n^{-1}$$
and
	$$a_1 r_1 a_1^{-1} \centerdot \cdots \centerdot a_m r_m a_m^{-1} \centerdot b_n s_n b_n^{-1} \equiv b_1 s_1 b_1^{-1} \centerdot \cdots \centerdot b_{n-1} s_{n-1} b_{n-1}^{-1}$$
are \textit{1-step equivalent to (\ref{idamre2}}). 

We say that an identity $\iota$ is \textit{equivalent} to an identity $\iota'$ if there exist identities $\iota_1, \cdots, \iota_n$ such that $\iota$ is 1-step equivalent to $\iota_1$, $\iota_i$ is 1-step equivalent to $\iota_{i+1}$ for $i \in \{1, \cdots, n-1\}$ and $\iota_n$ is 1-step equivalent to $\iota'$.
\end{definition}

\begin{remark} \label{difForIds} \rm An identity among relations can be equivalent to more than one identity in normal form; these are called \textit{the normal forms} of that identity. We prove that two normal forms of the same identity are cyclic permutation one of the other.

The proof is by induction on the number of terms in the right hand side of (\ref{idamre2}), where the claim is obvious when that number is 1. Let that number be $k > 1$ and the claim be true when that number is less than $k$. With (\ref{idamre2}) we can associate the following and only these two identities with $k - 1$ terms on the right,
	\begin{equation} \label{cpId1} a_1 r_1 a_1^{-1} \centerdot \dots \centerdot a_m r_m a_m^{-1}  \centerdot b_k s_k b_k^{-1} \equiv b_1 s_1^{-1} b_1^{-1} \centerdot \dots \centerdot b_{k-1} s_{k-1} b_{k-1}^{-1}\end{equation}
and
	\begin{equation} \label{cpId2} b_1 s_1^{-1} b_1^{-1} \centerdot a_1 r_1 a_1^{-1} \centerdot \dots \centerdot a_m r_m a_m^{-1} \equiv b_2 s_2 b_2^{-1} \centerdot \dots \centerdot b_k s_k b_k^{-1}.\end{equation}
Therefore by induction hypothesis the normal forms of the identity (\ref{idamre2}) are two sets of identities such that the elements in each of these sets are cyclic permutations one of the other. These two sets are the cyclic permutations of the identities (\ref{cpId1}) and (\ref{cpId2}).

It remains to prove that any two elements taken one from the first set and the other from the second set are cyclic permutation one of the other. Since being a cyclic permutation is an equivalence relation, it is enough to prove that one specific element of the first set is a cyclic permutation of one specific element of the second set. This is done by taking the following two elements:
 	$$a_1 r_1 a_1^{-1} \centerdot \dots \centerdot a_m r_m a_m^{-1}  \centerdot b_k s_k^{-1} b_k^{-1} \centerdot \dots b_1 s_1^{-1} b_1^{-1} \equiv 1$$
from the first set and 
	$$b_k s_k^{-1} b_k^{-1} \centerdot \dots b_1 s_1^{-1} b_1^{-1} \centerdot a_1 r_1 a_1^{-1} \centerdot \dots \centerdot a_m r_m a_m^{-1} \equiv 1,$$
from the second set. It is trivial to see that these elements are cyclic permutations one of the other.
\end{remark}

\medskip

Let $\langle \, X \, | \, R \, \rangle$ be a presentation for a group $G$, with $X$ the set of generators and $R$ that of basic relators. We will assume without loss of generality that $R$ contains the inverse of any of its elements and the reduced form of the cyclic permutations of any of its elements. If $r_1, \cdots, r_n \in R$ are such that the identity in normal form
	\begin{equation} \label{idamre} a_1 r_1 a_1^{-1} \centerdot \cdots \centerdot a_n r_n a_n^{-1} \equiv 1	\end{equation}
holds, then (\ref{idamre}) determines a product of conjugates of basic relators equal to 1 not only in $G$ but also in $\mathcal{F}(X)$ (we recall that $G$ is a quotient of $\mathcal{F}(X)$).

\bigskip

In order to formalize these notions we introduce some definitions (we will follow \cite{BH}). Let us set $Y := \mathcal{F}(X) \times R$, let us define the inverse of an element $(a, r) \in Y$ as $(a, r^{-1})$ and let us denote $H$ the free monoid on $Y \cup Y^{-1}$. $H$ is the set of finite sequences of elements of $Y$. We denote an element of $H$ as $[(a_1, r_1), \cdots, (a_n, r_n)]$, where $a_i \in \mathcal{F}(X)$ and $r_i \in R$. The trivial element of $H$ is the sequence with zero elements.

Let $h := [(a_1, r_1), \dots, (a_n, r_n)] \in H$ and let $(a, r)$, $(b, s)$ be two consecutive elements $(a_i, r_i), (a_{i+1}, r_{i+1})$ of $h$ for some $i \in \{1, \dots, n - 1\}$, in particular $a = a_i$, $r = r_i$, $b = a_{i+1}$, $s = r_{i+1}$. We define the following transformations on $h$ that change it to another element of $H$:
\begin{itemize}

	\item [--] a \textit{Peiffer deletion} deletes in $h$ the elements $(a, r)$, $(b, s)$ if $a = b$ and $r^{-1} = s$;

	
	\item [--] an \textit{exchange} replaces in $h$ the pair of elements $(a, r)$, $(b, s)$  either with the pair
		$$(b, s), (\rho(b s^{-1} b^{-1} a), r)$$
	(we call it an \textit{exchange of type A at the $i$-th position} or \textit{exchange of type A-$i$}) or with the pair
		$$(\rho(a r a^{-1} b), s), (a, r)$$
	(we call it an \textit{exchange of type B at the $i$-th position} or \textit{exchange of type B-$i$})).	
\end{itemize}
Peiffer deletions and exchanges leave unchanged the $(a_j, r_j)$ for $j \neq i, i+1$.

Given two elements $h_1, h_2 \in H$, we say that \textit{$h_1$ Peiffer collapses to $h_2$} if $h_2$ can be obtained from $h_1$ by applying Peiffer deletions and exchanges.

There is a bijection $\chi$ between $H$ and the set of products of conjugates of elements of $R$ given by associating the element $h = [(a_1, r_1), \dots, (a_n, r_n)] \in H$ with the following product of conjugates of elements of $R$,
	$$a_1 r_1 a_1^{-1} \centerdot \dots \centerdot a_n r_n a_n^{-1}.$$
Also we define a monoid homomorphism $\psi$ from $H$ to $\mathcal{F}(X)$ by $\psi(h) := \rho(a_1 r_1 a_1^{-1} \dots a_n r_n a_n^{-1})$. If $\psi(h) = 1$, that is if $h$ belongs to the kernel of $\psi$, then we say that $h$ determines the identity among relations in normal form (\ref{idamre}). We say that this identity among relations \textit{Peiffer collapses to 1} if $h$ Peiffer collapses to the trivial element of $H$.

The restriction of $\chi$ to the kernel of $\psi$ determines a bijection with the set of identities among relations in normal form involving elements of $R$.

\begin{remark} \label{iarPeiffCons} \rm We have seen in the introduction to this section that if $r_1$, $\cdots$, $r_n$ are relators of a group presentation $\mathcal{P} := \langle \, X \, | \, S \, \rangle$ then an identity among relations involving $r_1, \cdots, r_n$ determines an identity among relations for $\mathcal{P}$, that is an identity involving the basic relators of $\mathcal{P}$.
	
By virtue of the Corollary at page 159 of \cite{BH} we have also that if the identity involving $r_1, \cdots, r_n$ Peiffer collapses to 1 then also the identity involving basic relators determined by it Peiffer collapses to 1.
\end{remark}

\begin{definition} \label{defBasIAR} \rm An identity among relations in normal form is said \textit{basic} if the corresponding element of $H$ collapses to 1 by means of only Peiffer deletions. Let $h := [(a_1, r_1), \dots, (a_n, r_n)]$ be that element; if moreover $a_1 = a_2 = \dots = a_n$ then that identity among relations is said \textit{strictly basic}.
\end{definition}

\noindent \textbf{Example.} The identity among relations $a r a^{-1} \centerdot b s b^{-1} \centerdot b s^{-1} b^{-1} \centerdot a r^{-1} a^{-1} \equiv 1$ is basic. If moreover $a = b$ it is strictly basic.
\medskip

\begin{remark} \label{equCondBas} \rm As before we denote $Y$ the set $\mathcal{F}(X) \times R$ and let us consider the free group on $Y$. Since $H$ is the free monoid on $Y \cup Y^{-1}$, then an identity among relations is basic if and only if the corresponding element of $H$ is 1 in the free group on $Y$.
\end{remark}

\begin{remark} \label{equCondBas2} \rm We now show that if one normal form of an identity is basic, then all the normal forms of that identity are basic.

Indeed by Remark \ref{equCondBas}, an identity among relations is basic if and only if the corresponding element of $\mathcal{M}(Y \cup Y^{-1})$ reduces to 1 in the free group on $Y$. By Remark \ref{difForIds}, two normal forms of the same identity are cyclic permutations one of the other, therefore if one of them reduces to 1 in the free group, all its cyclic permutations reduce to 1 too because a cyclic permutation is a special case of conjugation and in a group the only conjugate to 1 is 1 itself.
\end{remark}

\begin{definition} \label{defBasIAR2} \rm An identity among relations is said \textit{(strictly) basic} if one (and by Remark \ref{equCondBas2} all) of its normal forms is (strictly) basic.
\end{definition}

\begin{remark} \label{remDefBasIAR} \rm Let an identity among relations of the form
	$$a_1 r_1 a_1^{-1} \centerdot \cdots \centerdot a_m r_m a_m^{-1} \equiv b_1 r_1 b_1^{-1} \centerdot \cdots \centerdot b_m r_m b_m^{-1}$$
be basic; then $a_i = b_i$ for every $i = 1, \cdots, m$.
\end{remark}





As stated at the beginning of Section \ref{sectVKdiag}, a reader not interested in the interpretation of Theorem \ref{mainTheor2} in terms of diagrams can skip the parts of Remark \ref{pcrFrom} related to van Kampen diagrams.

\begin{remark} \label{pcrFrom} \rm Let $R$ be a set of reduced words, that is $R \subset \mathcal{F}(X)$. Let us consider the following operations on the elements of $\mathcal{F}(X)$: reduced product, cyclically reduced product, cyclically reduced form, conjugations, reduced form of cyclic permutations. 

Let $\mathcal{N}$ be the normal closure of $R$ in $\mathcal{F}(X)$; then $\mathcal{N}$ is the subset of $\mathcal{F}(X)$ generated by $R$ and by the above operations. Indeed cyclic permutations and the cyclically reduced form are special cases of conjugations and the cyclically reduced product is obtained by composing the cyclically reduced form with the reduced product.

Let $\sigma$ be a sequence of the above listed operations on the elements of $R$ and let $u \in \mathcal{N}$ be the result of $\sigma$. We will show how to associate with $\sigma$ an element $[(a_1, r_1), \cdots, (a_n, r_n)]$ of $H$ with the property that
	$$\rho(a_1 r_1 a_1^{-1} \cdots a_n r_n a_n^{-1}) = u.$$

Moreover if $R$ contains only cyclically reduced words and $\sigma$ is made up by only cyclically reduced products, cyclically reduced forms and the reduced form of cyclic permutations we will show how to associate with $\sigma$ a van Kampen diagram for the presentation $\langle \, X \, | \, R \, \rangle$ whose boundary label is $u$. Indeed we will show that given the bouquet of lollipops corresponding to $a_1 r_1 a_1^{-1} \cdots a_n r_n a_n^{-1}$, the procedure described in Appendix \ref{sectVKdiag} gives a unique van Kampen diagram whose boundary label is $u$. We will also show that this diagram is made only by faces (that is it does not have spines) and that the label of its boundary is cyclically reduced.

\smallskip

- Let us take a sequence of length one. This is an element $r$ of $R$ and we associate with it the element $[(1, r)]$ of $H$. We also associate with it the van Kampen diagram made by a single face labeled by $r$.

We can suppose by induction hypothesis that there is a natural number $k$ such that for each sequence $\sigma$ of length less than $k$ we have associated with $\sigma$ an element of $H$ and (when the additional hypothesis stated above are verified) also a van Kampen diagram, both with the properties specified above.

- Let us given sequences $\sigma, \sigma'$ of length less than $k$ with results respectively $u$ and $u'$. Then by induction hypothesis there exist $r_1, \cdots, r_m$, $s_1, \cdots, s_n \in R$ and $a_1, \cdots, a_m$, $b_1, \cdots, b_n \in \mathcal{F}(X)$ such that we have associated with $\sigma$ an element $[(a_1, r_1), \cdots, (a_m, r_m)] \in H$ such that 
	$$u = \rho(a_1 r_1 a_1^{-1} \cdots a_m r_m a_m^{-1})$$ 
and with $\sigma'$ an element $[(b_1, s_1), \cdots, (b_n, s_n)] \in H$ such that 
	$$u' =\rho(b_1 s_1 b_1^{-1} \cdots b_n s_n b_n^{-1}).$$
Let us consider the sequence $\tau$ having all the operations of $\sigma$ and $\sigma'$ plus the reduced product of $u$ by $u'$. Then we associate with $\tau$ the element
	$$[(a_1, r_1), \cdots, (a_m, r_m), (b_1, s_1), \cdots, (b_n, s_n)] \in H;$$
obviously $\rho(a_1 r_1 a_1^{-1} \cdots a_m r_m a_m^{-1} b_1 s_1 b_1^{-1} \cdots b_n s_n b_n^{-1}) = \rho(u u')$.

Let us suppose that we have associated with $\sigma$ and $\sigma'$ two van Kampen diagrams with boundary labels $u$ and $u'$ and obtained from the bouquets of lollipops corresponding to $a_1 r_1 a_1^{-1} \cdots a_m r_m a_m^{-1}$ and $b_1 s_1 b_1^{-1} \cdots b_n s_n b_n^{-1}$ respectively. Let us consider the van Kampen diagram obtained by making coincide the initial points of these two diagrams: this diagram has boundary label $u u'$. For each cancellation in $u u'$ we fold the edges corresponding to the canceled letters as explained in Appendix \ref{sectVKdiag}. After all cancellations have been made, the diagram obtained has boundary label equal to $\rho(u u')$ and we associate it with $\tau$. Since $u$ and $u'$ are reduced, there is only one sequence of cancellations from $u u'$ to $\rho(u u')$ and thus this procedure gives a unique van Kampen diagram from each pair of van Kampen diagram associated with $\sigma$ and $\sigma'$.

- Now let us consider a sequence $\sigma_1$ having all the operations of $\sigma$ plus the conjugation of $u$ by a word $b$. Then we associate with $\sigma_1$ the element $[(c_1, r_1), \cdots, (c_m, r_m)] \in H$ where $c_i = \rho(b a_i)$. Obviously 
	$$\rho(c_1 r_1 c_1^{-1} \cdots c_m r_m c_m^{-1}) = \rho(b u b^{-1}).$$
Let us suppose that a van Kampen diagram $\delta$ with the properties above has been associated with $\sigma$. Then we consider a spine labeled by $b$ and we join the endpoint of that spine with the initial point of $\delta$. We associate with $\sigma_1$ the van Kampen diagram obtained after making all the foldings corresponding to cancellations of the word $b u b^{-1}$ as explained in Appendix \ref{sectVKdiag}. Since there can be different sequences of cancellations that give $\rho(b u b^{-1})$ from $b u b^{-1}$, this van Kampen diagram is not unique for a given $\delta$. However it is unique if the cancellations in $b u b^{-1}$ are independent, i.e., the different sequences of cancellations from $b u b^{-1}$ to $\rho(b u b^{-1})$ contain the same cancellations possibly in different order (see Remark \ref{remVK}). This happens for example when no letter canceled in $b u$ is canceled also in $u b^{-1}$, in particular if $u$ is cyclically reduced.

- Now let us consider a sequence $\sigma_2$ having all the operations of $\sigma$ plus the cyclically reduced form of $u$. The previous cases show how to associate with $\sigma_2$ an element of $H$ and a van Kampen diagram with the above properties because by virtue of (\ref{scope}) of Proposition \ref{summar} the cyclically reduced form is a special case of conjugation. We observe that a unique van Kampen diagram is obtained from any diagram associated with $\sigma$, because if $u = a \hat{\rho}(u) a^{-1}$ for some word $a$, then $\hat{\rho}(u)$ is obtained from $u$ by conjugating it with $a^{-1}$ and no letter canceled in $a^{-1} u$ is canceled in $u a$.

-Now let us consider a sequence $\sigma_3$ having all the operations of $\sigma$ plus the reduced form of a cyclic permutation of $u$. This means that there exist words $u_1, u_2$ such that $u = u_1 u_2$ and that the last operation of $\sigma_3$ is the conjugation of $u$ by either $u_2$ or by $u_1^{-1}$. This implies that we associate with $\sigma_3$ the element $[(c_1, r_1), \cdots, (c_m, r_m)] \in H$ where $c_i$ for $i = 1, \cdots, m$ can be either equal to $\rho(u_2 a_i)$ or to $\rho(u_1^{-1} a_i)$.

Now let $R$ contain only cyclically reduced words and let $\sigma$ be made up by only cyclically reduced products, cyclically reduced forms and the reduced form of cyclic permutations. By induction hypothesis we can suppose that there is a unique van Kampen diagram $\delta$ associated with $\sigma$, that this diagram is made only by faces and that the label of this diagram, which is $u_1 u_2$, is cyclically reduced. 

Let us given a spine labeled by either $u_1^{-1}$ or $u_2$ and let us join the endpoint of this spine with the initial point of $\delta$. Then the new diagram obtained has label either $u_1^{-1} u_1 u_2 u_1$ or $u_2 u_1 u_2 u_2^{-1}$. Since $u_1 u_2$ is cyclically reduced then the only cancellations done are $u_1^{-1} u_1$ in the first case and $u_2 u_2^{-1}$ in the second case. By Remark \ref{shifVK} in both cases the resulting diagram will be homeomorphic to $\delta$ but with initial point shifted by $|u_1|$ letters of the boundary cycle. Therefore the diagram associated with $\sigma_3$ is unique, is still made only by faces and the label of this diagram is cyclically reduced.

- Finally if $\tau$ is the sequence having all the operations of $\sigma$ and $\sigma'$ plus the cyclically reduced product of $u$ by $u'$, then the previous cases show how to associate with $\tau$ an element of $H$ with the properties stated above because the cyclically reduced product is the composition of the reduced product with the cyclically reduced form.

Now let $R$ contain only cyclically reduced words and let $\sigma$ be made up by only cyclically reduced products, cyclically reduced forms and the reduced form of cyclic permutations. By induction hypothesis we can suppose that there are unique van Kampen diagrams $\delta$ and $\delta'$ associated with $\sigma$ and $\sigma'$, that these diagrams are made only by faces and that the labels of these diagrams are cyclically reduced.

By what seen above there is a unique diagram associated with the reduced product of $u$ by $u'$ and built from $\delta$ and $\delta'$; this diagram contains only faces because $\delta$ and $\delta'$ contain only faces. We know that $\rho(u u') = b \hat{\rho}(u u') b^{-1}$ for some word $b$. By applying the cyclically reduced form to $\rho(u u')$ we obtain a unique van Kampen diagram: this diagram is made only by faces and the label of this diagram is $\hat{\rho}(u u')$, which is cyclically reduced. 
\end{remark}

\begin{remark} \label{seqFromProd} \rm We show how to associate with a product of conjugates of elements of $R$ a sequence of operations on $R$ as described in Remark \ref{pcrFrom}.

Indeed with $a_1 r_1 a_1^{-1} \centerdot \cdots \centerdot a_m r_m a_m^{-1}$ we associate the following sequence: conjugation of $r_1$ with $a_1$; conjugation of $r_2$ with $a_2$; $\cdots$; conjugation of $r_m$ with $a_m$; reduced product of $a_1 r_1 a_1^{-1}$ by $a_2 r_2 a_2^{-1}$; reduced product of $a_1 r_1 a_1^{-1} a_2 r_2 a_2^{-1}$ by $a_3 r_3 a_3^{-1}$; $\cdots$; reduced product of $a_1 r_1 a_1^{-1} \cdots  a_{m-1} r_{m-1} a_{m-1}^{-1}$ by $a_m r_m a_m^{-1}$.

In particular, given words $u$ and $v$, we associate with $u*v$ the product $\alpha u \alpha^{-1} \centerdot \alpha v \alpha^{-1}$, where $\alpha$ is such that $u*v = \rho(\alpha u v \alpha^{-1})$ (see (\ref{scope}) of Proposition \ref{summar}). \end{remark}

\begin{remark} \label{idFrom} \rm Let $u, u' \in \mathcal{N}$ be obtained respectively from sequences $\sigma$ and $\sigma'$ of operations on $R$ as described in Remark \ref{pcrFrom}, in particular in view of Remark \ref{seqFromProd} let $u, u'$ be the reduced forms of products of conjugates of elements of $R$. Let us suppose that $u \sim u'$; we show how to associate with $\sigma$, $\sigma'$ and the equivalence $u \sim u'$ an identity among relations involving elements of $R$.	
	
Indeed the procedure described in Remark \ref{pcrFrom} associates with $\sigma$ and $\sigma'$ elements $h := [(a_1, r_1), \cdots, (a_m, r_m)]$ and $h' := [(b_1, s_1), \cdots, (b_n, s_n)]$ of $H$ such that $\rho(a_1 r_1 a_1^{-1} \cdots a_m r_m a_m^{-1}) = u$ and $\rho(b_1 s_1 b_1^{-1} \cdots b_n s_n b_n^{-1}) = u'$.

If $u \sim v$ then $u$ and $v$ are conjugates and thus there exists a word $c$ such that $u = \rho(c v c^{-1})$. We associate with $\sigma$, $\sigma'$ and the equivalence $u \sim v$ the following identity among relations 
	$$a_1 r_1 a_1^{-1} \centerdot \cdots \centerdot a_m r_m a_m^{-1} \equiv d_1 s_1 d_1^{-1} \centerdot \cdots \centerdot d_n s_n d_n^{-1}$$
where $d_i = c b_i$.	
\end{remark}

\begin{remark} \label{idFrom2} \rm Let $R \subset \mathcal{F}(X)$, let $\sigma$ and $\sigma'$ be sequences of operations on $R$ as described in Remark \ref{pcrFrom}, let $w$ and $w'$ be the results of $\sigma$ and $\sigma'$ respectively and let $w \sim w'$. 
	
Let us suppose that $w \sim w'$ and that the identity among relations that by Remark \ref{idFrom} follows from this equivalence is 
	\begin{equation} \label{idFrom2Eq1} a_1 r_1 a_1^{-1} \centerdot \cdots \centerdot a_m r_m a_m^{-1} \equiv a_{m+1} r_{m+1} a_{m+1}^{-1} \centerdot \cdots \centerdot a_n r_n a_n^{-1}.
	\end{equation}

Let $\sigma''$ be a sequence of operations on $R$ that has all the operations of $\sigma'$ plus a cyclic permutation and let $w''$ be the result of $\sigma''$. This implies that $w' \sim w''$ and by transitivity that $w \sim w''$. We prove that the identity among relations that by Remark \ref{idFrom} follows from $\sigma''$ is
	\begin{equation} \label{idFrom2Eq2} b_1 r_1 b_1^{-1} \centerdot \dots \centerdot b_m r_m b_m^{-1} \equiv b_{m+1} r_{m+1} b_{m+1}^{-1} \centerdot \dots \centerdot b_n r_n b_n^{-1},\end{equation}	
where $b_1, \dots, b_n$ are words such that $b_k = b a^{-1} a_k$ for some words $a, b$ and for $k = 1, \dots, n$.
	
Indeed by the proof of Remark \ref{idFrom}	we have that there exists a word $a$ such that the products of conjugates of elements of $R$ associated with $\sigma$ and $\sigma'$ are $a_1 r_1 a_1^{-1} \centerdot \dots \centerdot a_m r_m a_m^{-1}$ and $c_{m+1} r_{m+1} c_{m+1}^{-1} \centerdot \dots \centerdot c_n r_n c_n^{-1}$ respectively, where $c_j = \rho(a^{-1} a_j)$ for $j = m+1, \dots, n$.
	
From Remark \ref{pcrFrom} we have that there exists a word $b$ such that the product of conjugates of elements of $R$ associated with $\sigma''$ is $b_{m+1} r_{m+1} b_{m+1}^{-1} \centerdot \dots \centerdot b_n r_n b_n^{-1}$, where $b_j = \rho(b c_j) = \rho(b a^{-1} a_j)$ for $j = m+1, \dots, n$.

This means that $w' = \rho(a^{-1} w a)$, $w'' = \rho(b w' b^{-1})$, therefore  $w'' = \rho(b a^{-1} w a b^{-1})$ and by setting $b_i : = b a^{-1} a_i$ for $i = 1, \dots, m$ we have that the identity among relations that by Remark \ref{idFrom} follows from the equivalence $w \sim w''$ is
	$$b_1 r_1 b_1^{-1} \centerdot \dots \centerdot b_m r_m b_m^{-1} \equiv b_{m+1} r_{m+1} b_{m+1}^{-1} \centerdot \dots \centerdot b_n r_n b_n^{-1},$$
where $b_k = b a^{-1} a_k$ for $k = 1, \dots, n$.
\end{remark}

\begin{remark} \label{basicToo} \rm If (\ref{idFrom2Eq1}) is basic then (\ref{idFrom2Eq2}) is basic too. Indeed by Remark \ref{equCondBas2} if (\ref{idFrom2Eq1}) is basic then any normal form of (\ref{idFrom2Eq1}) corresponds to 1 in the free group on $Y$. We have that also any normal form of (\ref{idFrom2Eq2}) corresponds to 1 in the free group on $Y$, because $b_k = b a^{-1} a_k$ for every $k$, so if for some $h, k$ we have that $a_h = a_k$ then $b_h = b_k$.
\end{remark}

\begin{remark} \label{idRev} \rm Let $R \subset \mathcal{F}(X)$ and let $\sigma$ be a sequence of operations on $R$ as described in Remark \ref{pcrFrom}. We define the \textit{reverse} of $\sigma$, denoted $\underline{\sigma}$, by taking the reverse of each operation of $\sigma$. With $\underline{\sigma}$ we can associate a product of conjugates of reverses of elements of $R$ that is the reverse of that associated with $\sigma$. In particular if 
	$$a_1 r_1 a_1^{-1} \centerdot \cdots \centerdot a_m r_m a_m^{-1}$$ 
is the product of conjugates of elements of $R$ associated with $\sigma$ then the one associated with $\underline{\sigma}$ is 
	$$\underline{a_m} \, \underline{r_m} \,\underline{a_m^{-1}} \centerdot \cdots \centerdot \underline{a_1} \, \underline{r_1} \, \underline{a_1^{-1}}.$$
In particular the result of $\underline{\sigma}$ is the reverse of the result of $\sigma$.
	
Now let $\sigma$ and $\sigma'$ be sequences of operations on $R$, let $w$ and $w'$ be the results of $\sigma$ and $\sigma'$ respectively and let $w \sim w'$. We have that $\underline{w}$ and $\underline{w'}$ are the results of $\underline{\sigma}$ and $\underline{\sigma'}$, that $\underline{w} \sim \underline{w'}$ and that the identity among relations that follows from this equivalence is the reverse of that follows from $w \sim w'$ as shown in Remark \ref{idFrom2}. In particular the first identity is basic in $r_1,  \cdots, r_m$ if and only if the second is basic in $\underline{r_1},  \cdots, \underline{r_m}$.	
\end{remark}

\textit{Address:}

Carmelo Vaccaro

Laboratoire de Mathématiques d'Orsay

Université Paris-Saclay 

Bâtiment 307, rue Michel Magat

91400 Orsay

\medskip

\textit{e-mail:} \textsf{carmelo.vaccaro@universite-paris-saclay.fr}

 \end{document}